\documentclass[10pt]{amsart}

\usepackage{amsfonts,amsmath,amssymb,amsxtra,amsthm}

\usepackage{mathrsfs,mathalfa}
\usepackage{mathtools}
\usepackage[margin=25mm]{geometry}
\usepackage{graphicx}
\usepackage[dvipsnames]{xcolor}
\usepackage{tikz}
\usetikzlibrary{decorations.pathreplacing,calligraphy}
\usetikzlibrary{positioning}
\usepackage{tikz-cd}
\tikzcdset{scale cd/.style={every label/.append style={scale=#1},
    cells={nodes={scale=#1}}}}
\usepackage{hyperref}
\usepackage{csquotes,xpatch}
\usepackage[shortlabels]{enumitem}
\usepackage{quiver}
\usepackage[normalem]{ulem}
\usepackage{marginnote}
\usepackage{tikzit}
\usepackage[noabbrev,capitalise]{cleveref}

\usepackage{braket}

\usepackage{geometry}
\usepackage{stmaryrd}

\input xy
\xyoption{all}
\usepackage[all,knot]{xy}

\colorlet{mylinkcolor}{Blue}
\colorlet{mycitecolor}{Maroon}
\colorlet{myurlcolor}{PineGreen}

\hypersetup{
	linkcolor  = mylinkcolor!,
	citecolor  = mycitecolor!,
	urlcolor   = myurlcolor!,
	colorlinks = true,
}

\newcommand{\Ab}{\operatorname{Ab}}

\newcommand{\Ann}{\operatorname{Ann}\nolimits}
\newcommand{\Zg}{\operatorname{Zg}\nolimits}
\newcommand{\Sp}{\operatorname{Sp}\nolimits}
\newcommand{\Cone}{\operatorname{Cone}\nolimits}
\newcommand{\op}{\operatorname{op}\nolimits}

\newcommand{\fp}{\operatorname{fp}\nolimits}
\newcommand{\Flat}{\operatorname{Flat}\nolimits}

\newcommand{\End}{\operatorname{End}\nolimits}
\newcommand{\Hom}{\operatorname{Hom}\nolimits}

\newcommand{\Imm}{\operatorname{Im}\nolimits}

\newcommand{\cper}{\operatorname{per }\nolimits}

\newcommand{\Ker}{\operatorname{Ker}\nolimits}

\newcommand{\Mor}{\operatorname{Mor}\nolimits}
\newcommand{\Ext}{\operatorname{Ext}\nolimits}

\newcommand{\Mod}{\operatorname{Mod}\nolimits}
\renewcommand{\mod}{\operatorname{mod}\nolimits}

\newcommand{\Z}{\operatorname{\mathbb{Z}}\nolimits}

\newcommand{\bbD}{\mathop{}\!\mathbb{D}}

\newcommand{\N}{\operatorname{\mathbb{N}}\nolimits}
\newcommand{\R}{\operatorname{\mathbb{R}}\nolimits}

\newcommand{\Obj}{\operatorname{Ob}\nolimits}
\newcommand{\length}{\operatorname{length}\nolimits}

\newcommand{\id}{\mathbf{1}}
\newcommand{\arr}[1]{\stackrel{sharp1}{\rightarrow}}
\newcommand{\larr}[1]{\stackrel{sharp1}{\leftarrow}}

\newcommand{\ca}{{\mathcal A}}
\newcommand{\cA}{{\mathcal A}}
\newcommand{\cb}{{\mathcal B}}
\newcommand{\cc}{{\mathcal C}}
\newcommand{\cC}{{\mathcal C}}
\newcommand{\cd}{{\mathcal D}}
\newcommand{\cE}{{\mathcal E}}
\newcommand{\ce}{{\mathcal E}}
\newcommand{\cF}{{\mathcal F}}

\newcommand{\CK}{{\mathcal K}}

\newcommand{\co}{{\mathcal O}}

\newcommand{\cS}{{\mathcal S}}
\newcommand{\ct}{{\mathcal T}}

\newcommand{\cT}{{\mathcal T}}

\newcommand{\cx}{{\mathcal X}}

\newcommand{\D}{{\operatorname{D}}}

\renewcommand{\tilde}[1]{\widetilde{sharp1}}

\newcommand{\ul}[1]{\underline{sharp1}}
\newcommand{\ol}[1]{\overline{sharp1}}
\renewcommand{\hat}[1]{\widehat{sharp1}}

\newcommand{\emphbf}[1]{\emph{\textbf{#1}}}

\newtheorem{theorem}{Theorem}[section]
\newtheorem{corollary}[theorem]{Corollary}

\newtheorem{lemma}[theorem]{Lemma}
\newtheorem{proposition}[theorem]{Proposition}

\newtheorem{thmx}{Theorem}

\theoremstyle{definition}
\newtheorem{definition}[theorem]{Definition}
\newtheorem{example}[theorem]{Example}
\newtheorem{remark}[theorem]{Remark}
\newtheorem{assumption}[theorem]{Setup}

\numberwithin{equation}{section}

\begin{document}

\title{A functorial approach to rank functions on triangulated categories}

\author{Teresa Conde}
\address{Teresa Conde, Fakult\"at f\"ur Mathematik, Universit\"at Bielefeld, Universitätsstra{\ss}e 25, 33615 Bielefeld, Germany}
\email{tconde@math.uni-bielefeld.de}
\author{Mikhail Gorsky}
\address{Mikhail Gorsky, Fakult\"at f\"ur Mathematik, Universit\"at Wien, Oskar-Morgenstern-Platz 1, 1090 Wien, Austria}
\email{mikhail.gorskii@univie.ac.at}
\author{Frederik Marks}
\address{Frederik Marks, Institut f\"ur Algebra und Zahlentheorie, Universit\"at Stuttgart, Pfaffenwaldring 57, 70569 Stuttgart, Germany}
\email{marks@mathematik.uni-stuttgart.de}
\author{Alexandra Zvonareva}
\address{Alexandra Zvonareva, Institut f\"ur Algebra und Zahlentheorie, Universit\"at Stuttgart, Pfaffenwaldring 57, 70569 Stuttgart, Germany and Institute of Mathematics, Czech Academy of Sciences,  \v{Z}itn\'a 25,   115 67   Praha 1,   Czech Republic}
\email{alexandra.zvonareva@mathematik.uni-stuttgart.de, zvonareva@math.cas.cz}

\keywords{rank function, additive function, triangulated category, functor category, localisation}
\subjclass[2020]{16B50,16E45,18A25,18E10,18E35,18G80}

\maketitle
\begin{center}
    \emph{Dedicated to Henning Krause on the occasion of his 60th birthday}
\end{center}

\begin{abstract}
We study rank functions on a triangulated category $\cC$ via its abelianisation $\mod\cC$. We prove that every rank function on $\cC$ can be interpreted as an additive function on $\mod\cC$. As a consequence, every integral rank function has a unique decomposition into irreducible ones. Furthermore, we relate integral rank functions to a number of important concepts in the functor category $\Mod\cC$.  
We study the connection between rank functions and functors from $\cC$ to locally finite triangulated categories, generalising results by Chuang and Lazarev. 
In the special case $\cC=\cT^c$ for a compactly generated triangulated category $\cT$, 
this connection becomes particularly nice, providing a link between rank functions on $\cC$ and smashing localisations of $\cT$. In this context,
any integral rank function can be described using the composition length with respect to certain endofinite objects in $\cT$.
Finally, if $\cC=\cper (A)$ for a differential graded algebra $A$, we classify homological epimorphisms $A\to B$ with $\cper (B)$ locally finite via special rank functions which we call idempotent.
\end{abstract}

\section{Introduction}
\label{sec:introduction}

Rank functions on triangulated categories were recently introduced  by Chuang and Lazarev in \cite{Chuang2021}. Roughly speaking, for a triangulated category $\cC$, a rank function on $\cC$ is an  assignment of a non-negative real number (or an element of another partially ordered group) to every object of $\cC$, satisfying certain axioms. Rank functions can be defined alternatively on morphisms of $\cC$.
The motivation of Chuang and Lazarev to introduce and study rank functions stemmed from the work by Cohn and Schofield on Sylvester rank functions (\cite{Cohn2008,Schofield1985}), which are defined on the category of finitely presented modules over a ring $A$ and can be used to classify ring morphisms from $A$ into simple artinian rings or skew-fields. Each Sylvester rank function naturally extends to a rank function on the perfect derived category $\cper (A)$, but not all rank functions on $\cper (A)$ arise in this way.
Another standard example of a rank function is the dimension of the total cohomology of an object in the bounded derived category of a finite-dimensional algebra. Mass functions associated with  Bridgeland stability conditions provide further interesting examples (\cite{Chuang2021,Ikeda2021h}). 

Developing the analogy between rank functions on triangulated categories and Sylvester rank functions, Chuang and Lazarev established a connection between rank functions and functors to simple triangulated categories, that is, triangulated categories with an indecomposable generator such that any triangle is a direct sum of split triangles. In particular, in \cite[Theorem 6.4]{Chuang2021} they obtained a bijection between thick subcategories $\CK$ of $\cC$ with simple Verdier quotient $\cC/\CK$ and localising prime rank functions on $\cC$, that is, normalised rank functions with integral values, subject to rather strong restrictions. In case $\cC$ is the perfect derived category $\cper (A)$ of a dg algebra $A$, these two classes are also in bijection with equivalence classes of finite homological epimorphisms from $A$ to simple artinian dg algebras (\cite[Theorem 6.5]{Chuang2021}).

The point of view we are suggesting in this paper is that  of  functor categories. Functorial methods have been successfully applied in representation theory in the past. Considering the category of 
contravariant
additive functors $\Mod(\mod A)$ from the category of finitely presented modules $\mod A$ over an Artin algebra $A$ to the category of abelian groups led to classic contributions such as Auslander--Reiten theory and Auslander correspondence, and to modern trends such as higher Auslander correspondence and higher-dimensional homological algebra. In the context of triangulated categories, the functorial approach led to important developments such as purity and Ziegler spectra for triangulated categories. 

We start by reinterpreting a rank function $\rho$ on $\cC$ as a function $\widetilde{\rho}$ on the category $\mod\cC$ of finitely presented additive functors from 
$\cC^{\op}$
to abelian groups. The function $\widetilde{\rho}$ turns out to be additive on short exact sequences and invariant under the action of translation induced from  $\cC$. This reinterpretation is our first main result, which allows us to use the deeply developed theory of additive functions on abelian categories.

\begin{thmx}[Theorem \ref{thm:TransInvariantAddFunctions<->rankFunctions}]
\label{theorem:translationResult}
	Let $\cC$ be a skeletally small triangulated category. There is a bijective correspondence between:
	\begin{enumerate}[(a)]
		\item translation-invariant additive functions $\widetilde{\rho}$ on $\mod \cC$;
		\item rank functions $\rho$ on $\cC$.
	\end{enumerate}
	The correspondence maps an additive function $\widetilde{\rho}$ to the rank function $\rho$ given by $\rho(f)\coloneqq\widetilde{\rho}(\Imm \Hom_\cC(-,f))$. The inverse maps $\rho$ to the additive function $\widetilde{\rho}$ defined by $\widetilde{\rho}(F)\coloneqq\rho(f)$ for $F\cong \Imm \Hom_\cC(-,f)$.
\end{thmx}

\noindent This theorem was inspired by results  on characters on additive categories due to Crawley-Boevey (\cite{Crawley-Boevey1994a,Crawley-Boevey1994}) and on cohomological length functions due to Krause  (\cite{Krause2016}). 
Now we can reap the fruits of passing to the functor category. By adapting results from \cite{Crawley-Boevey1994a,Crawley-Boevey1994}, we obtain a decomposition theorem for integral rank functions.

\begin{thmx}[Theorem \ref{thm:decomposition}]
	Every integral rank function on a skeletally small triangulated category $\cC$ can be uniquely  decomposed as a locally finite sum
	of irreducible rank functions.
\end{thmx}

\noindent Here, by irreducible rank function we mean a non-zero integral rank function that cannot be further decomposed into a sum of two non-zero integral rank functions. Every prime rank function, for example, can be shown to be irreducible. 
Following this strategy and using results of Crawley-Boevey, Herzog and Krause (\cite{Crawley-Boevey1994a,Crawley-Boevey1994,Herzog1997,Krause1997}), rank functions can be further related to many notions appearing in the context of functor categories, e.g.~to cohomological endofinite functors in $\Mod\cC$, Serre subcategories of $\mod\cC$, hereditary torsion subcategories of finite type of $\Mod\cC$, and closed subsets of the Ziegler spectrum of $\Mod\cC$ (see Figure \ref{fig:intro} for details).

The functorial approach allows us to extend the bijection between localising prime rank functions on $\cc$ and thick subcategories of $\cc$ with simple Verdier quotient, established in \cite[Theorem 6.4]{Chuang2021}. For this purpose, we introduce the class of exact rank functions, which contains all localising rank functions, and we use the notion of a CE-quotient functor, introduced by Krause in \cite{Krause2005}. Instead of functors to simple triangulated categories, we are now considering functors to locally finite triangulated categories, that is, triangulated categories $\cE$ such that $\mod \cE$ is an abelian length category. Instead of prime rank functions, we use basic ones, by which we mean integral rank functions with no repetitions in the irreducible summands appearing in their decomposition. This can be seen as another normalisation condition. 

\begin{thmx}[Theorem \ref{thm: FDT_generalised}]
Let $\cC$ be a skeletally small triangulated category. There is a bijective correspondence between:
	\begin{enumerate}[(a)]
		\item exact basic rank functions on $\cc$;
		\item equivalence classes of CE-quotient functors from $\cc$ to locally finite triangulated categories.
	\end{enumerate}
	Moreover, the underlying rank function is localising if and only if the CE-quotient is equivalent to a Verdier quotient.
\end{thmx}

The CE-quotient functors have a more natural interpretation in the particularly nice setting when $\cC$ is the subcategory $\cT^c$ of compact objects of a compactly generated triangulated category $\cT$. In this context, rank functions can be related to endofinite objects and definable subcategories of $\cT$ (see Figure \ref{fig:intro} for details). The connection with endofinite objects allows us to explicitly describe basic rank functions on $\cT^c$. Together with the decomposition theorem, this provides a complete description of integral rank functions in this setting. 

\begin{thmx}[Theorem \ref{thm:FundamentalCorrespondenceCompGen}]
Let $\cT$ be a compactly generated triangulated category. Every basic rank function on $\cT^c$ is of the form $\length_{\End_{\cT}(Z)} \Hom_\cT(-,Z)|_{\cT^c}$ for some basic endofinite object $Z$ in $\cT$ such that $\Sigma Z\cong Z$.
\end{thmx}

The 
correspondence between rank functions on $\cT^c$ and definable subcategories of $\cT$
suggests the following change of perspective. Instead of describing rank functions via quotients of $\cT^c$, we could describe them using certain subcategories of $\cT$. 
In fact, we show in Corollary \ref{cor:idemp-cdrho_triang} that a basic rank function on  $\cT^c$ is exact if and only if the corresponding definable subcategory is triangulated.
This triangulated category is then compactly generated and by passing to compact objects one can recover the associated CE-quotient. Here, we use that a CE-quotient functor is, roughly speaking, a restriction of a smashing localisation to the level of compacts. Kernels of such smashing localisations are usually referred to as smashing subcategories, and are essential when studying decompositions of $\cT$ into smaller triangulated categories (see, for example, \cite{Nicolas}). Additionally, 
exact rank functions on $\cT^c$ coincide with idempotent rank functions, that is, rank functions $\rho$ such that the ideal of morphisms satisfying $\rho(f)=0$ is idempotent. As a result, we obtain the following theorem:

\begin{thmx}[Theorem \ref{thm_smashing}]
	Let $\cT$ be a compactly generated triangulated category. There is a bijection between:
	\begin{enumerate}[(a)]
	   	\item idempotent basic rank functions on $\cT^c$;
		\item smashing subcategories $\cS$ of $\cT$ with $(\cT/\cS)^c$ locally finite.
	\end{enumerate}
\end{thmx}

\noindent 
This correspondence restricts to a bijection between localising basic rank functions and compactly generated smashing subcategories as in (b), allowing us to test a restricted version of the telescope conjecture for compactly generated triangulated categories using rank functions (see Corollary \ref{cor_smashing}).
In case $\cC$ is the perfect derived category of a dg algebra $A$, the connection between rank functions on $\cper (A)$ and smashing subcategories of the derived category $\D(A)$ provides the following generalisation of \cite[Theorem 6.5]{Chuang2021}:

\begin{thmx}[Theorem \ref{cor_dg}]
\label{thm:introF}
Let $A$ be a dg algebra. There is a bijection between:
	\begin{enumerate}[(a)]
	\item idempotent basic rank functions on $\cper (A)$;
		\item equivalence classes of homological epimorphisms $A\to B$ with $\cper (B)$ locally finite.
	\end{enumerate}
\end{thmx}

\noindent Recall that a morphism of dg algebras $A\to B$ is a
homological epimorphism if the induced restriction functor $\D(B)\to \D(A)$ is fully faithful. The bijection in Theorem \ref{thm:introF} is explicit and gives a complete description of the equivalence classes of the corresponding homological epimorphisms.

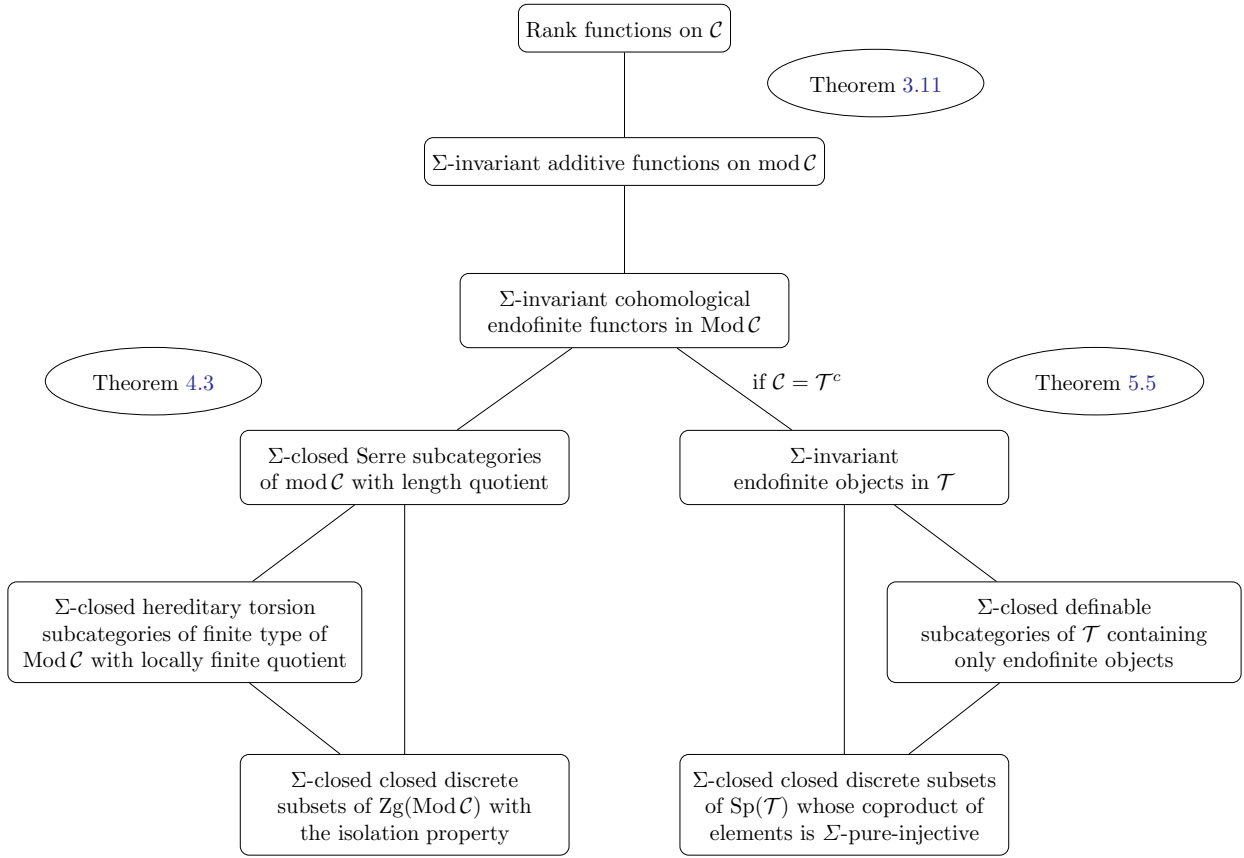
\begin{figure}[h]\label{fig:intro}
\begin{center}
    % Node styles
\tikzstyle{rectangle_1}=[fill=white, draw=black, shape=rectangle, text height=0.4cm, rounded corners]
\tikzstyle{fixed_width_rectangle_2}=[fill=white, draw=black, shape=rectangle, text width=5cm, text height=0.4cm, align=center, rounded corners]
\tikzstyle{fixed_width_rectangle_3}=[fill=white, draw=black, shape=rectangle, align=center, text height=0.4cm, text width=5.4cm, rounded corners]
\tikzstyle{ellipse}=[fill=white, draw=black, shape=ellipse, text height=0.4cm, text width=2.2cm, align=center]

% Edge styles
\tikzstyle{arrow}=[-]
    \scalebox{0.835}{\begin{tikzpicture}
	\begin{pgfonlayer}{nodelayer}
		\node [style={rectangle_1}] (2) at (0, 14.75) {Rank functions on $\cC$\strut};
		\node [style={rectangle_1}] (3) at (0, 10.5) {$\Sigma$-invariant additive functions on $\mod \cC$\strut};
		\node [style={fixed_width_rectangle_2}] (4) at (0, 5.75) {$\Sigma$-invariant cohomological endofinite functors in $\Mod \cC$\strut};
		\node [style={fixed_width_rectangle_2}] (5) at (-7, 0.75) {$\Sigma$-closed Serre subcategories of $\mod \cC$ with length quotient\strut};
		\node [style={fixed_width_rectangle_2}] (6) at (7, 0.75) {$\Sigma$-invariant \mbox{endofinite objects in $\cT$}\strut};
		\node [style={fixed_width_rectangle_3}] (7) at (-14, -4.5) {$\Sigma$-closed hereditary torsion \mbox{subcategories} of finite type
			of $\Mod \cC$ with locally finite quotient\strut};
		\node [style={fixed_width_rectangle_3}] (8) at (14, -4.5) {$\Sigma$-closed definable \mbox{subcategories} of $\cT$ \mbox{containing} only endofinite objects\strut};
		\node [style={fixed_width_rectangle_2}] (9) at (-7, -10) {$\Sigma$-closed closed discrete subsets of $\Zg(\Mod \cC)$ with the isolation property\strut};
		\node [style={fixed_width_rectangle_2}] (10) at (7, -10) {$\Sigma$-closed closed discrete subsets of $\Sp(\cT)$ whose coproduct of elements is $\varSigma$-pure-injective\strut};
		\node [style=none] (13) at (5.5, 3.5) {if $\cc={\cT}^c$};
		\node [style=ellipse] (15) at (8, 13) {Theorem \ref{thm:TransInvariantAddFunctions<->rankFunctions}\strut};
		\node [style=ellipse] (17) at (-15, 3.5) {Theorem \ref{thm:LongBijRankFunc}\strut};
		\node [style=ellipse] (18) at (15, 3.5) {Theorem \ref{thm:FundamentalCorrespondenceCompGen}\strut};
	\end{pgfonlayer}
	\begin{pgfonlayer}{edgelayer}
		\draw [style=arrow] (2) to (3);
		\draw [style=arrow] (3) to (4);
		\draw [style=arrow] (4) to (5);
		\draw [style=arrow] (4) to (6);
		\draw [style=arrow] (5) to (7);
		\draw [style=arrow] (6) to (8);
		\draw [style=arrow] (7) to (9);
		\draw [style=arrow] (9) to (5);
		\draw [style=arrow] (6) to (10);
		\draw [style=arrow] (10) to (8);
	\end{pgfonlayer}
\end{tikzpicture}}
\end{center}
\caption{Correspondences considered in the paper.}
\end{figure}

\subsection*{Further directions} One of our motivations to study rank functions is their similarity to valuations on objects in exact categories considered in \cite{FangGorsky1}. An upcoming work of Xin Fang and the second named author will further explore this analogy by showing that each rank function induces a flat degeneration of the derived Hall algebra of $\cC$, whenever the latter is well-defined, and by introducing a common generalisation of valuations and rank functions in the setting of extriangulated categories.

\subsection*{Organisation of the paper}

The paper is organised as follows. In Section \ref{sec:preliminaries}, we review necessary background on locally coherent categories and additive functions. In Section \ref{sec:RankFunctions}, we prove the correspondence between rank functions on triangulated categories and additive functions on their abelianisations.  In Section \ref{sec:Integral}, we study integral rank functions, prove the decomposition theorem and study exact and localising rank functions. Section \ref{sec:compact} is dedicated to the case of rank functions on the subcategory of compact objects in a compactly generated triangulated category. In Section \ref{sec:Example_Cluster}, we analyse rank functions on the cluster category of type $A_3$ -- an example of a locally finite triangulated category. We finish the paper with two appendices. In Appendix \ref{Appendix A}, we provide a proof of the bijection between basic additive functions on the category of finitely presented objects $\fp \cA$ of a locally coherent category $\cA$ and isoclasses of basic injective objects in $\cA$ of finite endolength. This is a version of the classical correspondence due to Crawley-Boevey between irreducible additive functions and indecomposable injectives of finite endolength (see \cite{Crawley-Boevey1994a}). In Appendix \ref{appendix:group_valued}, we discuss rank functions with values in partially ordered groups (generalising $d$-periodic rank functions from \cite{Chuang2021}) and extend Theorem \ref{thm:TransInvariantAddFunctions<->rankFunctions} to this context.   

\subsection*{Acknowledgements}
Most of this work was carried out during the employment of TC at the University of Stuttgart and she acknowledges support from the Deutsche Forschungsgemeinschaft (DFG) through
the grant KO 1281/18. TC also acknowledges partial support from DFG through the grant SFB-TRR 358/1 2023 — 491392403. MG was partially supported by the French ANR grant CHARMS~(ANR-19-CE40-0017). This work is a part of a project that has received funding from the European Research Council (ERC) under the European Union’s Horizon 2020 research and innovation programme (grant agreement No. 101001159). Parts of this work were done during stays of MG at the University of Stuttgart, and he is very grateful to Steffen Koenig for the hospitality. FM acknowledges support through a postdoctoral fellowship of the Baden-Württemberg Stiftung. AZ acknowledges support by the Czech Academy of Sciences CAS (RVO 67985840). Moreover, FM and AZ acknowledge support from DFG through a scientific network on silting theory. We thank the anonymous referee for carefully reading the manuscript and for suggesting various improvements in the exposition.

\subsection*{Conventions}

All categories are assumed to be locally small. All limits and colimits considered are therefore small. The letters $\cF$ and $\cA$ shall represent abelian categories, whereas $\cC$ and $\cT$ will typically denote triangulated categories, with $\cF$ and $\cC$ being in addition skeletally small. 
All subcategories are assumed to be full and closed under isomorphisms. For a subcategory $\cS$ in a triangulated category $\cT$, we denote by $\cS^\perp$ the subcategory of $\cT$ given by $\{Z\text{ in }\cT\mid \Hom_\cT(X,Z)=0\text{ for all }X\text{ in }\cS\}$. The subcategory $^\perp\cS$ is defined dually.

\section{Background on additive functions}
\label{sec:preliminaries}

In this section, we collect background material and draw connections between the notions of finite endolength object, additive function and Ziegler spectrum. The section culminates in Theorem \ref{thm:lastPreliminaries}, which will later be used to study integral rank functions on triangulated categories.

\subsection{Locally coherent categories}
\label{subsection lcc}

An object $X$ in an additive category $\cA$ with filtered colimits is called \emphbf{finitely presented} if the functor $\Hom_\cA (X,-):\cA \to \Ab$ preserves filtered colimits. The full subcategory of all finitely presented objects in $\cA$ is denoted by $\fp \cA$.

\begin{definition}
	An abelian category $\cA$ with filtered colimits is called \emphbf{locally coherent} if  $\fp \cA$ is a skeletally small abelian subcategory of $\cA$ such that
	every object in $\cA$ is a filtered colimit of objects in $\fp \cA$.
\end{definition}
In the following, $\cA$ will be a locally coherent category. Note that any locally coherent category is automatically Grothendieck, so $\cA$ is complete and cocomplete, has injective envelopes and exact filtered colimits (see \cite[\S 2.4]{Crawley-Boevey1994}). 
A short exact sequence in $\cA$
\[
\begin{tikzcd}
0 \ar[r] & X \ar[r] & Y \ar[r] & Z\ar[r] & 0
\end{tikzcd}
\]
is said to be \emphbf{pure-exact} if it induces a short exact sequence of abelian groups
\[
\begin{tikzcd}
0 \ar[r] & \Hom_\cA (W,X) \ar[r] & \Hom_\cA (W,Y) \ar[r] & \Hom_\cA (W,Z) \ar[r] & 0
\end{tikzcd}
\]
for every $W$ in $\fp \cA$. An object $X$ is \emphbf{pure-injective} if every pure-exact sequence with the first term $X$ splits, and \emphbf{$\varSigma$-pure-injective} if any coproduct of copies of $X$ is pure-injective. Moreover, $X$ is called \emphbf{fp-injective} if $\Ext_\cA^1 (W,X)=0$ for every $W$ in $\fp \cA$. The class of fp-injective objects is closed under direct factors and coproducts -- to see this one can use the characterisation of fp-injectivity in \cite[\S 2.5]{Crawley-Boevey1994}. A standard verification shows that the fp-injective pure-injective objects coincide with the injectives in the category $\cA$. Indeed, if $X$ is injective, it is both fp-injective and pure-injective. Conversely, if $X$ is fp-injective, then any short exact sequence starting at $X$ is pure-exact, as $\Ext_\cA^1 (W,X)=0$ for every $W$ in $\fp \cA$. 
If $X$ is moreover pure-injective, this implies that any short exact sequence starting at $X$ splits, so $X$ is injective.

The notion of a finite endolength object will be of central importance in this work. We are especially interested in injective objects of finite endolength.

\begin{definition}
	An object $X$ in $\cA$ has \emphbf{finite endolength} if the $\End_{\cA} (X)$-module $\Hom_\cA (W,X)$ has finite length for every $W$ in $\fp \cA$. 
\end{definition}

By \cite[\S 3.5, \S 3.6]{Crawley-Boevey1994}, every finite endolength object $X$ in $\cA$ is $\varSigma$-pure-injective and can be decomposed as a coproduct of indecomposable objects with local endomorphism rings. In fact, this decomposition is essentially unique by the Krull--Remak--Schmidt--Azumaya theorem. The class of finite endolength objects is closed under finite coproducts, direct factors, and  products and coproducts of copies of the same object (see \cite[Corollaries 13.1.15, 13.1.13]{Krause2022}). 
Note that the injective objects of finite endolength coincide with the fp-injective objects of finite endolength, since finite endolength objects are, in particular, pure-injective and for pure-injective objects fp-injectivity is equivalent to injectivity, as mentioned above.

 A locally coherent category $\cA$ is \emphbf{locally finite} if every object in $\fp \cA$ has finite composition length, that is, $\fp \cA$ is a \emphbf{length category}. Locally finite categories have particularly nice properties. By \cite[Proposition 7.1]{Herzog1997}, the injective objects in such categories are precisely the fp-injectives and, consequently, the class of injectives is closed under coproducts. Moreover, all injective objects have finite endolength:

\begin{lemma}
\label{lem:injFinEndolength}
Suppose that $\cA$ is locally finite. Then every injective object in $\cA$ has finite endolength. 
\end{lemma}
\begin{proof}
Consider an injective object $Q$ in $\cA$ and let $L$ be a simple object in $\fp \cA$ with $\Hom_{\cA}(L,Q)\neq 0$. Since every non-zero $f\in \Hom_{\cA}(L,Q)$ is a monomorphism and $Q$ is injective, every such $f$ generates $\Hom_{\cA}(L,Q)$ as a simple $\End_{\cA}(Q)$-module. By assumption, we can use induction on the composition length of every $W$ in $\fp \cA$ to conclude that $Q$ has finite endolength.
\end{proof}

 The prototype of a locally coherent category is the category $\Mod \cC$ of all additive contravariant functors from a skeletally small abelian or triangulated category $\cC$ to the category of abelian groups. Morphisms in $\Mod \cc$ are natural transformations between additive contravariant functors and a sequence of functors is a short exact sequence if and only if evaluating on each object of $\cc$ yields a short exact sequence in the category of abelian groups. We shall denote the category $\fp(\Mod \cC)$ simply by $\mod \cC$. The objects in $\mod \cC$ are exactly the functors which occur as cokernels of morphisms between \emphbf{representable functors}, i.e. between functors naturally isomorphic to $\Hom_{\cC}(-,X):\cC^{\op} \to \Ab$ for some $X$ in $\cC$. The direct factors of representables are the finitely presented projective objects in $\Mod \cC$. 
Let $H$ be an object in $\Mod \cC$ and consider an object $X$ in $\cC$. By identifying $\Hom_{\Mod \cC}(\Hom_\cC(-,X), H)$ with $H(X)$ via the Yoneda lemma,  $H(X)$ becomes a left module over the ring $\End_{\Mod \cC}(H)$. 

\begin{lemma}
\label{lem:finiteendolengthforfunctors}
	 An object $H$ in $\Mod\cC$ has finite endolength if and only if $H(X)$ has finite length as an $\End_{\Mod \cC}(H)$-module for every $X$ in $\cC$.
\end{lemma}
\begin{proof}
	One implication is clear, since representable functors are finitely presented. Conversely, suppose that $H(X)\cong \Hom_{\Mod \cC}(\Hom_\cC(-,X), H)$ is an $\End_{\Mod \cC}(H)$-module of finite length for every $X$ in $\cC$. Consider a functor $F$ in $\mod \cC$. There exists an epimorphism $\alpha:\Hom_\cC (-,X) \twoheadrightarrow F$ for some $X$ in $\cC$. By applying the functor $\Hom_{\Mod \cC} (-,H)$ to $\alpha$, we obtain a monomorphism $\Hom_{\Mod \cC} (F,H) \hookrightarrow \Hom_{\Mod \cC}(\Hom_\cC(-,X), H)$ of $\End_{\Mod \cC}(H)$-modules. Since $\Hom_{\Mod \cC}(\Hom_\cC(-,X), H)$ has finite length, so does $\Hom_{\Mod \cC} (F,H)$.
\end{proof}

\subsection{The Ziegler spectrum of a locally coherent category}
\label{subsec:zieglerspectrumlcc}
With every locally coherent category $\cA$ we may associate a topological space $\Zg \cA$, called
the \emphbf{Ziegler spectrum} of $\cA$. Its underlying set is the set of all isoclasses of indecomposable injective objects in $\cA$. 
Given a full subcategory $\cS$ of $\fp \cA$, define
\[
\co (\cS)\coloneqq\{[Q] \in \Zg \cA \mid \Hom_\cA (X,Q)\neq 0\text{ for some }X\text{ in }\cS \}.
\]
These are the open sets for the Ziegler topology on $\Zg \cA$.

Before stating the next result, recall that a full subcategory of $\ca$ is \emphbf{torsion} if it is closed under quotients, extensions and coproducts (see \cite{Dickson}). A torsion subcategory $\cx$ of $\cA$ is \emphbf{hereditary} if, in addition, it is closed under subobjects, and it is of \emphbf{finite type} if the right adjoint of the inclusion $\cx\hookrightarrow\cA$ commutes with filtered colimits. Finally, a subcategory $\cS$ of an abelian category is called \emphbf{Serre} if for every short exact sequence in the underlying abelian category
\[
\begin{tikzcd}
0 \ar[r] & X \ar[r] & Y \ar[r] & Z\ar[r] & 0,
\end{tikzcd}
\]
we have that $Y$ belongs to $\cS$ if and only if both $X$ and $Z$ belong to $\cS$. 

\begin{theorem}[{\cite[Theorems 2.8, 3.8]{Herzog1997}, \cite[Corollaries 2.10, 4.5]{Krause1997}}]
\label{thm:ZieglerSpecLocallyCohCat}
Let $\cA$ be a locally coherent category. The assignments $\cx \mapsto \cx \cap \fp \cA$ and $\cS \to \co (\cS)$ induce an inclusion-preserving bijective correspondence between, respectively:
	\begin{enumerate}[(a)]
	\item hereditary torsion subcategories of $\cA$ of finite type;
	\item Serre subcategories of $\fp \cA$;
	\item open sets of $\Zg \cA$.
	\end{enumerate}
\end{theorem}

\subsection{Additive functions on abelian categories}
\label{subsec:additivefunctionslcc}
Let $\cA$ be a locally coherent category and let $H$ be an injective object in $\cA$ which has finite endolength. 
We can assign to any object $W$ in $\fp \cA$ the non-negative integer
\begin{equation}
\label{eq:addFuncFromObj}
\widetilde{\rho}_H(W)\coloneqq\length_{\End_{\cA} (H)}\Hom_\cA(W,H).
\end{equation}
This assignment is additive on short exact sequences in $\fp \cA$, i.e.  
\[
\widetilde{\rho}_H(Y)=\widetilde{\rho}_H(X)+\widetilde{\rho}_H(Z) \text{ for any short exact sequence } \begin{tikzcd}[column sep=scriptsize]
0 \ar[r] & X \ar[r] & Y \ar[r] & Z\ar[r] & 0
\end{tikzcd} \text{ in } \fp \cA.
\]

\begin{definition}
	\label{def:additiveZfunction}
	An \emphbf{additive function} $\widetilde{\rho}$ on a skeletally small abelian category $\cF$ is an assignment of a non-negative real number $\widetilde{\rho}(X)$ to each object $X$ of $\cF$, which is additive on short exact sequences.
	We call $\widetilde{\rho}$ \emphbf{integral} if it takes values in $\Z$.
\end{definition}

Note that defining an additive function on $\cF$ is equivalent to defining a group morphism $\widetilde{\rho}:K_0 (\cF)\to \mathbb{R}$ satisfying $\widetilde{\rho}([W])\geq 0$ for every $W$ in $\cF$. It is clear that the sum of additive functions is again an additive function. In fact, any family $(\widetilde{\rho}_i)_{i\in I}$ of additive functions on $\cF$ for which the set $\{i \in I \mid \widetilde{\rho}_i(W)\neq 0\}$ is finite for every $W$ in $\cF$ gives rise to a well-defined additive function $\sum_{i\in I} \widetilde{\rho}_i$. This function is called the \emphbf{locally finite sum} of $(\widetilde{\rho}_i)_{i\in I}$.
Another way of creating new additive functions out of existing ones is by precomposition with an exact functor.

\begin{lemma}
\label{lem:AddFuncPrecomposedFunctor}
	Let $\cF$ and $\cF'$ be skeletally small abelian categories and suppose that $\Gamma:\cF \to \cF'$ is an exact functor. Consider an additive function $\widetilde{\rho}$ on $\cF'$. The assignment $\widetilde{\rho}^\Gamma(W)\coloneqq\widetilde{\rho}(\Gamma (W))$ for $W$ in $\cF$ defines an additive function $\widetilde{\rho}^\Gamma$ on $\cF$. If $\Gamma$ is essentially surjective and the additive function $\widetilde{\rho}$ decomposes as a sum of two non-zero additive functions, then $\widetilde{\rho}^\Gamma$ decomposes in this way as well.
\end{lemma}

\begin{proof}
	The first part is clear. Suppose that $\Gamma$ is essentially surjective and that $\widetilde{\rho}=\widetilde{\rho}_1+\widetilde{\rho}_2$ is a decomposition of $\widetilde{\rho}$ as a sum of two non-zero additive functions. Then $\widetilde{\rho}^ \Gamma=\widetilde{\rho}_1^ \Gamma+\widetilde{\rho}_2^ \Gamma$ and neither $\widetilde{\rho}_1^ \Gamma$ nor $\widetilde{\rho}_2^ \Gamma$ can be zero, since every $W'$ in $ \cF '$ is isomorphic to $\Gamma (W)$ for some $W$ in $ \cF$.
\end{proof}

An additive function is \emphbf{irreducible} if it is non-zero,  integral, and cannot be expressed as a sum of two non-zero integral additive functions. As indicated above, finite endolength injective objects in $\cA$ produce integral additive functions on $\fp\cA$. To a certain extent, the converse is also true.

\begin{theorem}[{\cite{Crawley-Boevey1994a}}]
\label{thm:IrrAdditiveFunctions<->InjFinEndolength}
	Let $\cA$ be a locally coherent category. Following the notation in \eqref{eq:addFuncFromObj}, the assignment $[H]\mapsto \widetilde{\rho}_{H}$ induces a bijective correspondence between:
	\begin{enumerate}[(a)]
		\item isoclasses of indecomposable injective objects in $\cA$ of finite endolength;
		\item irreducible additive functions on $\fp\cA$.
	\end{enumerate}
Moreover, every integral additive function on a skeletally small abelian category can be written in a unique way as a locally finite sum of irreducibles.
\end{theorem}

\noindent The second part of Theorem \ref{thm:IrrAdditiveFunctions<->InjFinEndolength} follows from \cite{Crawley-Boevey1994a} by observing that the correspondence $\cA \mapsto \fp \cA$ provides a bijection between locally coherent categories and skeletally small abelian categories (up to equivalence). The inverse maps a skeletally small abelian category to its (additive) ind-completion -- this is a particular case of the correspondence in \cite[\S1.4]{Crawley-Boevey1994} (see also \cite[\S 2, Proposition 2]{Roos1969}). Alternatively, see \cite[Lemma 13.1.1]{Krause2022}.

\subsection{Fundamental correspondence for integral additive functions}
\label{subsection:fundamental_corresp}
In this subsection, we define basic additive functions on $\fp\cA$, extend Theorem \ref{thm:IrrAdditiveFunctions<->InjFinEndolength} to a bijection between these functions and isoclasses of basic injective objects in $\cA$ of finite endolength, and link this bijection to Theorem \ref{thm:ZieglerSpecLocallyCohCat} by restricting the corresponding classes therein. To this end, we recall some background on Serre localisations dating back to \cite{Gabriel1962} (see also \cite{Herzog1997,Krause1997}).

Let $\cA$ be locally coherent and $\cS$ be a Serre subcategory of $\fp \cA$. We denote by $\overrightarrow{\cS}$ the smallest torsion subcategory of $\cA$ containing $\cS$. The category $\overrightarrow{\cS}$ is the hereditary torsion subcategory of $\cA$ associated with $\cS$ in Theorem \ref{thm:ZieglerSpecLocallyCohCat} and it is precisely the closure of $\cS$ under filtered colimits in $\cA$ (see \cite[\S 4.4]{Crawley-Boevey1994}). Since $\overrightarrow{\cS}$ is a hereditary torsion subcategory of $\cA$ of finite type, we get a localisation sequence
\begin{equation}
		\label{eq:loc}
		\begin{tikzcd}
		\overrightarrow{\cS} \ar[r, "I"] & \cA \ar[r, "E"] \ar[l, "P",bend left]  & \cA/\overrightarrow{\cS}, \ar[start anchor=-161,l,"R",bend left]
		\end{tikzcd}
\end{equation}	
where both $\overrightarrow{S}$ and the Serre localisation $\ca/\overrightarrow{S}$ are locally coherent categories. This satisfies the following properties, which will be of particular importance in the proof of Theorem \ref{thm:lastPreliminaries} (see Appendix \ref{Appendix A}):
\begin{enumerate}
			\item $I$ is the inclusion and $E$ is the Serre localisation functor with $\Imm I=\Ker E$. The functors $(I,P)$ and $(E,R)$ form adjoint pairs. 
			The functor $R$ is fully faithful and identifies $\cA/\overrightarrow{\cS}$ with the full subcategory of $\cA$ given by all objects $X$ with $\Hom_\cA(\overrightarrow{\cS},X)=0=\Ext_\cA^1(\overrightarrow{\cS},X).$ 
			\item The functors $R$ and $P$ commute with filtered colimits, and so the functors $I$ and $E$ restrict to finitely presented objects. More precisely, $\fp\overrightarrow{\cS}=\cS$ and 
			$$(\fp\cA)/\cS\simeq\fp(\cA/\overrightarrow{\cS})=\Imm(E|_{\fp \cA}).$$
			\item An object $\overline{Q}$ in $\cA/\overrightarrow{\cS}$ is injective if and only if $R(\overline{Q})$ is injective in $\cA$. 
			The assignment $[\overline{Q}]\mapsto[R(\overline{Q})]$ induces a homeomorphism from the Ziegler spectrum $\Zg(\cA/\overrightarrow{\cS})$ of $\cA/\overrightarrow{\cS}$ to the closed subset $\Zg \cA \setminus \co (\cS)$ of $\Zg\cA$ endowed with the subspace topology (\cite[Proposition 3.6]{Herzog1997}, \cite[Corollary 4.4]{Krause1997}).
			\item The functor $R$ preserves and reflects finite endolength: an object $Z$ in $\cA/\overrightarrow{S}$ has finite endolength if and only if $R(Z)$ has finite endolength. Indeed, for every $W$ in $\fp\cA$ and $Z$ in $\cA/\overrightarrow{\cS}$, the natural isomorphism 
\[
	\begin{tikzcd}
	\Hom_{\cA/\overrightarrow{\cS}}(E(W),Z)\ar[r, "\psi_{W,Z}"] & \Hom_{\cA}(W,R(Z))
	\end{tikzcd}
	\]
	 satisfies $\psi_{W,Z}(\alpha \circ f)=R(\alpha) \circ \psi_{W,Z}(f)$ for $f\in\Hom_{\cA/\overrightarrow{\cS}}(E(W),Z)$ and $\alpha\in \End_{\cA/\overrightarrow{\cS}} (Z)$. Moreover, $R$ induces a ring isomorphism $\End_{\cA/\overrightarrow{\cS}} (Z)\to \End_{\cA} (R(Z))$. Hence, the structure of $\Hom_{\cA/\overrightarrow{\cS}}(E(W),Z)$ as an $\End_{\cA/\overrightarrow{\cS}} (Z)$-module coincides with that of $\Hom_{\cA}(W,R(Z))$ as an $\End_{\cA} (R(Z))$-module. Since $\fp(\ca/\overrightarrow{S})$ coincides with $\Imm( E|_{\fp \cA})$, $Z$ has finite endolength if and only if $R(Z)$ does.
\end{enumerate}

\begin{remark}
The left-hand side of the localisation sequence \eqref{eq:loc} is less well-behaved in terms of endolength. It is not difficult to see that $Z$ in $\overrightarrow{\cS}$ has finite endolength whenever $I(Z)$ does. However, the converse is not true. For a counterexample, take $\ca$ to be $\Ab=\Mod \Z$ and let $\cS$ be the subcategory of finite abelian groups. Note that $\fp \cA$ is the category of finitely generated abelian groups and $\cS$ is a Serre subcategory of $\fp \cA$. In fact, $\cS$ is a length category.
For every prime $p$, the Pr\"ufer group $\Z(p^\infty)$ is the filtered colimit over $(\N,\leq)$ of the monomorphisms $\Z/p^n \Z \hookrightarrow \Z/ p^{n+1}$ induced by multiplication with $p$. Hence, every $\Z(p^\infty)$ belongs to the locally finite category $\overrightarrow{\cS}$. Moreover, each $\Z(p^\infty)$ is an injective in $\cA$ (see \cite[Theorem 3.4, Remark (2)]{Matlis1958}) and, consequently, also in $\overrightarrow{\cS}$, so Lemma \ref{lem:injFinEndolength} guarantees that $\Z(p^\infty)$ has finite endolength in $\overrightarrow{\cS}$. However, $\Z(p^\infty)$ does not have finite endolength in $\cA$. This is because $\Hom_{\ca}(\Z,\Z(p^\infty))\cong \Z(p^\infty)$ has infinite length as a module over $\End_{\ca}(\Z(p^\infty))$. To see this, note that $\End_{\ca}(\Z(p^\infty))$ is the ring $\Z_p$ of the $p$-adic integers, i.e.~up to isomorphism, $\End_{\ca}(\Z(p^\infty))$ is the filtered limit over $(\N,\leq^{\op})$ of the epimorphisms $\Z/p^{n+1} \Z \twoheadrightarrow \Z/ p^{n}\Z$. By construction, there exist ring morphisms $\Z_p \to \Z/p^n\Z$ for every $n\in \N$. These turn each $\Z/p^n\Z$ into a module over $\Z_p$, and the monomorphisms $\Z/p^n \Z \hookrightarrow \Z/ p^{n+1}\Z$ give an infinite chain of $\Z_p$-submodules of $\Z(p^\infty)$.
\end{remark}

Given an additive function $\widetilde{\rho}$ on a skeletally small abelian category $\cF$, define
\[\Ker \widetilde{\rho}\coloneqq\{X \text{ in }\cF \mid \widetilde{\rho}(X)=0 \}.
\]
Observe that $\Ker\widetilde{\rho}$ is a Serre subcategory of $\cF$.
The next proposition follows from \cite{Herzog1997} using the bijection between locally coherent categories and skeletally small abelian categories.

\begin{proposition}[{\cite[Proposition 8.4]{Herzog1997}}]
	\label{prop:LocAdditiveFunctions}
Let $\cF$ be a skeletally small abelian category and let $\widetilde{\rho}$ be an additive function on $\cF$.  If $\widetilde{\rho}$ is integral, then $\cF/\Ker\widetilde{\rho}$ is a length category.
\end{proposition}

An object $H$ in a locally coherent category shall be called \emphbf{basic} if it can be decomposed as a coproduct of pairwise non-isomorphic indecomposable objects with local endomorphism rings. Similarly, say that an additive function is \emphbf{basic} if it is integral and the summands appearing in its decomposition as a locally finite sum of irreducible additive functions are pairwise distinct. The following corollary can be deduced from the proof of Theorem \ref{thm:IrrAdditiveFunctions<->InjFinEndolength} appearing in \cite[Theorem 8.6]{Herzog1997}.

\begin{corollary}
\label{cor:charac_loc_finite_abelian}
    A skeletally small abelian category $\cF$ is a length category if and only if it admits a basic additive function $\widetilde{\rho}$ with $\Ker \widetilde{\rho}=0$. In this case, the function $\widetilde{\rho}$ is unique and it is given by the composition length of objects in $\cF$.
\end{corollary}
\begin{proof}
Let $\cF$ be a length category. The composition length of objects yields an integral additive function $\widetilde{\ell}$ with trivial kernel. To see that it is basic, suppose
 that $\{L_i\}_{i\in I}$ is the set of representatives of isoclasses of simple objects in $\cF$.
 Due to additivity, any additive function on $\cF$ is completely determined by its values on the simples $\{L_i\}_{i\in I}$.
 The integral additive functions $\widetilde{\rho}_i$ given by $\widetilde{\rho}_i(L_j)=\delta_{ij}$ form a complete set of irreducible additive functions on $\cF$. Thus, $\widetilde{\ell}=\sum_{i\in I}\widetilde{\rho}_i$ is a basic additive function on $\cF$ with trivial kernel.  The uniqueness follows from the last part of Theorem \ref{thm:IrrAdditiveFunctions<->InjFinEndolength}. Conversely, suppose that $\cF$ admits a basic additive function $\widetilde{\rho}$ with trivial kernel. Then $\cF/\Ker \widetilde{\rho}\simeq\cF$ is a length category by Proposition \ref{prop:LocAdditiveFunctions}.
\end{proof}

As a consequence of Corollary \ref{cor:charac_loc_finite_abelian}, we can deduce the converse of Lemma \ref{lem:injFinEndolength}.

\begin{lemma}
\label{lem:convInjFinEndolength}
Suppose that $\cA$ is a locally coherent category such that every injective has finite endolength. Then the category $\cA$ is locally finite.
\end{lemma}

\begin{proof}
Let $\{M_i\}_{i\in I}$ be a set of representatives of isoclasses of objects in $\fp\ca$ and let $\{Q_i\}_{i\in I}$ be the set of their injective envelopes in $\ca$. The object $H\coloneqq\prod_{i\in I}Q_i$ is injective and, thus, has finite endolength by our assumption. The kernel $\Ker \widetilde{\rho}_H$ of the additive function $\widetilde{\rho}_H$ on $\fp\cA$ is trivial. Throwing out repeating summands from the decomposition of $\widetilde{\rho}_H$ yields a basic additive function with trivial kernel. By Corollary \ref{cor:charac_loc_finite_abelian}, $\cA$ is locally finite.
\end{proof}

As we have just seen, locally finite categories can be identified among locally coherent categories by the existence of a special integral additive function (see Corollary \ref{cor:charac_loc_finite_abelian}) or by the fact that all of their injective objects have finite endolength (see Lemmas \ref{lem:injFinEndolength} and \ref{lem:convInjFinEndolength}). According to results in \cite{Krause1998}, locally finite categories are also characterised by the fact that their Ziegler spectrum is discrete and satisfies the isolation property. For a locally coherent category $\ca$, we say that a closed subset $C$ of $\Zg \cA$ satisfies the \emphbf{isolation property} if for every  point $[Q] \in C$ such that $\{[Q]\}$ is open in its closure $\overline{\{[Q]\}}$, the Serre localisation $(\fp \cA)/\Ker(\Hom_\ca (-,Q)|_{\fp\cA})$ has a simple object. 
Here, as usual, the set $\overline{\{[Q]\}}$ is endowed with the subspace topology inherited from $\Zg \cA$. 
The isolation property is a useful hypothesis that is observed in many cases. For example, it always holds for $\Zg\cA$ if the Krull--Gabriel dimension of $\fp\cA$ is defined or if $\cA$ has no superdecomposable injectives -- this follows from \cite[Corollary 12.5, Theorem 11.2]{Krause1998}. Note that in \cite{Krause1998} the isolation property is defined in a slightly different context. The precise relation to the version we use will be explained in Appendix \ref{Appendix A}.
The following theorem encapsulates the takeaway message of this section. Its proof is also given in Appendix \ref{Appendix A}.

\begin{theorem}
	\label{thm:lastPreliminaries}
		Let $\cA$ be a locally coherent category. The assignments $[H]\mapsto \widetilde{\rho}_{H}$,  $\widetilde{\rho}\mapsto \overrightarrow{\Ker\widetilde{\rho}}$, $\cx \mapsto \cx \cap \fp \cA$ and $\cS \mapsto \Zg \cA \setminus\co(\cS)$ induce a bijective correspondence between, respectively:
		\begin{enumerate}[(a)]
			\item isoclasses of basic injective objects in $\cA$ of finite endolength;
			\item basic additive functions on $\fp\cA$;
			\item hereditary torsion subcategories of finite type $\cx$ of $\cA$ such that $\cA/\cx$ is a locally finite category;
			\item Serre subcategories $\cS$ of $\fp \cA$ such that $(\fp\cA)/\cS$ is a length category;
			\item closed discrete subsets of $\Zg \cA$ satisfying the isolation property.
		\end{enumerate}
\end{theorem}

The correspondence between (a) and (b) restricts to the bijection in Theorem \ref{thm:IrrAdditiveFunctions<->InjFinEndolength}. From the proof of Theorem \ref{thm:lastPreliminaries} it is apparent that this correspondence respects decompositions. Each basic object $H$ admits, by definition, a decomposition $H=\coprod_{i\in I}H_i$ into a coproduct  of  pairwise non-isomorphic indecomposable objects $H_i$ with local endomorphism rings. If $H$ is injective, so are the $H_i$. If, moreover, $H$ has finite endolength as in (a), then the objects $H_i$ have finite endolength, as explained in Subsection \ref{subsection lcc}. We then have $\widetilde{\rho}_H=\sum_{i\in I}\widetilde{\rho}_{H_i}$. The correspondence between (c), (d) and (e) is a restricted version of the bijections in Theorem \ref{thm:ZieglerSpecLocallyCohCat}. 

We conclude this section with further observations that will be used later on. These remarks can be deduced from Theorem \ref{thm:lastPreliminaries} and from the material in Appendix \ref{Appendix A}.

\begin{remark}
\label{pullback}
The inverse of the assignment $\widetilde{\rho} \mapsto \Ker \widetilde{\rho}$ from (b) to (d) maps a Serre subcategory $\cS$ of $\fp\cA$  to $\widetilde{\ell}^E$, where $E:\fp\cA\to (\fp\cA)/\cS$ denotes the corresponding localisation functor and  $\widetilde{\ell}$ denotes the composition length on $(\fp\cA)/\cS$ (see Lemma \ref{lem:AddFuncPrecomposedFunctor} and the proof of Proposition \ref{prop:IrrAdditiveFunctions<->InjFinEndolength2}).
\end{remark}

\begin{remark}\label{IsolatedRemark}
Going from (a) to (e) in Theorem \ref{thm:lastPreliminaries} corresponds precisely to taking the indecomposable summands $\{[H_i]\}_{i\in I}$ of $H=\coprod_{i\in I}H_i$ (see part (1) of Proposition \ref{prop:IrrAdditiveFunctions<->InjFinEndolength2}). In fact, the sets appearing in (e) could be equivalently described as subsets of $\Zg \cA$ such that the coproduct of representatives of all their elements has finite endolength. Another equivalent characterisation of the  sets appearing in (e) is the following:
\begin{enumerate}
    \item[(e')] closed discrete subsets of $\Zg\cA$ such that the coproduct of (representatives of) all their elements is $\varSigma$-pure-injective.
\end{enumerate}
Since objects of finite endolength are $\varSigma$-pure-injective, it is clear that all the sets appearing in (e) are in (e'). The converse follows from Lemmas \ref{lem:IsoClosedIsoQuotient} and \ref{lem:IsolationLocFin} and the properties of localisation sequences. 
\end{remark}

\section{Rank functions on triangulated categories}
\label{sec:RankFunctions}

In this section, we study rank functions on  triangulated categories via additive functions on the corresponding module categories.  We show that rank functions on $\cC$ can be identified with translation-invariant additive functions on $\mod \cC$.

\subsection{The module category of a triangulated category}\label{subsection RF1}
Suppose henceforth that $\cC$ is a skeletally small triangulated category with translation functor $\Sigma: \cC \to \cC$. In that case, the category $\Mod \cC$ is locally coherent and $\mod \cC$ is a Frobenius abelian category with the direct factors of representable functors as projective-injective objects. Hence, a functor $F$ in $\Mod \cC$ is finitely presented if and only if it admits a copresentation  by representable functors or, equivalently, if $F$ is naturally isomorphic to $\Imm \Hom_\cC(-,f)$ for some morphism $f$ in $\cC$. 

Recall that any triangulated functor $\Phi: \cC \to \cd$ between skeletally small triangulated categories gives rise to an exact \emphbf{restriction functor} $\Phi_*:\Mod \cd \to \Mod \cc$ via precomposition with $\Phi$. The functor $\Phi_*$ has an exact left adjoint $\Phi^*:\Mod \cc \to \Mod \cd$, the \emphbf{induction functor}, which is characterised as the unique colimit preserving functor from $\Mod\cc$ to $\Mod \cd$ that maps $\Hom_\cC(-,f)$  to $\Hom_\cd (-,\Phi(f))$ (see \cite[\S2.2, Lemma 2.2]{Krause2000} or \cite[\S1, Lemmas 3.6, 3.7]{Krause2005}). In particular, $\Phi^*$ preserves representable functors and finitely presented objects, so it restricts to a functor $\Phi^*|_{\mod \cc}:\mod \cc \to \mod \cd$. In fact, since $\Phi^*$ is exact, we get $\Phi^*(\Imm \Hom_\cC(-,f))\cong \Imm \Phi^*(\Hom_\cC(-,f))\cong \Imm \Hom_\cd(-,\Phi f)$. The induction $\Sigma^*: \Mod \cC \to \Mod \cC$ of the translation functor  is an equivalence,  mapping $\Hom_\cC(-,f)$  to $\Hom_\cC (-,\Sigma f)$. It is not hard to see that, up to a natural isomorphism, $\Sigma^*$ is given by the restriction of $\Sigma^{-1}$.

A functor from a triangulated category to an abelian category is called \emphbf{cohomological} if it maps triangles to exact sequences. Unless mentioned otherwise, all the cohomological functors we consider are assumed to be objects in $\Mod \cC$. We use the following well-known characterisation of cohomological functors.

\begin{lemma}[{\cite[Lemma 2.7]{Krause2000}}]
\label{lem:cohomologicalFpinjectiveFlat}
	Let $\cC$ be a skeletally small triangulated category. The following statements are equivalent for an object $H$ in $\Mod \cC$:
	\begin{enumerate}[(a)]
		\item $H$ is a cohomological functor;
		\item $H$ is fp-injective;
		\item $H$ is a filtered colimit of representable functors.
	\end{enumerate}
\end{lemma}

We say that an object $H$ in $\Mod \cC$ is \emphbf{endofinite} if $H(X)$ has finite length as a module over $\End_{\Mod \cC}(H)$ for every $X$ in $\cC$. By Lemma \ref{lem:finiteendolengthforfunctors}, the endofinite objects in $\Mod \cC$ are exactly the objects of finite endolength. Note that the term endofinite was used differently in \cite{Krause2000} (see Definition 1.2 loc.~cit.). We use this term in a broader sense. Since fp-injective objects of finite endolength are injective by Subsection \ref{subsection lcc}, we have the following corollary.

\begin{corollary}
\label{cor:CohEndofin=InjFinEndoLength}
	Finite endolength injectives in $\Mod \cC$ are precisely the endofinite cohomological functors.
\end{corollary}

\subsection{Rank functions}
\label{subsec:RankFun<->SigmaInvAddFun}
We shift our attention to rank functions on triangulated categories, as introduced by Chuang and Lazarev in \cite{Chuang2021}.

\begin{definition}[{\cite[Definition 2.3]{Chuang2021}}]
	\label{def:rankfunction}
	 A \emphbf{rank function} $\rho$ on $\cC$ assigns to each morphism $f$ in $\cC$ an element $\rho(f)\in \R$ such that $\rho$ satisfies the following axioms:
	\begin{enumerate}[(M1)]
	\item \emphbf{Non-negativity:} $\rho(f) \geq 0$ for every morphism $f$ in $\cC$;
	\item \emphbf{Additivity:} $\rho(f \oplus g) = \rho(f) + \rho(g)$ for any two morphisms $f$ and $g$;
	\item \emphbf{Rank-nullity:} $\rho(f) + \rho(g)=\rho(\id_Y)$ for any triangle
	\[
	\begin{tikzcd}
	X \ar[r, "f"] & Y \ar[r, "g"] & Z \ar[r] & \Sigma X;
	\end{tikzcd}
	\]
		\item \emphbf{$\Sigma$-invariance:} $\rho(\Sigma f) = \rho(f)$ for any morphism $f$ in $\cC$.
	\end{enumerate}
\end{definition}

Alternatively, rank functions can be defined on objects of $\cC$.

\begin{definition}[{\cite[Definition 2.1]{Chuang2021}}]
	\label{def:rankfunctionobj}
	 A \emphbf{rank function on objects} $\rho_{ob}$ of $\cC$ assigns to each object $X$ in $\cC$ an element $\rho_{ob}(X)\in \R$ such that $\rho_{ob}$ satisfies the axioms:
\begin{enumerate}
	\item[(O1)] \emphbf{Non-negativity:} $\rho_{ob}(X)\geq 0$ for every object $X$ in $\cC$;
	\item[(O2)] \emphbf{Additivity:}  $\rho_{ob}(X\oplus Y)=\rho_{ob}(X)+\rho_{ob}(Y)$ for $X$ and $Y$ in $\cC$;
	\item[(O3)] \emphbf{Triangle inequality:} $\rho_{ob}(Y) \leq \rho_{ob}(X) + \rho_{ob}(Z)$ for any triangle
	\[
	    \begin{tikzcd}
        	X \ar[r] & Y \ar[r] & Z \ar[r] & \Sigma X;
	    \end{tikzcd}
	\]
	\item[(O4)] \emphbf{$\Sigma$-invariance:} $\rho_{ob}(\Sigma X) = \rho_{ob} (X)$ for every $X$ in $\cC$.
\end{enumerate}
\end{definition}
\noindent By \cite[Proposition 2.4]{Chuang2021}, both definitions above are equivalent by setting $\rho_{ob}(X)\coloneqq\rho(\id_X)$ and
\[\rho(f: X\to Y)\coloneqq\dfrac{\rho_{ob}(X)+\rho_{ob}(Y)-\rho_{ob}(\Cone f)}{2}.\]

\begin{remark}\label{remark functions}
Note that it is implicit in axiom (M2) that rank functions are constant on isomomorphism classes of morphisms, since coproducts are defined only up to isomorphism. Two morphisms $f$ and $f'$ in some category are isomorphic ($f \sim f'$) if there exist isomorphisms $\alpha$ and $\beta$ such that $f' \circ \alpha=\beta \circ f$.  For any assignment $\rho$ on $\Mor \cC$ satisfying (M2), $f \sim f'$ implies that both $f$ and $f'$ represent $f\oplus 0$, so the values of $\rho$ on $f$ and $f'$ coincide. 
 The invariance of $\rho$ on isomorphism classes of morphisms independently follows from (M3). It is enough to check that (M3) implies that $\rho(f)=\rho(f\circ \alpha)=\rho(\beta\circ f)$ for isomorphisms $\alpha$ and $\beta$. A triangle starting with $f$ 
\[
\begin{tikzcd}[ampersand replacement=\&]
	X \ar[r, "f"] \& Y \ar[r, "g"] \& Z \ar[r, "h"] \& \Sigma X
\end{tikzcd}
\]
is isomorphic to the sequence of morphisms
\[
\begin{tikzcd}[column sep = large,ampersand replacement=\&]
	X' \ar[r, "f \circ \alpha"] \& Y \ar[r, "g"] \& Z \ar[r, "\Sigma (\alpha^{-1}) \circ h"] \& \Sigma X'
\end{tikzcd}
\]
 which therefore must also be a triangle in $\cC$. The application of (M3) to both triangles yields $\rho(f)=\rho(f \circ \alpha)$. The proof of the equality $\rho(f)=\rho(\beta \circ f)$ is analogous. Since  $\cC$ is skeletally small,  $\Mor \cC/ \sim$ forms a set.
By the above, $\rho$ can be regarded as an actual function on this set.
Analogously, any assignment satisfying (O2) is constant on isomorphism classes of objects in  $\cC$, and so can be regarded as an actual function on a set.
\end{remark}

\begin{example}
\label{ex:fin-dim_alg}
Let $A$ be a finite-dimensional algebra over a field $k$. Denote by $\operatorname{D^b}(A)$ the bounded derived category of finite-dimensional $A$-modules and by  $\operatorname{K^b(proj}A\hspace{-2pt}\operatorname{)}\simeq \cper (A)$ the homotopy category of bounded complexes of finite-dimensional projective $A$-modules, which naturally embeds into $\operatorname{D^b}(A)$. For any object $X$ in $\operatorname{K^b(proj}A\hspace{-2pt}\operatorname{)}$ we can consider the rank function on the category $\operatorname{D^b}(A)$ defined on objects as follows:
\[
	\vartheta_X(Y)\coloneqq\sum_{i\in\Z}\dim_k\Hom_{\operatorname{D^b}(A)}(\Sigma^iX,Y).
\]
Indeed, $\vartheta_X(Y)$ is well-defined, since $\dim_k\Hom_{\operatorname{D^b}(A)}(\Sigma^iX,Y)\neq 0$ only for finitely many $i\in \Z$. Axioms (O1), (O2) and (O4) are clear from the construction. Applying $\Hom_{\operatorname{D^b}(A)}(\Sigma^iX,-)$ to a triangle
\[
    \begin{tikzcd}[ampersand replacement=\&]
	    Y \ar[r] \& Z \ar[r] \& W \ar[r] \& \Sigma Y
    \end{tikzcd}
\]
yields a long exact sequence 
\begin{equation}\label{LongExSeqForEx}
    \begin{tikzcd}[ampersand replacement=\&]
	    \cdots \ar[r] \& \Hom_{\operatorname{D^b}(A)}(\Sigma^iX,Y) \ar[r] \& \Hom_{\operatorname{D^b}(A)}(\Sigma^iX,Z) \ar[r] \& \Hom_{\operatorname{D^b}(A)}(\Sigma^iX,W) \ar[r] \& \cdots
	\end{tikzcd}
\end{equation}
which implies that 
\[ \dim_k\Hom_{\operatorname{D^b}(A)}(\Sigma^iX,Z)\leq \dim_k\Hom_{\operatorname{D^b}(A)}(\Sigma^iX,Y)+\dim_k\Hom_{\operatorname{D^b}(A)}(\Sigma^iX,W)\]
for any $i\in\Z$ and thus $\vartheta_X(Z)\leq \vartheta_X(Y)+\vartheta_X(W).$

In particular, in case $X=A$, we get that $\vartheta_A(Y)$ is the dimension of the total cohomology of $Y$.
Dually, for any object $Z$ in $\operatorname{D^b}(A)$ we can consider the following rank function on the category  $\operatorname{K^b(proj}A\hspace{-2pt}\operatorname{)}$:
 \[
 	\vartheta^Z(Y)\coloneqq\sum_{i\in\Z}\dim_k\Hom_{\operatorname{D^b}(A)}(Y,\Sigma^iZ).
 \]
 Note that the same formula defines a rank function on $\operatorname{D^b}(A)$, if the object $Z$ is in the homotopy category $\operatorname{K^b(inj}A\hspace{-2pt}\operatorname{)}$ of bounded complexes of finite-dimensional injective $A$-modules.
  
 Let us denote by $\bbD\coloneqq\Hom_k(-,k)$ the standard duality. The Nakayama functor 
 \[\nu\coloneqq-\otimes_A^L\bbD A: \operatorname{K^b(proj}A\hspace{-2pt}\operatorname{)} \to \operatorname{K^b(inj}A\hspace{-2pt}\operatorname{)}\]
 satisfies the property
\begin{equation}\label{Serre}
    \bbD \Hom_{\operatorname{D^b}(A)}(X,Y)\cong \Hom_{\operatorname{D^b}(A)}(Y,\nu X) 
\end{equation}
for any objects $X$ in $\operatorname{K^b(proj}A\hspace{-2pt}\operatorname{)}$ and $Y$ in $\operatorname{D^b}(A)$.
One can easily see that the following two rank functions on $\operatorname{D^b}(A)$ coincide:
\[ \vartheta_X=\vartheta^{\nu X} \text{ for } X \text{ in } \operatorname{K^b(proj}A\hspace{-2pt}\operatorname{)}.\]
Observe that in the case  when the global dimension of $A$ is finite, the categories $\operatorname{D^b}(A)$, $\operatorname{K^b(proj}A\hspace{-2pt}\operatorname{)}$ and $\operatorname{K^b(inj}A\hspace{-2pt}\operatorname{)}$ coincide, and so it follows from \eqref{Serre} that $\nu$ becomes a Serre functor on $\operatorname{D^b}(A)$.
\end{example}

\begin{example}
Similarly to the previous example, we can take $A$ to be a dg algebra over a field $k$. Then for an object $X$ in the perfect derived category $\cper (A)$ one can define a rank function $\vartheta_X$ on the subcategory $\D_{\operatorname{fd}}(A)$ of the derived category of $A$,
consisting of dg $A$-modules whose total cohomology is finite-dimensional over $k$. As in Example \ref{ex:fin-dim_alg}, one can also consider a rank function $\vartheta^Z$ on $\cper (A)$ for $Z$ in $\D_{\operatorname{fd}}(A)$.
\end{example}

\begin{example}
Let $\mathbb{X}$ be a projective scheme over a field $k$. Denote by $\cper (\mathbb{X})$ the derived category of perfect complexes over $\mathbb{X}$ and by $\operatorname{D^b}(\operatorname{coh}\mathbb{X})$ the bounded derived category of coherent sheaves over $\mathbb{X}$. Similarly to Example \ref{ex:fin-dim_alg}, for $X\in\cper (\mathbb{X})$ we can consider rank functions of the form $\vartheta_X$ on $\operatorname{D^b}(\operatorname{coh}\mathbb{X})$ and rank functions of the form $\vartheta^Z$ on $\cper (\mathbb{X})$ for $Z$ in $\operatorname{D^b}(\operatorname{coh}\mathbb{X})$. The fact that the corresponding functions are well-defined follows from \cite[Lemmas 7.46, 7.49]{Rouquier}.
\end{example}

\subsection{Rank functions via the module category}
\label{Subsecaddtorank}

An additive function $\widetilde{\rho}$ on $\mod \cC$ induces an assignment $\rho$ on the class $\Mor \cC$ of morphisms in $\cC$ by setting $\rho(f)\coloneqq \widetilde{\rho}(\Imm \Hom_{\cC}(-,f))$. The assignment $\rho$ satisfies the conditions (M1), (M2) and (M3).
Indeed, (M1) and (M2) are immediate. For (M3) consider a triangle in $\cC$ 
	\[
	\begin{tikzcd}
	X \ar[r, "f"] & Y \ar[r, "g"] & Z \ar[r] & \Sigma X
	\end{tikzcd}
	\]
	and the induced exact sequence in $\mod \cC$
	\[
	\begin{tikzcd}
	\Hom_\cC(-,X) \ar[r,"f^\dagger"] & \Hom_\cC(-,Y) \ar[r, "g^\dagger"] & \Hom_\cC(-,Z).
	\end{tikzcd}
	\]
    By assumption,
	$\widetilde{\rho}(\Hom_\cC(-,Y))=\widetilde{\rho}(\Imm f^\dagger)+\widetilde{\rho}(\Imm g^\dagger)$,
	which can be rewritten as $\rho(\id_Y)=\rho(f) + \rho(g)$.
	
Going in the opposite direction, given an assignment $\rho: \Mor \cC \to \R$ satisfying conditions (M1), (M2) and (M3), it may be possible to find morphisms $f$ and $f'$ in $\cC$ such that $\rho(f)\neq \rho (f')$ but $\Imm \Hom_{\cC}(-,f)\cong \Imm \Hom_{\cC}(-,f')$.  Hence, we may not be able to construct a well-defined additive function $\widetilde{\rho}$ from $\rho$, and the injective correspondence $\widetilde{\rho} \mapsto \rho$ may in fact not be surjective. Such pairs of morphisms $(f, f')$ do not exist if $\cC$ is Krull--Schmidt.

\begin{lemma}
    Let $\cC$ be Krull--Schmidt and $\rho: \Mor \cC \to \R$ be an assignment satisfying condition (M2). Let $X\xrightarrow{f} Y$ and $X'\xrightarrow{f'} Y'$ be morphisms in $\cC$ such that $\Imm \Hom_{\cC}(-,f)\cong \Imm \Hom_{\cC}(-,f')$. Then  $\rho(f)=\rho(f').$
\end{lemma}
\begin{proof}
    Denote by $\pi$  the projection $\Hom_{\cC}(-,X)\rightarrow \Imm \Hom_{\cC}(-,f)$ and by $\iota$ the inclusion $\Imm \Hom_{\cC}(-,f)\rightarrow \Hom_{\cC}(-,Y)$. The morphisms $\pi'$ and $\iota'$ are defined accordingly. Twisting $\pi'$ and $\iota'$ with an isomorphism we can identify $\Imm \Hom_{\cC}(-,f)$ and $\Imm \Hom_{\cC}(-,f')$. 
Since $\cC$ is Krull--Schmidt, we may apply \cite[Corollary 1.4]{KrauseSaorin} and this guarantees the existence of a decomposition $\Hom_{\cC}(-,Y)=\Hom_{\cC}(-,Y_1)\oplus \Hom_{\cC}(-,Y_2)$ such that $\iota=\begin{pmatrix}\iota_1 \\  0\end{pmatrix}$ and $\iota_1$ is left minimal. Analogously, there is a decomposition $\Hom_{\cC}(-,Y')=\Hom_{\cC}(-,Y'_1)\oplus \Hom_{\cC}(-,Y'_2)$ such that $\iota'=\begin{pmatrix}\iota'_1 \\  0\end{pmatrix}$ with $\iota'_1$ left minimal.

Since $\Hom_{\cC}(-,Y)$ and $\Hom_{\cC}(-,Y')$ are injective in $\mod \cC$, the identity $\Imm \Hom_{\cC}(-,f) = \Imm \Hom_{\cC}(-,f')$ gives a pair of morphisms between $\Hom_{\cC}(-,Y)$ and $\Hom_{\cC}(-,Y')$ which induces an isomorphism $\Hom_{\cC}(-,Y_1)\cong \Imm \Hom_{\cC}(-,Y'_1)$ by the minimality of $\iota_1$ and $\iota'_1$. Note that $\iota_1 \sim \iota'_1$ via this isomorphism (see Remark \ref{remark functions} for notation). Dually we get decompositions $\Hom_{\cC}(-,X)=\Hom_{\cC}(-,X_1)\oplus \Hom_{\cC}(-,X_2)$ with $\pi= \begin{pmatrix}\pi_1 &  0\end{pmatrix}$ and $\Hom_{\cC}(-,X')=\Hom_{\cC}(-,X'_1)\oplus \Hom_{\cC}(-,X'_2)$ with $\pi'=\begin{pmatrix}\pi'_1 &  0\end{pmatrix}$ together with an isomorphism $\Hom_{\cC}(-,X_1)\cong \Imm \Hom_{\cC}(-,X'_1)$ providing $\pi_1 \sim \pi'_1$. 

Finally, $\Hom_{\cC}(-,f)=\iota\circ\pi=\begin{pmatrix}\iota_1\circ\pi_1 & 0 \\ 0& 0\end{pmatrix}$ and $\Hom_{\cC}(-,f')=\iota'\circ\pi'=\begin{pmatrix}\iota'_1\circ\pi'_1 & 0 \\ 0& 0\end{pmatrix}$. By the Yoneda lemma, $\iota_1\circ\pi_1=\Hom_{\cC}(-,f_1)$ and $\iota'_1\circ\pi'_1=\Hom_{\cC}(-,f'_1)$ for some $X_1\xrightarrow{f_1}Y_1$ and $X'_1\xrightarrow{f'_1}Y'_1$. Since $\iota_1\circ\pi_1 \sim \iota'_1\circ\pi'_1$, we have $f_1 \sim f'_1$. By (M2), we get $\rho(f)=\rho(f_1)+\rho(0)=\rho(f_1)=\rho(f'_1)=\rho(f'_1)+\rho(0)=\rho(f').$
\end{proof}

Even if we obtain a well-defined function $\widetilde{\rho}$ via the assignment $\widetilde{\rho}(F)\coloneqq\rho(f)$ for $F\cong \Imm \Hom_\cC(-,f)$, the main difficulty is proving additivity of $\widetilde{\rho}$. The functions for which axioms (M1), (M2) and (M3) hold satisfy the following properties.
\begin{lemma}
	\label{lem:PropPreRankFunctions}
	Let $\rho: \Mor \cC \to \R$ be an assignment satisfying properties (M1), (M2) and (M3). Then the following statements hold:
	\begin{enumerate}
		\item $\rho(f)+ \rho(\Sigma f)=\rho(\id_Y)-\rho(\id_{\Cone f})+\rho(\id_{\Sigma X})$ for every $f: X \to Y$ in $\cC$;
		\item  $\rho(f)+ \rho(\Sigma f)=\rho(f')+ \rho(\Sigma f')$ for morphisms $f$ and $f'$ satisfying $\Imm \Hom_{\cC}(-,f)\cong \Imm \Hom_{\cC}(-,f')$.
	\end{enumerate}
\end{lemma}
\begin{proof}
Let $f:X \to Y$ and $f':X' \to Y'$ be morphisms in $\cC$ such that $\Imm \Hom_{\cC}(-,f)\cong \Imm \Hom_{\cC}(-,f')$. Consider triangles
\[
\begin{tikzcd}
X \ar[r, "f"] & Y \ar[r, "g"] & Z \ar[r, "h"] & \Sigma X,
\end{tikzcd} \quad
\begin{tikzcd}
X' \ar[r, "f'"] & Y' \ar[r, "g'"] & Z' \ar[r, "h'"] & \Sigma X'.
\end{tikzcd}
\]
Using (M3) and (M2), we deduce that
\begin{align}
\label{eq:PropPreRankFunctions}
\rho(f)+\rho(\Sigma f)&=(\rho(f)+\rho(g))-(\rho(g)+\rho(h)) + (\rho(h)+\rho(\Sigma f))=\rho(\id_Y)-\rho(\id_Z)+\rho(\id_{\Sigma X})\\
&=\rho(\id_{Y\oplus Z'\oplus \Sigma X})-\rho(\id_{Z'}) - \rho(\id_{Z}) \nonumber.
\end{align}
The second equality in \eqref{eq:PropPreRankFunctions} proves statement (1). A similar identity holds for $\rho(f')+\rho(\Sigma f')$. Using the assumption on $f$ and $f'$ in (2), by \cite[Lemma A.1]{Krause2016}, the objects $Y\oplus Z'\oplus \Sigma X$ and $Y'\oplus Z\oplus \Sigma X'$ are isomorphic, so $\rho(\id_{Y\oplus Z'\oplus \Sigma X})=\rho(\id_{Y'\oplus Z\oplus \Sigma X'})$ following Remark \ref{remark functions}. Hence $\rho(f)+\rho(\Sigma f)=\rho(f')+\rho(\Sigma f')$, which proves assertion (2).
\end{proof}

In order to relate rank functions on $\cC$ to additive functions on $\mod \cC$, we need the following notion. 
 An additive function $\widetilde{\rho}$ on $\mod \cC$ is called \emphbf{$\Sigma$-invariant} if $\widetilde{\rho}({\Sigma}^*F)=\widetilde{\rho}(F)$ for every $F$ in $\mod \cC$. 
We are now ready to state our first main result.

\begin{theorem}
	\label{thm:TransInvariantAddFunctions<->rankFunctions}
	Let $\cC$ be a skeletally small triangulated category. There is a bijective correspondence between:
	\begin{enumerate}[(a)]
		\item $\Sigma$-invariant additive functions $\widetilde{\rho}$ on $\mod \cC$;
		\item rank functions $\rho$ on $\cC$.
	\end{enumerate}
	The correspondence maps an additive function $\widetilde{\rho}$ to the rank function $\rho$ given by $\rho(f)\coloneqq\widetilde{\rho}(\Imm \Hom_\cC(-,f))$. The inverse maps $\rho$ to the additive function $\widetilde{\rho}$ defined by $\widetilde{\rho}(F)\coloneqq\rho(f)$ for $F\cong \Imm \Hom_\cC(-,f)$.
\end{theorem}

\begin{proof}
	Let $\widetilde{\rho}$ be a $\Sigma$-invariant additive function on $\mod \cC$. As discussed before, the assignment $\rho$ given by $\rho(f)\coloneqq\widetilde{\rho}(\Imm \Hom_\cC(-,f))$ satisfies (M1), (M2) and (M3). Moreover, we have
	\begin{equation}
	\label{eq:qTransInv}
	\rho(\Sigma f)
	=\widetilde{\rho}(\Imm \Hom_\cC(-,\Sigma f))
	=\widetilde{\rho}( \Sigma^*(\Imm \Hom_\cC(-,f)))
	=\widetilde{\rho}( \Imm \Hom_\cC(-,f))
	=\rho( f).
	\end{equation}
	Hence, $\rho$ is a rank function on $\cC$ and the mapping $\widetilde{\rho} \mapsto \rho$ is well-defined and injective.
	
	In order to prove that the correspondence is surjective, consider a rank function $\rho$ on $\cC$. We want to define an additive function $\widetilde{\rho}$ on $\mod \cc$ such that $\rho(f)=\widetilde{\rho}(\Imm \Hom_\cC(-,f))$ for $f$ in $\cC$. Let $F$ be in $\mod \cC$ and set $\widetilde{\rho}(F)\coloneqq\rho(f)$, where $f$ is such that $F \cong \Imm \Hom_{\cC}(-,f)$. Then $\widetilde{\rho}(F)$ is well-defined, since any two morphisms $f$ and $f'$ in $\cC$ with $\Imm \Hom_{\cC}(-,f) \cong \Imm \Hom_{\cC}(-,f')$ satisfy the identity $\rho(f)=\rho(f')$ by Lemma~\ref{lem:PropPreRankFunctions}. Moreover, an identity analogous to \eqref{eq:qTransInv} guarantees that $\widetilde{\rho}$ is $\Sigma$-invariant. It remains to show that $\widetilde{\rho}(F)=\widetilde{\rho}(F_1)+\widetilde{\rho}(F_2)$ whenever $F$ is the extension of  $F_2$ by $F_1$ in $\mod \cC$. Consider such $F$, $F_1$ and $F_2$ in $\mod \cC$. Using the Frobenius structure on $\mod \cC$ (see Subsection \ref{subsection RF1}), there are morphisms $f_1:X_1\to Y_1$ and $f_2:X_2\to Y_2$ in $\cC$ such that $F_1\cong \Imm \Hom_{\cC}(-,f_1)$ and $F_2\cong \Imm \Hom_{\cC}(-,f_2)$. We can complete the following commutative diagram using the Horseshoe lemma
	\[
	\begin{tikzcd}
	&0 \ar[d]& \\
	\Hom_\cC (-,X_1) \ar[d,  "\iota^\dagger"']\ar[r, two heads]  & F_1 \ar[r, hook]\ar[d] & \Hom_\cC (-,Y_1)\ar[d,  "\kappa^\dagger"]\\
	\Hom_\cC (-,X_1 \oplus X_2) \ar[r, , dashed, two heads] \ar[d,  "\pi^\dagger"'] & F \ar[r, dashed,hook] \ar[d]& \Hom_\cC (-,Y_1 \oplus Y_2). \ar[d, "\nu^\dagger"] \\
	\Hom_\cC (-,X_2)  \ar[r, two heads]  & F_2 \ar[r, hook] \ar[d] & \Hom_\cC (-,Y_2) \\
	&0&
	\end{tikzcd}
	\]
	Here the morphisms $\iota^\dagger,\pi^\dagger$ and $\kappa^\dagger, \nu^\dagger$ come from split triangles in $\cc$ and the top and the bottom rows are induced from $f_1$ and $f_2$. By the Yoneda lemma, the middle row of the diagram above stems from a morphism $f$ in $\cC$ and we have a commutative diagram whose columns split
	\begin{equation}
	\label{eq:two_squares}
	\begin{tikzcd}
	X_1 \ar[r, "f_1"] \ar[d,"\iota"] & Y_1  \ar[d,"\kappa"]\\
	X_1 \oplus X_2 \ar[r, "f"] \ar[d,"\pi"] & Y_1 \oplus Y_2.\ar[d,"\nu"] \\
	X_2 \ar[r, "f_2"] & Y_2 
	\end{tikzcd}
    \end{equation}
	Note that $f_2=\nu \circ f \circ \iota'$, where $\iota'$ is a section of $\pi$, so $f_2$ is the unique morphism making the bottom square of \eqref{eq:two_squares} commute. By the triangulated version of the $3\times3$ lemma (see \cite[Lemma 2.6]{May2001}), the diagram \eqref{eq:two_squares} can be completed to a diagram with triangles as rows and columns, as partially depicted below:
	\[
	\begin{tikzcd}
	X_1 \ar[r, "f_1"] \ar[d,"\iota"] & Y_1  \ar[d,"\kappa"] \ar[r] & Z_1 \ar[d, dashed, "\exists\,\gamma"]\ar[r] & \Sigma X_1 \ar[d,"\Sigma \iota"] \\
	X_1 \oplus X_2 \ar[r, "f"] \ar[d,"\pi"] & Y_1 \oplus Y_2 \ar[r]  \ar[d,"\nu"]& Z \ar[d, dashed,"\exists\,\delta"] \ar[r] & \Sigma (X_1 \oplus X_2). \ar[d,"\Sigma \pi"]\\
	X_2 \ar[r,  "f_2"] & Y_2 \ar[r, dashed] & Z_2 \ar[r, dashed] & \Sigma X_2
	\end{tikzcd}
	\]
	This induces a commutative diagram in $\mod \cC$
	\begin{equation}
	\label{eq:3x5diagram}
	\begin{tikzcd}
	[column sep=small]
	&0 \ar[d]& 0 \ar[d]& 0 \ar[d]& 0 \ar[d]& 0 \ar[d] &\\
	0 \ar[r] &F_1 \ar[d] \ar[r]& \Hom_\cC(-,Y_1) \ar[d]\ar[r]& \Hom_\cC(-,Z_1)\ar[d, "\gamma^\dagger"]\ar[r]& \Hom_\cC(-,\Sigma X_1)\ar[d]\ar[r] & \Sigma^*F_1 \ar[d]\ar[r]&  0\\
	0 \ar[r] &F \ar[d] \ar[r]& \Hom_\cC(-,Y_1\oplus Y_2) \ar[d]\ar[r]& \Hom_\cC(-,Z)\ar[d, "\delta^\dagger"]\ar[r]& \Hom_\cC(-,\Sigma (X_1\oplus X_2))\ar[d]\ar[r] & \Sigma^*F \ar[d]\ar[r]&  0,\\
	0 \ar[r] &F_2 \ar[d] \ar[r]& \Hom_\cC(-,Y_2) \ar[d]\ar[r]& \Hom_\cC(-,Z_2)\ar[d]\ar[r]& \Hom_\cC(-,\Sigma X_2)\ar[d]\ar[r] & \Sigma^*F_2 \ar[d]\ar[r]&  0\\
	&0&0&0&0&0&
	\end{tikzcd}
	\end{equation}
	where the rows and all columns, except possibly the middle one, are exact. We claim that the central column is exact too. For that, we apply the $3 \times 3$ lemma three times. Consider the two leftmost columns in \eqref{eq:3x5diagram} and take the cokernels of the horizontal monomorphisms --  in this way we obtain a third column which must be a short exact sequence by the $3 \times 3$ lemma. A similar strategy is applied to the two rightmost columns: we consider the kernels of the horizontal epimorphisms and get a third column which is exact. It is then possible to enclose the middle column of \eqref{eq:3x5diagram} between two exact columns, so that we obtain a $3\times 3$ diagram whose rows and outermost columns are exact. The central column of \eqref{eq:3x5diagram} is therefore exact, since $\delta^\dagger\circ \gamma^\dagger=0$. In particular, it is split exact, since all objects involved are projective-injective in $\mod \cC$. Thus, the triangle 
	\[
	\begin{tikzcd}
 Z_1 \ar[r, "\gamma"] & Z \ar[r, "\delta"] & Z_2 \ar[r] & \Sigma Z_1
	\end{tikzcd}
	\]
	is split and $Z\cong Z_1 \oplus Z_2$. As a result, we get
	\begin{align}
	\label{eq:additiveToRank}
	2(\widetilde{\rho}(F_1) - \widetilde{\rho}(F) + \widetilde{\rho}(F_2)) = &  2(\rho(f_1)-\rho(f)+\rho(f_2))\\
	=&(\rho(f_1)+\rho(\Sigma f_1))-(\rho(f)+\rho(\Sigma f))+(\rho(f_2)+\rho(\Sigma f_2))\nonumber\\
	=& \begin{multlined}[t]
	(\rho(\id_{Y_1})-\rho(\id_{Z_1})+\rho(\id_{\Sigma X_1})) -(\rho(\id_{Y_1\oplus Y_2})-\rho(\id_{Z_1\oplus Z_2})+\rho(\id_{\Sigma (X_1\oplus X_2)}))  \\+ (\rho(\id_{Y_2})-\rho(\id_{Z_2})+\rho(\id_{\Sigma X_2})) \end{multlined}\nonumber\\
	=& 0.\nonumber
	\end{align}
	Here, the second equality follows from (M4), the third is a consequence of Lemma \ref{lem:PropPreRankFunctions} and the last equality results from (M2). This finishes the proof.
\end{proof}

\begin{remark}
\label{rmk:inequalities}
 The inequalities $\rho(f\circ g) \leq \rho(f)$ and $\rho(f\circ g) \leq \rho(g)$, proved in \cite{Chuang2021}, follow from Theorem \ref{thm:TransInvariantAddFunctions<->rankFunctions} noting that $\Imm \Hom_\cc(-,f\circ g)$ is a subobject of $\Imm\Hom_\cc(-,f)$ and a quotient of $\Imm\Hom_\cc(-,g)$. Likewise the inequalities $\rho(f)+\rho(g)\leq \rho\begin{pmatrix}f & h \\ 0& g\end{pmatrix}$ and $\rho(f+g)\leq\rho(f)+\rho(g)$ in \cite{Chuang2021} follow from Theorem \ref{thm:TransInvariantAddFunctions<->rankFunctions} and~(M2). For the first inequality, note that $F\coloneqq\Imm \Hom_\cc(-,\begin{pmatrix}f\\ 0\end{pmatrix})$ is a subobject of $H\coloneqq\Imm \Hom_\cc(-,\begin{pmatrix}f & h\\ 0 &g\end{pmatrix})$ and that $G\coloneqq\Imm \Hom_\cc(-,\begin{pmatrix}0& g\end{pmatrix})$ is a quotient of $H$. Let $C$ be the cokernel of the canonical monomorphism $\iota: F\to H$. Note that $C$ lies in $\mod \cc$ and
 \[
 \rho\begin{pmatrix}f & h \\ 0& g\end{pmatrix}=\widetilde{\rho}(H)=\widetilde{\rho}(F)+\widetilde{\rho}(C)=\rho \begin{pmatrix}f\\ 0\end{pmatrix}+\widetilde{\rho}(C)=\rho(f)+\widetilde{\rho}(C).
 \]
 It is not hard to check that the canonical epimorphism $\pi: H \to G$ is such that $\pi\circ \iota =0$, so $\pi$ must factor through the cokernel of $\iota$. Hence $G$ is a quotient of $C$ and $\widetilde{\rho}(C)\geq \widetilde{\rho}(G)=\rho(g)$, which proves the desired inequality. For the inequality $\rho(f+g)\leq\rho(f)+\rho(g)$, notice that $\Imm\Hom_\cc (-,f+g)$ is a subquotient of $\Imm\Hom_\cc (-,f\oplus g)$.
 
\end{remark}

\begin{remark} 
There is a result due to Krause 
similar to Theorem \ref{thm:TransInvariantAddFunctions<->rankFunctions}, but concerning slightly different classes of functions (see \cite{Krause2016}). One side of the bijective correspondence there is given by integral additive functions $\widetilde{\rho}$ on $\mod \cC$ satisfying the following condition, instead of $\Sigma$-invariance: for each $F$ in $\mod \cC$, there exists $n \in \mathbb{Z}$ such that $\widetilde{\rho}((\Sigma^*)^nF) = 0.$ On the other side of the bijection, instead of rank functions, Krause considers functions for which (O3) and (O4) in Definition \ref{def:rankfunctionobj} are replaced by other suitable conditions, and which he calls cohomological functions. A typical example of a cohomological function is given by $Y \mapsto \dim_k\Hom_{\operatorname{D^b}(A)}(X,Y)$ for $Y$ in $\operatorname{D^b}(A)$ and some fixed object $X$ in $\cper (A)$, where $A$ is a finite-dimensional algebra over a field $k$ (compare with Example \ref{ex:fin-dim_alg}). Results similar to \cite[Theorem 1.4.(3)]{Krause2016} will be discussed in the next section.
\end{remark}

\section{Integral rank functions}
\label{sec:Integral}
 We apply the general theory from Section \ref{sec:preliminaries} and the conclusions from Section \ref{sec:RankFunctions} to study integral rank functions, which will be our main focus for the rest of the paper. 
 
 Assume, as before, that $\cC$ is a skeletally small triangulated category.
 We call a rank function $\rho$ on $\cC$  \emphbf{integral} if $\rho(f)\in \Z$ for every morphism $f$ in $\cc$. Note that similarly to the case of additive functions, we can define a \emphbf{locally finite sum}  of rank functions $(\rho_i)_{i\in I}$ whenever for each $f$ in $\cC$ only finitely many $\rho_i(f)$ are non-zero. Observe that the class of integral rank functions is  closed under locally finite sums.
Given a rank function $\rho$ on $\cC$, we shall 
denote its restriction to objects by $\rho_{ob}$ and its lift to $\mod \cc$ by $\widetilde{\rho}$. We say that an additive function is \emphbf{$\Sigma$-irreducible} if it is $\Sigma$-invariant, integral, non-zero, and cannot be expressed as a sum of two non-zero $\Sigma$-invariant integral additive functions. Similarly, a rank function is said to be \emphbf{irreducible} if it is integral, non-zero, and cannot be written as a sum of two non-zero integral rank functions.

\begin{remark}
	\label{rmk:IntegralRankFunc<->SigmaInvAddIntFunc}
	The bijection in Theorem \ref{thm:TransInvariantAddFunctions<->rankFunctions} restricts to a bijection between ($\Sigma$-irreducible) $\Sigma$-invariant integral additive functions on $\mod \cC$ and (irreducible) integral rank functions on $\cC$.
\end{remark}

\subsection{Decomposition theorem and further bijections}

The goal of this subsection is to prove a decomposition theorem for integral rank functions and to connect rank functions to cohomological functors and Serre subcategories of $\mod \cC$.

Given two integral additive functions $\widetilde{\sigma}$ and $\widetilde{\tau}$ on $\mod \cC$, we say that $\widetilde{\sigma}$ is \emphbf{$\Sigma$-equivalent} to $\widetilde{\tau}$ if $\widetilde{\sigma} = \widetilde{\tau}^{ {(\Sigma^*)}^n}$ for some $n\in \Z$. This defines an equivalence relation on the set of all integral additive functions on $\mod \cC$. 
	The \emphbf{$\Sigma$-orbit} $O_{\widetilde{\rho}}$ of an integral additive function $\widetilde{\rho}$ on $\mod \cC$ is the equivalence class of integral additive functions  $\Sigma$-equivalent to $\widetilde{\rho}$, i.e.~$O_{\widetilde{\rho}}=\{\widetilde{\rho}^{(\Sigma^*)^n} \mid n \in \Z\}$.
 
Note that, by Lemma \ref{lem:AddFuncPrecomposedFunctor}, an additive function $\widetilde{\rho}$ on $\mod \cC$ is irreducible if and only if $\widetilde{\rho}^{\Sigma^*}$ is irreducible, as $\Sigma^*$ and $(\Sigma^*)^{-1}$ are essentially surjective. In particular, if $\widetilde{\rho}$ is irreducible, then so is every function in $O_{\widetilde{\rho}}$.

\begin{theorem}
\label{thm:decomposition}
	Let $\cC$ be a skeletally small triangulated category. The following statements hold:
	\begin{enumerate}
		\item every $\Sigma$-invariant integral additive function on $\mod \cC$ can be expressed in a unique way as a locally finite sum of $\Sigma$-irreducible additive functions on $\mod \cC$;
		\item every integral rank function on $\cC$ can be decomposed in a unique way as a locally finite sum of irreducible rank functions on $\cC$.
	\end{enumerate} 
\end{theorem}

\begin{proof}
	In order to prove (1), consider a $\Sigma$-invariant integral additive function $\widetilde{\rho}$. According to Theorem \ref{thm:IrrAdditiveFunctions<->InjFinEndolength}, $\widetilde{\rho}$ decomposes as a locally finite sum of irreducible additive functions $\widetilde{\rho}=\sum_{i\in I} \widetilde{\rho}_i$ in an essentially unique way. Observe that
	\[
	\sum_{i\in I} \widetilde{\rho}_i=\widetilde{\rho}=\widetilde{\rho}^{\Sigma^*}=\sum_{i\in I} {\widetilde{\rho}_i}^{\Sigma^*}.
	\]
	Since the additive functions ${\widetilde{\rho}_i}^{\Sigma^*}$ are irreducible, there exists a bijection $\varphi:I\to I$ such that ${\widetilde{\rho}_i}^{\Sigma^*}=\widetilde{\rho}_{\varphi (i)}$ for every $i\in I$. Fix $j\in I$. Since ${\widetilde{\rho}_{j}}^{{(\Sigma^*)}^n}=\widetilde{\rho}_{\varphi^n (j)}$ for all $n\in\Z$, every function in the $\Sigma$-orbit $O_{\widetilde{\rho}_{j}}$ appears as a summand in the decomposition $\widetilde{\rho}=\sum_{i\in I} \widetilde{\rho}_i$. Hence $\widetilde{\rho}$ can be expressed as $\widetilde{\rho}=\widetilde{\sigma}_1+\widetilde{\sigma}_2$ with $\widetilde{\sigma}_1=\sum_{\widetilde{\tau}\in O_{\widetilde{\rho}_{j}}} \widetilde{\tau}$, and this decomposition of $\widetilde{\sigma}_1$ is locally finite. Since both $\widetilde{\rho}$ and $\widetilde{\sigma}_1$ are $\Sigma$-invariant, so is $\widetilde{\sigma}_2$.  Clearly, the function $\widetilde{\sigma}_1$ is $\Sigma$-irreducible. In case $\widetilde{\rho}$ is $\Sigma$-irreducible, $\widetilde{\sigma}_2$ must be zero. For the general case, we have that $\{\widetilde{\rho}_i\}_{i\in I}=\bigcup_{i\in I} O_{\widetilde{\rho}_i}$ and one can use the axiom of choice to decompose this set as a disjoint union $\coprod_{k\in K}O_{\widetilde{\rho}_{k}}$ with $k\in K\subseteq I$. As a result, we get that 
	\[\widetilde{\rho}=\sum_{k\in K}\sum_{\substack{i\in I\\ \widetilde{\rho}_i=\widetilde{\rho}_k}}\sum_{\widetilde{\tau}\in O_{\widetilde{\rho}_{k}}} \widetilde{\tau}.\]
	This proves part (1). Part (2) follows from (1) using Theorem \ref{thm:TransInvariantAddFunctions<->rankFunctions} (see also Remark \ref{rmk:IntegralRankFunc<->SigmaInvAddIntFunc}).
\end{proof}

In order to prove our next result 
we need some further terminology.
We say that a rank function is \emphbf{basic} if it is integral and the summands appearing in its decomposition as a locally finite sum of irreducible rank functions, as in Theorem \ref{thm:decomposition}, are pairwise distinct. Moreover, a functor $H$ in $\Mod \cC$ is said to be \emphbf{$\Sigma$-invariant} if $\Sigma^* H\cong H$, i.e. $H\circ \Sigma$ is naturally isomorphic to $H$. Finally, any class $\mathcal{Y}$ of objects (resp.~morphisms) in $\cC$ shall be called \emphbf{$\Sigma$-closed} if $\mathcal{Y}=\{\Sigma Y\mid Y\text{ in }\mathcal{Y} \}$ (resp.~$\mathcal{Y}=\{\Sigma f\mid f\text{ in }\mathcal{Y} \}$). By abuse of notation, any class of objects or morphisms in $\Mod \cc$ which is mapped onto itself by  $\Sigma^*$ will also be called $\Sigma$-closed. Since the classes we will be considering are closed under isomorphisms, this amounts to requiring that the action of $\Sigma$ or $\Sigma^*$ and of the respective quasi-inverses stays within the class. Similarly, we can consider \emphbf{$\Sigma$-orbits} of objects in this situation.

\begin{theorem}
\label{thm:LongBijRankFunc}
Let $\cC$ be a skeletally small triangulated category. There is a bijective correspondence between:
\begin{enumerate}[(a)]
	\item basic rank functions on $\cC$;
	\item isoclasses of $\Sigma$-invariant basic endofinite cohomological functors in $\Mod \cC$;
	\item $\Sigma$-closed Serre subcategories $\cS$ of $\mod \cc$ such that the quotient $(\mod \cC)/\cS$ is a length category;
	\item $\Sigma$-closed hereditary torsion subcategories of finite type $\cx$ of $\Mod \cC$ with $(\Mod \cC)/\cx$ locally finite;
	\item $\Sigma$-closed closed discrete subsets of $\Zg (\Mod \cC)$ satisfying the isolation property.
\end{enumerate}
\end{theorem}

\begin{proof}
By Theorems \ref{thm:TransInvariantAddFunctions<->rankFunctions} and \ref{thm:decomposition}, basic rank functions on $\cC$ are in bijection with basic $\Sigma$-invariant additive functions on $\mod \cC$. Recall that, by Corollary \ref{cor:CohEndofin=InjFinEndoLength}, endofinite cohomological functors are precisely the finite endolength injectives in $\Mod \cC$. By Theorem \ref{thm:lastPreliminaries}, a basic endofinite cohomological functor $H$ corresponds to a unique basic additive function $\widetilde{\rho}_{H}$. Therefore,  $\Sigma^* H \cong H$ if and only if $\widetilde{\rho}_H = \widetilde{\rho}_{\Sigma^*H}.$ By combining this with the identity $\widetilde{\rho}_H=(\widetilde{\rho}_{\Sigma^*H})^{\Sigma^*}$, which follows from the fact that $\Sigma^*$ is an equivalence on $\Mod \cC$, we obtain the bijection between (a) and (b).
The remaining bijections follow from Theorem \ref{thm:lastPreliminaries} and the assignments therein.
\end{proof}

To get the bijections between (a)-(e) explicitly, one can use the assignments from Theorems \ref{thm:TransInvariantAddFunctions<->rankFunctions} and \ref{thm:lastPreliminaries}. This yields the following chain of assignments \[\rho \mapsto \overrightarrow{\Ker\widetilde{\rho}}, \text{ } \cx \mapsto \cx \cap \fp \cA, \text{ } \cS \mapsto \Zg \ca \setminus\co (\cS), \text{ } C \mapsto \coprod_{[H]\in C} H, \text{ } H \mapsto \rho_H\] going from (a) to (d), to (c), to (e), to (b), and back to (a). Since irreducible rank functions are in bijection with $\Sigma$-irreducible additive functions, the correspondence between (a) and (b) in Theorem \ref{thm:LongBijRankFunc} restricts to a bijection between  irreducible rank functions on $\cc$ and $\Sigma$-indecomposable $\Sigma$-invariant endofinite cohomological functors in $\Mod \cC$. Here, we call a non-zero $\Sigma$-invariant endofinite cohomological functor \emphbf{$\Sigma$-indecomposable} if it can not be decomposed as a direct sum of two non-zero cohomological functors of the same kind.
The counterpart of (e) are then closed discrete $\Sigma$-orbits of  points in the Ziegler spectrum satisfying the isolation property, or equivalently, $\Sigma$-orbits of endofinite objects in the Ziegler spectrum such that the coproduct over the orbit is still endofinite (see Remark \ref{IsolatedRemark}).
The corresponding restriction of (c) are $\Sigma$-closed Serre subcategories $\cS$ of $\mod \cC$ such that  $(\mod \cC)/\cS$ is a length category with one simple object up to translation (induced from $\Mod \cc$). These correspond to locally finite quotients $(\Mod \cC)/\overrightarrow{S}$ with one indecomposable injective object up to translation.
Note that the translation $\overline{\Sigma^*}$ on $(\Mod\cc)/\overrightarrow{S}$ induced from $\Sigma^*$ on $\Mod\cc$ is still an equivalence of categories and $\overline{\Sigma^*}\circ E\cong E\circ \Sigma^*$ and $\Sigma^* \circ R \cong R \circ \overline{\Sigma^*}$ for $E$ and $R$ as in \eqref{eq:loc}. Therefore, the last part follows, for example, from the fact that the homeomorphism $\Zg (\ca/ \overrightarrow{S})\to \Zg \cA \setminus \co (\cS)$ induced by $R$ commutes with the translations and from the bijection between simple objects in $\fp\cA$ and indecomposable injectives in $\cA$ for a locally finite category~$\cA$.

Skeletally small triangulated categories $\cc$ whose module category $\Mod\cc$ is locally finite will play an important role in our study of rank functions. Such triangulated categories shall also be called \emphbf{locally finite} by abuse of terminology. Here is the simplified statement of Theorem \ref{thm:LongBijRankFunc} in that case:

\begin{corollary}
\label{thm:LocFinRankFunc}
If the triangulated category $\cC$ is locally finite, there is a bijective correspondence between:
\begin{enumerate}[(a)]
	\item basic rank functions on $\cC$;
	\item isoclasses of $\Sigma$-invariant basic cohomological functors in $\Mod \cC$;
	\item $\Sigma$-closed Serre subcategories of $\mod \cc$;
	\item $\Sigma$-closed hereditary torsion subcategories of finite type of $\Mod \cc$;
	\item $\Sigma$-closed subsets of $\Zg (\Mod \cC)$.
\end{enumerate}
\end{corollary}

\begin{proof}
Recall that cohomological functors are precisely the fp-injective objects in $\Mod \cC$ by Lemma \ref{lem:cohomologicalFpinjectiveFlat}. If $\cC$, and thus $\Mod \cC$, is locally finite, these are in turn the injective objects in $\Mod \cC$, which all have finite endolength by Lemma \ref{lem:injFinEndolength}. Now, any coproduct of  injectives in $\Mod\cC$ is an injective of finite endolength, so any $\Sigma$-closed subset of $\Zg (\Mod \cC)$ appears as a subset from (e) in Theorem \ref{thm:LongBijRankFunc} (see Remark \ref{IsolatedRemark}). The remaining bijections follow from Theorem \ref{thm:ZieglerSpecLocallyCohCat} and the fact that $\Zg (\Mod \cC)$ is discrete in the locally finite case (see Theorem~\ref{thm:lastPreliminaries} or Lemma \ref{lem:IsolationLocFin}). 
\end{proof}
\noindent Note that if $\cC$ is not only locally finite but also idempotent complete, the points of the Ziegler spectrum correspond to indecomposable representable functors (see \cite[Theorem 2.1]{Krause2012}), so they are parameterised by the isoclasses of indecomposable objects in $\cC$.

A triangulated category $\cc$ is called \emphbf{periodic} if $\Sigma^d\cong \id_{\cc}$ for some $d\neq 0$. Periodic triangulated categories include homotopy categories of matrix factorisations, periodic derived categories, and certain cluster categories. 
For periodic triangulated categories the $\Sigma$-orbits of endofinite cohomological functors have finite cardinality, so the coproduct of objects appearing in a $\Sigma$-orbit of an indecomposable endofinite cohomological functor remains endofinite. Using that, the statement of Theorem \ref{thm:LongBijRankFunc} can be simplified as well. In particular, we get the following corollary:

\begin{corollary} \label{cor:periodic}
If the triangulated category $\cC$ is periodic, there is a bijective correspondence between:
\begin{enumerate}[(a)]
	\item irreducible rank functions on $\cC$;
	\item $\Sigma$-orbits of indecomposable endofinite cohomological functors in $\Mod \cC$.
\end{enumerate}
\end{corollary}

\subsection{Exact and localising rank functions}

The main goal of this subsection is to extend the bijection in \cite[Theorem 6.4]{Chuang2021}. Before recalling that result, a few definitions are needed.

Given a triangulated functor $\Phi:\cc \to \cE$ between skeletally small triangulated categories and a rank function $\rho$ on $\cE$, one may define a new rank function $\rho^\Phi$ on $\cc$ via $\rho^\Phi(f)=\rho (\Phi (f))$. On the level of the module category, this corresponds, via Theorem \ref{thm:TransInvariantAddFunctions<->rankFunctions}, to the ($\Sigma$-invariant) additive function ${\widetilde{\rho}}^{\Phi^*}$. 

\begin{definition}[{\cite[\S 2, Definition 6.2]{Chuang2021}}] A rank function $\rho$ on $\cc$ is:
\begin{enumerate}
	\item \emphbf{localising} if it is integral and $\rho(f)=0$ implies that the morphism $f$ factors through an object $K$ satisfying $\rho(\id_K)=\rho_{ob}(K)=0$;
	\item \emphbf{prime} if it is integral and $\cc$ has a generator $X$ such that $\rho_{ob}(X)=1$.
\end{enumerate} 
\end{definition}

\noindent Here, by a \emphbf{generator} of $\cc$ we mean an object $X$ such that $\cC$ coincides with its smallest thick subcategory $\langle X \rangle$ containing $X$. A triangulated category is said to be \emphbf{simple} if it has an indecomposable generator and every triangle is a (finite) direct sum of split triangles. By \cite[Corollary 2.21]{Chuang2021}, such categories admit a unique prime rank function.

For a rank function $\rho$ on $\cC$ one can define its \emphbf{kernel on objects} 
\[
\Ker \rho_{ob} \coloneqq\{X \text{ in }\cC \mid \rho_{ob}(X) = 0\}.
\]
The subcategory $\Ker \rho_{ob}$ is clearly a thick subcategory of $\cC$ by the definition of a rank function. Rank functions with trivial kernel on objects are called \emphbf{object-faithful}. Every rank function $\rho$ descends to an object-faithful rank function $\overline{\rho}$ on $\cC/\Ker \rho_{ob}$ satisfying ${\overline{\rho}}^\kappa=\rho$, where $\kappa:\cc \to \cc/\Ker \rho_{ob}$ is the Verdier localisation functor (see \cite[Corollary 4.4]{Chuang2021}).  

\begin{theorem}[{\cite[Theorem 6.4]{Chuang2021}}]
	\label{thm:localising_CL}
Let $\cc$ be a skeletally small triangulated category with a generator. The correspondence $\rho \mapsto \Ker \rho_{ob}$ induces a bijection between:
\begin{enumerate}[(a)]
    \item localising prime rank functions on $\cc$;
 \item thick subcategories $\CK$ of $\cc$ with simple Verdier quotient $\cc/\CK$.
 \end{enumerate}
\end{theorem}
\noindent The rank functions appearing in Theorem \ref{thm:localising_CL} are always irreducible:
\begin{lemma}
	Suppose that $\cc$ has a generator. Then every prime rank function on $\cc$ is irreducible.
\end{lemma}
\begin{proof}
	Let $X$ be a generator of $\cc$ and let $\rho$ be integral with $\rho(\id_X)=1$. Write $\rho=\sigma+\tau$ with $\sigma$ and $\tau$ integral rank functions. Since $\rho(\id_X)=1$, then either ${\sigma}(\id_X)=0$ or ${\tau}(\id_X)=0$. If ${\sigma}(\id_X)=0$, then $\langle X \rangle=\cC$ is contained in the thick subcategory $\Ker {\sigma}_{ob}$. Hence $\sigma=0$.
\end{proof}

We shall extend the bijection from Theorem \ref{thm:localising_CL} to all localising basic rank functions on arbitrary skeletally small triangulated categories. In fact, this extended bijection will follow as a corollary of a more general correspondence. To state and prove this correspondence more background is needed.
The \emphbf{annihilator} $\Ann \pi$ of an additive functor $\pi:\cc \to \cE$ is the class of all morphisms in $\cc$ that are mapped to zero by $\pi$. Note that $\Ann \pi$ is an additive ideal and it is $\Sigma$-closed whenever $\pi$ is a triangulated functor. Using annihilators one can define various types of ideals in triangulated categories. Following \cite{Krause2000}, we say that an ideal in $\cC$ is \emphbf{cohomological} if it has the form $\Ann F$ for some cohomological functor $F: \cC \to \cA$ to an abelian category $\cA$. By \emphbf{exact} we mean any ideal of the form $\Ann \pi$ for a cohomological quotient functor $\pi:\cC \to \cE$ (see the definition below).

\begin{definition}[{\cite[Definitions 4.1, 7.1, 9.1]{Krause2005}}] A triangulated functor $\pi:\cc \to \cE$ between skeletally small triangulated categories is:
	\begin{enumerate}
		\item a \emphbf{cohomological quotient functor} if for every abelian category $\ca$ and every cohomological functor $H:\cc \to \ca$ satisfying $\Ann \pi \subseteq \Ann H$, there exists, up to a unique isomorphism, a unique cohomological functor $H': \cE \to \ca$ such that $H=H' \circ \pi$;
		\item an \emphbf{exact quotient functor} if for every triangulated category $\ce'$ and every triangulated functor $G:\cc \to \ce'$ satisfying $\Ann \pi \subseteq \Ann G$, there exists, up to a unique isomorphism, a unique triangulated functor $G': \cE \to \cE'$ such that $G=G' \circ \pi$;
		\item a \emphbf{CE-quotient functor} if it is both a cohomological quotient functor and an exact quotient functor.
	\end{enumerate}
\end{definition}

 On the class of CE-quotient functors starting in $\cc$ one can define an equivalence relation by saying that $\pi_1: \cc \to \cE_1$ and $\pi_2: \cc \to \cE_2$ are  \emphbf{equivalent} if there exists an equivalence $\Phi:\cE_1 \to \cE_2$ such that $\pi_2=\Phi \circ \pi_1$.
Verdier localisations are the standard examples of CE-quotient functors (see \cite[Examples 4.3, 7.2]{Krause2005}). Since equivalent CE-quotient functors have the same kernel, distinct thick subcategories of $\cc$ induce non-equivalent Verdier localisation functors.
The following bijection follows from \cite{Krause2005}:

\begin{theorem}
	\label{thm:consequence_coh_loc}
	Let $\cc$ be a skeletally small triangulated category. The assignment sending the functor $\pi:\cc \to \cE$ to $\Ker \pi^*|_{\mod \cc}$ induces a bijection between:
\begin{enumerate}[(a)]
	\item equivalence classes of CE-quotient functors starting in $\cC$;
	\item $\Sigma$-closed Serre subcategories $\cS$ of $\mod \cc$ for which the Serre localisation $\Mod \cc \to (\Mod \cc)/\overrightarrow{\cS}$ admits an exact right adjoint.
\end{enumerate}
This restricts to a one-to-one correspondence between equivalence classes of CE-quotient functors from $\cC$ to locally finite triangulated categories and $\Sigma$-closed Serre subcategories $\cS$ of $\mod \cc$ for which the Serre localisation $\Mod \cc \to (\Mod \cc)/\overrightarrow{\cS}$ admits an exact right adjoint and $(\Mod \cc)/\overrightarrow{\cS}$ is locally finite.
\end{theorem}
\begin{proof}
By \cite[Theorem 9.2]{Krause2005}, the assignment sending the functor $\pi:\cc \to \cE$ to $\Ann \pi$ induces a bijection between the set of equivalence classes of CE-quotient functors starting in $\cC$ and the set of exact ideals of $\cC$. By \cite[Lemma 8.4]{Krause2005}, cohomological ideals are in bijection with Serre subcategories of $\mod \cc$ via the assignment sending an ideal $I$ in $\cC$ to the subcategory  $\Imm(I)\coloneqq \{F\text{ in }\mod \cc \mid F\cong \Imm \Hom_\cC(-,f) \text{ for some } f\in I\}$. By \cite[Proposition 8.8]{Krause2005}, for a $\Sigma$-closed ideal $I$, being exact is equivalent to the fact that $\Mod \cc \to (\Mod \cc)/\overrightarrow{\Imm(I)}$ admits an exact right adjoint. Combining these results with the fact that $\Imm (\Ann \pi)=\Ker \pi^*|_{\mod \cc}$, we obtain the bijection between (a) and (b) with the desired assignment. 
	 
  Let us check how this bijection restricts in the locally finite case.
	 Let $\pi:\cc \to \cE$ be a CE-quotient functor. For convenience, write $\cS_\pi$ instead of $\Ker \pi^*|_{\mod \cc}$. According to \cite[Theorem 4.4]{Krause2005}, the functor $\pi^*|_{\mod \cc}: \mod\cc \to \mod \cE$ can be identified with the Serre localisation $\mod \cc \to (\mod \cc)/\cS_\pi$. The localisation $\mod \cc \to (\mod \cc)/\cS_\pi$ corresponds to the restriction $\overline{\pi}|_{\mod \cc}: \mod \cc \to \fp((\Mod \cc)/\overrightarrow{\cS_\pi})$ of the Serre localisation $\overline{\pi}:\Mod \cc \to (\Mod \cc)/\overrightarrow{\cS_\pi}$. 
	 In particular, $\fp((\Mod \cc)/\overrightarrow{\cS_\pi})\simeq \mod \cE$, so $(\Mod \cc)/\overrightarrow{\cS_\pi}$ is locally finite exactly when $\cE$ is a locally finite triangulated category. 
\end{proof}

Given a rank function $\rho$, denote the ideal $\{f\text{ in }\cc \mid \rho(f)=0\}$ by $\Ker \rho$ and call this the \emphbf{kernel on morphisms} of $\rho$. By \cite[Remark 2.7, Lemma 4.6]{Chuang2021}, $\Ker \rho$ is an additive ($\Sigma$-closed) ideal (see also Remark \ref{rmk:inequalities}). In fact, this ideal is cohomological.

\begin{lemma}
\label{lem:kernel-cohom_ideal}
Let $\rho$ be a rank function on $\cC$. Then $\Ker \rho$ is a $\Sigma$-closed cohomological ideal in $\cC$.
\end{lemma}

\begin{proof}
By the bijection in Theorem \ref{thm:TransInvariantAddFunctions<->rankFunctions}, to any rank function $\rho$ we can assign a $\Sigma$-invariant additive function $\widetilde{\rho}$ on $\mod \cC$. The functor $F$ 
\[
F: \cC \to  \mod \cC \to (\mod \cC) / \Ker \widetilde{\rho}
\]
given by the composition of the Yoneda embedding with the Serre localisation functor is cohomological. We get that $\Ker \rho = \Ann F$, so $\Ker \rho$ is a $\Sigma$-closed cohomological ideal.
\end{proof}
\begin{definition}
\label{definition_exact}
	A rank function $\rho$ on $\cc$ is:
	\begin{enumerate}
		\item \emphbf{exact} if it is integral and the Serre localisation functor $\Mod \cc \to (\Mod \cc)/\overrightarrow{\Ker\widetilde{\rho}}$ has an exact right adjoint;
		\item \emphbf{idempotent} if it is integral and the ideal $\Ker \rho$ is idempotent, meaning that every $f$ in $\Ker \rho$ is of the form $g\circ h$ for some $g$ and $h$ in $\Ker \rho$.
	\end{enumerate} 
\end{definition} 

Note that we could have defined exact rank functions as integral rank functions $\rho$ whose kernel on morphisms $\Ker \rho$ is an exact ideal. The equivalence of these two definitions follows from \cite[Proposition 8.8]{Krause2005}. Note that each exact ideal is cohomological (\cite[\S 2]{Krause2005}). The notion of an idempotent rank function will play a more prominent role in the next section, where we restrict ourselves to the context of compact objects in a compactly generated triangulated category. In this setting, exact and idempotent rank functions coincide. In general, we can observe that every localising rank function is both exact and idempotent:

\begin{proposition} \label{prop: kernels}
Let $\rho$ be an integral rank function on $\cC$ and let $\widetilde{\rho}$ be its lift to $\mod \cC$. Denote by $\kappa$ the Verdier localisation $\kappa: \cC \to \cC/\Ker\rho_{ob}$. Then we have:
\begin{enumerate}
    \item  $\rho$ is localising if and only if $\Ker\kappa^*|_{\mod \cc}  = \Ker\widetilde{\rho}$.
    \item if $\rho$ is localising, it is both exact and idempotent.
\end{enumerate}
\end{proposition}
\begin{proof}
Let $\rho$ be a rank function on $\cc$. On the one hand, a morphism $f$ factors through an object in $\Ker\rho_{ob}$ if and only if $f$ belongs to $\Ann \kappa$, and this happens if and only if $\Imm\Hom_{\cc}(-,f)$ is in $\Ker \kappa^*|_{\mod \cc}$. On the other hand, $f$ lies in $\Ker \rho$ if and only if $\Imm\Hom_{\cc}(-,f)$ is in $\Ker \widetilde{\rho}$. Finally, if $f$ factors through an object in $\Ker\rho_{ob}$, then $\rho(f)=\overline{\rho}^\kappa(f)=0$, so $f$ belongs to $\Ker \rho$. This proves that $\Ker \kappa^*|_{\mod \cc} \subseteq \Ker\widetilde{\rho}$. Part (1) then follows from  the definition of localising rank function.

To prove part (2), assume that $\rho$ is localising. By applying Theorem \ref{thm:consequence_coh_loc} to the CE-quotient functor $\kappa:\cc \to \cc/ \Ker\rho_{ob}$ and using part (1), we get that $\rho$ is exact. To prove that $\rho$ is idempotent, consider $f:X\to Y$ in $\Ker \rho$. There exist $h:X\to K$ and $g:K\to Y$ with $K$ in $\Ker\rho_{ob}$ and $f=g\circ h$. Since $\rho(g)=\overline{\rho}^{\kappa}(g)=0$ and $\rho(h)=\overline{\rho}^{\kappa}(h)=0$, the ideal $\Ker \rho$ is idempotent. 
\end{proof}

Suppose that $\cc$ is a locally finite triangulated category. Denote the composition length on $\mod \cc$ by $\widetilde{\ell_\cc}$. This yields a $\Sigma$-invariant additive function on $\mod \cc$, since the length of an object is preserved by $\Sigma^*$, as $\Sigma^*$ is an equivalence. Consequently, there exists a canonical rank function on $\cc$, denoted by $\ell_\cc$, which corresponds to $\widetilde{\ell_\cc}$ via the bijection in Theorem \ref{thm:TransInvariantAddFunctions<->rankFunctions}. In fact, since $\widetilde{\ell_\cc}$ is basic, so is $\ell_\cc$. For example, it is easy to check that a simple triangulated category $\cC$ is locally finite (see \cite[Proposition 2.3]{Krause2012}), and in that case, the map $\ell_\cc$ coincides with the unique prime rank function on $\cc$. For an arbitrary locally finite triangulated category $\cC$, we have $\ell_\cc(f)=0$ if and only if $\Imm\Hom_{\cc}(-,f)=0$, and this holds exactly when $f=0$. In other words, $\Ker \ell_\cc=0$. We call rank functions with trivial kernel on morphisms \emphbf{morphism-faithful}. Locally finite triangulated categories are characterised by the existence of an integral morphism-faithful rank function.

\begin{proposition}
\label{prop:charac_loc_fin_tria}
A skeletally small triangulated category $\cc$ is locally finite if and only if it admits a morphism-faithful basic rank function. In this case, such function is unique and is given by $\ell_\cc$.
\end{proposition}
\begin{proof}
	As mentioned previously, if $\cC$ is locally finite,  $\ell_\cc$ is morphism-faithful and basic. Now, the result follows from the construction of $\ell_\cc$ and from Corollary \ref{cor:charac_loc_finite_abelian}, recalling that $\cc$ is locally finite if and only if $\mod \cc$ is a length category.
\end{proof}

We are finally ready to prove the generalisation of Theorem \ref{thm:localising_CL}: 

\begin{theorem}
	\label{thm: FDT_generalised} 
	Let $\cc$ be a skeletally small triangulated category. The correspondence mapping $\pi:\cc \to \cE$ to $\ell_\cE^\pi$ induces a bijection between:
	\begin{enumerate}[(a)]
		\item equivalence classes of CE-quotient functors from $\cC$ to locally finite triangulated categories;
		\item exact basic rank functions on $\cc$.
	\end{enumerate}
Moreover,  the rank function appearing in (b) is localising if and only if the corresponding CE-quotient is a Verdier localisation. Hence, there is a bijection between thick subcategories $\CK$ of $\cC$ with a locally finite Verdier quotient $\cc/\CK$ and localising basic rank functions on $\cc$.
\end{theorem}
\begin{proof}
    By Theorem~\ref{thm:consequence_coh_loc}, the assignment sending $\pi$ to $\Ker \pi^*|_{\mod \cc}$ induces a bijective correspondence between (a) and $\Sigma$-closed Serre subcategories $\cS$ of $\mod \cc$ for which $(\Mod \cc)/\overrightarrow{\cS}$ is locally finite and the Serre localisation $\Mod \cc \to (\Mod \cc)/\overrightarrow{\cS}$ admits an exact right adjoint. By Theorem~\ref{thm:LongBijRankFunc}, these are in one-to-one correspondence with basic rank functions $\rho$ on $\cC$ such that the Serre localisation $\Mod \cc \to (\Mod \cc)/\overrightarrow{\Ker \widetilde{\rho}}$ admits an exact right adjoint. 
    Combining these two bijections, we obtain a unique rank function $\rho$ such that $\Ker \widetilde{\rho} = \Ker \pi^*|_{\mod \cc}$. By Definition~\ref{definition_exact}, the basic rank functions $\rho$ appearing in the last correspondence are precisely the exact basic rank functions, so we get the bijection between (a) and (b). As explained in the proof of Theorem \ref{thm:consequence_coh_loc}, $\pi^*|_{\mod \cc}$ is the Serre localisation with respect to $\Ker \pi^*|_{\mod \cc}$. Thus, the rank function in (b) corresponding to $\pi$ is $\ell_\cE^\pi$, since in Theorem \ref{thm:lastPreliminaries} the corresponding additive function is $(\widetilde{\ell_\cE})^{\pi^*}$ (see Remark \ref{pullback}).
		
	For the second statement: Let $\CK$ be a thick subcategory of $\cc$ with $\cc/\CK$ locally finite. Consider the corresponding Verdier localisation $\pi: \cc \to \cc/\CK$ and the rank function $\ell_{\cc/\CK}^\pi$ on $\cc$. Since $\ell_{\cc/\CK}$ is morphism-faithful, then $\Ker \ell_{\cc/\CK}^\pi=\Ann \pi$. Hence, every morphism in $\Ker \ell_{\cc/\CK}^\pi$ factors through an object in $\CK$ and $\ell_{\cc/\CK}^\pi$ is localising. Conversely, pick a localising basic rank function $\rho$. Consider the functor $\kappa:\cc\to \cc/\Ker \rho_{ob}$. By Proposition \ref{prop: kernels}, $\Ker\kappa^*|_{\mod\cc}=\Ker \widetilde{\rho}$. Using the bijection between (a) and (b), there exists a CE-quotient functor $\pi:\cc \to \cE$ with $\cE$ locally finite and $\rho=\ell_{\cE}^{\pi}$. As $\ell_{\cE}$ is morphism-faithful, we get $\Ann\pi=\Ker \rho$ and $\Ker \pi^*|_{\mod \cc}=\Ker \widetilde{\rho}$. By Theorem \ref{thm:consequence_coh_loc}, the equality $\Ker \pi^*|_{\mod \cc}=\Ker \kappa^*|_{\mod \cc}$ implies that $\kappa$ and $\pi$ are equivalent CE-quotients functors. So $\pi$ is a Verdier localisation with kernel $\Ker \rho_{ob}$. In particular, the assignment $\CK \mapsto \ell_{\cc/\CK}^\pi$ gives a bijection between thick subcategories $\CK$ of $\cC$ with a locally finite Verdier quotient $\cc/\CK$ and localising basic rank functions on $\cc$.
\end{proof}

The reader may wonder what is the precise relation between exact, idempotent and localising rank functions. 
A prime (basic) rank function which is both exact and idempotent, but not localising, will be discussed in Example \ref{ex:Wodzicki}.
An example of a prime (basic) rank function which is neither idempotent nor exact -- and, therefore, not localising -- will be presented in detail in Section \ref{sec:Example_Cluster}. 
It would be interesting to have an example distinguishing exact and idempotent rank functions or to prove that these two classes coincide. The latter holds true if the underlying triangulated category $\cC$ is the category of compact objects in a compactly generated triangulated category (see Lemma \ref{lem:CE.idemp}). The proof of this fact relies on \cite[Corollary 12.6]{Krause2005}, which uses in an essential way the existence of the ambient large triangulated category. 
Note that all idempotent complete small triangulated categories admitting an enhancement arise as such categories of compact objects (see Remark~\ref{rem: algebraic,etc.=>compacts}). So in most natural situations exact and idempotent rank functions coincide.

\begin{remark}
    One can consider a more general class of localising, exact, and idempotent rank functions by removing the integrality assumption from their definitions. Proposition \ref{prop: kernels} would remain correct. However, Theorem \ref{thm: FDT_generalised} does not extend so easily to this more general setting. The reason being that the characterisation of locally finite categories relies heavily on integrality assumptions (see Proposition \ref{prop:charac_loc_fin_tria} and Corollary \ref{cor:charac_loc_finite_abelian}). It would be interesting to characterise triangulated categories admitting a morphism-faithful, but not necessarily integral, rank function and to describe all such functions. 
    On a related note, observe that if $\pi:\cc \to \cE$ is a CE-quotient functor and $\rho$ is a rank function on $\cE$, then $\rho^\pi$ is integral if and only if $\rho$ is integral. 
	This follows from the fact that $\pi^*$ is essentially surjective and from Theorem \ref{thm:TransInvariantAddFunctions<->rankFunctions}.
\end{remark}

\section{The case of compact objects}\label{sec:compact}

In this section, we restrict ourselves to the following setting: 
 $\cc=\cT^c$ is the full subcategory of compact objects in a compactly generated triangulated category $\cT$. 
Recall that a triangulated category $\cT$ is \emphbf{compactly generated} if it has coproducts and if the full subcategory $\cT^c$ of compact objects in $\cT$ is skeletally small and generates $\cT$, that is, $Y$ in $\cT$ is zero whenever $\Hom_\cT(X,Y)=0$ for every object $X$ in $\cT^c$. Here, an object $X$ in $\cT$ is called \emphbf{compact} if the functor $\Hom_\cT(X,-): \cT \to \Ab$ preserves coproducts.

\begin{remark}
\label{rem: algebraic,etc.=>compacts}

Note that each idempotent complete small triangulated category admitting an enhancement by a stable $\infty$-category is equivalent to the full subcategory of compact objects in a compactly generated triangulated category (see \cite[Proposition 5.4.2.2, Lemma 5.4.2.4]{Lurie_HTT}, \cite[Proposition 1.1.3.6, Lemma 1.2.4.6, Remark 1.4.4.3]{Lurie_HA}). In particular, each algebraic idempotent complete small triangulated category has this property, which also follows from the existence of a dg enhancement (see \cite[\S 4.3]{Keller1994deriving}, \cite[Theorem 3.8, Corollary~3.7]{Keller2006}).
\end{remark}

We will connect rank functions on $\cc$ with several pieces of data associated with the category $\cT$.

\subsection{Rank functions, endofinite objects and definable subcategories}

Every compactly generated triangulated category $\cT$ admits a pure-exact structure. The restricted Yoneda functor $\Upsilon:\cT\to \Mod \cC$, defined by
\[
\begin{tikzcd}
X \ar[r,mapsto] 
& \Upsilon(X)=\Hom_\cT(-,X)|_{\cC} ,
\end{tikzcd}
\]
provides the link between pure-injectives in $\cT$ and injectives in $\Mod \cC$. This is analogous to the study of pure-injectivity in a locally finitely presented category $\cA$ with products via the analysis of the injectives in the associated locally coherent category $\Flat(\Mod((\mod((\fp \cA)^{\op}))^{\op}))$ of flat functors in $\Mod((\mod((\fp \cA)^{\op}))^{\op})$ -- we refer to \cite[\S 3.3]{Crawley-Boevey1994} for further details and to \cite[Theorem 2.8]{Krause1998} for a generalisation of the results in \cite{Crawley-Boevey1994}.

\begin{remark}
\label{rmk:propRestYoneda}
	In general, the functor $\Upsilon$ is not fully faithful, but it commutes with arbitrary products in $\cT$, just like the classic Yoneda embedding. Because of the restriction to compacts, $\Upsilon$ also commutes with arbitrary coproducts.
\end{remark}

\noindent A triangle
\[
\begin{tikzcd}
\label{eq:pureExactTriangle}
X \ar[r, "f"] & Y \ar[r, "g"] & Z \ar[r] & \Sigma X
\end{tikzcd}
\]
in $\cT$ is \emphbf{pure-exact} if
\[
\begin{tikzcd}
0 \ar[r] & \Upsilon (X) \ar[r, "\Upsilon (f)"] & \Upsilon (Y) \ar[r, "\Upsilon (g)"] & \Upsilon (Z) \ar[r] & 0
\end{tikzcd}
\]
forms a short exact sequence in $\Mod \cC$. An object $X$ in $\cT$ is \emphbf{pure-injective} provided that every pure-exact triangle whose first term is $X$ splits. The object $X$ is \emphbf{$\varSigma$-pure-injective} if any coproduct of copies of $X$ is pure-injective.
Using Brown representability, it can be shown that the isoclasses of pure-injective objects in $\cT$ stand in one-to-one correspondence with the isoclasses of injectives in $\Mod \cC$:

\begin{theorem}[{\cite[Lemma 1.7, Theorem 1.8, Corollary 1.9]{Krause2000}}]
\label{thm:compactlyGeneratedTriaPureInj}
	The following statements hold:
	\begin{enumerate}
		\item $X$ is a pure-injective object in $\cT$ if and only if $\Upsilon(X)$ is  injective in $\Mod \cC$;
		\item every injective in $\Mod \cC$ is isomorphic to $\Upsilon (X)$ for some pure-injective object $X$ in $\cT$;
		\item the functor $\Upsilon$ induces an equivalence between the subcategory of pure-injectives in $\cT$ and the subcategory of injectives in $\Mod \cC$.
	\end{enumerate}
\end{theorem}

Following \cite[Definition 1.1]{Krause1999}, an object $Z$ in $\cT$ is called \emphbf{endofinite} if $\Hom_\cT(X,Z)$ is an $\End_{\cT}(Z)$-module of finite length for every compact object $X$ in $\cT$. Endofinite objects are $\varSigma$-pure-injective (see for example \cite[\S7]{bennett2023characterisations}). This goes back to \cite[Theorem 1.2]{Krause1999}. We get the following restricted version of the equivalence in part $(3)$ of Theorem \ref{thm:compactlyGeneratedTriaPureInj}.

\begin{lemma}
\label{cor:BijAndDecompCompGenTria}
		The functor $\Upsilon$ induces an equivalence between the  subcategory of $\varSigma$-pure-injective objects in $\cT$ and the subcategory of $\varSigma$-pure-injective cohomological functors in $\Mod \cC$. This restricts to an equivalence between the subcategory of endofinite objects in $\cT$ and the subcategory of endofinite cohomological functors in $\Mod \cC$.
\end{lemma}

\begin{proof}
Recall from Subsection~\ref{subsection lcc} that an object in $\Mod \cC$ is injective if and only if it is fp-injective and pure-injective, and the class of fp-injective objects is closed under direct factors and coproducts. By Lemma~\ref{lem:cohomologicalFpinjectiveFlat}, a functor in $\Mod \cC$ is fp-injective if and only if it is cohomological. Thus, for $F$ in $\Mod \cC$,  any coproduct of copies of $F$ is injective if and only if $F$ is $\varSigma$-pure-injective and cohomological. By Theorem~\ref{thm:compactlyGeneratedTriaPureInj} and the fact that $\Upsilon$ commutes with coproducts, we see that this happens if and only if $F \cong \Upsilon(X)$ for a $\varSigma$-pure-injective object $X$ in $\cT$. Together with part (3) of Theorem~\ref{thm:compactlyGeneratedTriaPureInj}, this implies the first statement.

Recall that by Corollary \ref{cor:CohEndofin=InjFinEndoLength} endofinite cohomological functors are precisely the finite endolength injectives in $\Mod \cC$. In order to prove that the equivalence in part (3) of Theorem \ref{thm:compactlyGeneratedTriaPureInj} restricts to the equivalence in the second statement, it is enough to observe that for $Z$ in $\cT$ and $X$ in $\cC$, there is a natural isomorphism
	\begin{equation*}
	\Hom_{\cT}(X,Z) = (\Upsilon(Z))(X)\cong \Hom_{\Mod \cC} (\Hom_{\cC}(-,X),\Upsilon(Z)) = \Hom_{\Mod \cC} (\Upsilon(X),\Upsilon(Z)).
	\end{equation*}
Finally, recall that every object of finite endolength in $\Mod \cc$ is $\varSigma$-pure-injective. 
\end{proof}

Moreover, the dictionary provided by Theorem \ref{thm:compactlyGeneratedTriaPureInj} allows us to identify the elements in $\Zg (\Mod \cC)$ with the isoclasses of indecomposable pure-injective objects in $\cT$. The \emphbf{Ziegler spectrum} of a compactly generated triangulated category $\cT$, denoted by $\Sp \cT$, is defined to be the set of all isoclasses of indecomposable pure-injective  objects.
The identification of $\Sp \cT$ with $\Zg (\Mod \cC)$ via $\Upsilon$ induces a topology on $\Sp \cT$, as described in Subsection \ref{subsec:zieglerspectrumlcc}. Recall that a subcategory $\cd$ of $\cT$ is called \emphbf{definable} if there exists a family of functors $(F_i)_{i \in I}$ of the form
\[
\begin{tikzcd}[column sep=large]
	\Hom_\cT(Y_i,-)\ar[r, "{\Hom_\cT(f_i,-)}"]& \Hom_\cT(X_i,-)\ar[r]&F_i\ar[r]& 0,
\end{tikzcd} \text{ for } \begin{tikzcd}
	f_i: X_i  \ar[r]& Y_i
\end{tikzcd} \text{ in } \cC
\]
 such that $\cd=\{X\text{ in }\cT\mid F_i(X)=0 \text{ for all }i\in I\}$. By definition, a definable subcategory of $\cT$ is closed under coproducts. Note that definable subcategories in $\Mod\cC$ can be defined similarly.

We are ready to prove the main result of this subsection, using the fundamental correspondence in \cite{Krause2002a}. An object $Z$ in $\cT$ is called \emphbf{$\Sigma$-invariant} if $Z\cong \Sigma Z$. Moreover, for an ideal $I$ in $\cc$ we denote by $I^\perp$ the full subcategory of $\cT$ given by $\{Z\text{ in }\cT\mid \Hom_\cT(f,Z)=0\text{ for all } f\in I\}$.

\begin{theorem}
	\label{thm:FundamentalCorrespondenceCompGen}
	There is a bijective correspondence between:
	\begin{enumerate}[(a)]
		\item basic rank functions on $\cC$;
		\item isoclasses of basic $\Sigma$-invariant endofinite objects in $\cT$;
		\item $\Sigma$-closed closed discrete subsets of $\Sp \cT$ such that the coproduct of all their elements is
$\varSigma$-pure-injective;
		\item $\Sigma$-closed definable subcategories of $\cT$ whose objects are endofinite.
	\end{enumerate}
Consequently, every basic rank function $\rho$ on $\cC$, when evaluated on objects, is of the form 
$$\length_{\End_{\cT}(Z_\rho)} \Hom_\cT(-,Z_\rho)|_{\cC}$$ 
for some basic $\Sigma$-invariant endofinite object $Z_\rho$ in $\cT$. Moreover, the kernel on morphisms $\Ker\rho$ is the cohomological ideal  $$\Ker\rho = \Ann \Upsilon(Z_\rho) = \{f\text{ in }\cc\mid \Hom_\cT(f,Z_\rho)=0\},$$ and the definable subcategory associated with $\rho$ is $(\Ker\rho)^\perp$, which identifies with the additive closure of $Z_\rho$.
\end{theorem}

\begin{proof}
Combining Theorem \ref{thm:LongBijRankFunc} with Lemma \ref{cor:BijAndDecompCompGenTria}, we get the bijection between (a) and (b). In fact, given a basic $\Sigma$-invariant endofinite object $Z$ in $\cT$, we consider the endofinite cohomological functor $\Upsilon(Z)$, which in turn gives rise to the additive function \[\widetilde{\rho}_{\Upsilon(Z)}=\length_{\End_{\Mod\cC} (\Upsilon(Z))}\Hom_{\Mod\cC}(-,\Upsilon(Z)).\] 
The corresponding rank function evaluated on an object $X$ in $\cC$ is of the form 
\[\rho_{\Upsilon(Z)}(\id_X)=\length_{\End_{\Mod\cC} (\Upsilon(Z))}\Hom_{\Mod\cC}(\Hom_\cC(-,X),\Upsilon(Z))=\length_{\End_{\cT}(Z)} \Hom_\cT(X,Z).\]

Moreover, to pass from (b) to (c), we take the indecomposable summands of the endofinite object. They form a closed subset of $\Sp\cT$, corresponding to a $\Sigma$-closed closed subset of $\Zg (\Mod \cC)$ such that the coproduct of all its elements is endofinite cohomological. By Corollary~\ref{cor:CohEndofin=InjFinEndoLength}, this condition on the coproduct is equivalent to finite endolength. By Remark~\ref{IsolatedRemark}, such subsets of $\Zg (\Mod \cC)$ coincide with $\Sigma$-closed closed discrete subsets such that the coproduct of all their elements is $\varSigma$-pure-injective. Identifying $\Zg (\Mod \cC)$ and $\Sp \cT$ again, these are  precisely the subsets in (c). To show that the passage from (b) to (c) induces a bijective correspondence, we refer again to Remark~\ref{IsolatedRemark} to see that the subsets in (c) correspond precisely to the $\Sigma$-closed closed discrete subsets of $\Zg (\Mod \cC)$ satisfying the isolation property. These are in bijection with (a) by Theorem \ref{thm:LongBijRankFunc} and, therefore, in bijection with (b) as well.

The bijection between (c) and (d) is a consequence of the fundamental correspondence in \cite{Krause2002a} that establishes a bijection between definable subcategories of $\cT$ and closed subsets of $\Sp \cT$ by assigning to $\cd$ the subset $\{[X]\in\Sp\ct\mid X\in\Obj\cd\}$. The converse is given by taking the definable closure of the objects representing the elements of a closed subset, that is, the smallest definable subcategory of $\cT$ containing these objects. In our restricted setting, this converse takes a particularly nice form. More precisely, if $C$ is a closed subset of $\Sp\ct$ such that the coproduct of the representatives of its elements is endofinite, then the corresponding definable subcategory of $\cT$ is given by the additive closure of this coproduct. Indeed, by \cite[Corollary 13.1.13]{Krause2022}, the definable closure of an endofinite object in $\Mod \cC$ coincides with its additive closure. Now we can transport this fact to $\cT$ using \cite[Corollary 4.4]{angeleri2017torsion} and Lemma \ref{cor:BijAndDecompCompGenTria}.
Finally, let $\rho$ be a rank function on $\cc$ with associated endofinite object $Z_\rho$ in $\cT$. It remains to check that
\[\Ker\rho=\{f\text{ in }\cc\mid \Hom_\cT(f,Z_\rho)=0\}.\]
From the assignments in the fundamental correspondence in \cite{Krause2002a}, we then obtain the desired description of the associated definable subcategory as $(\Ker\rho)^\perp$. To this end, observe that $f\in\Ker\rho$ if and only if $\Hom_{\Mod\cC}(\Imm\Hom_\cc(-,f),\Upsilon(Z_\rho))=0$. Using that $\Upsilon(Z_\rho)$ is injective in $\Mod\cC$, the latter is equivalent to $\Hom_{\Mod\cC}(\Upsilon(f),\Upsilon(Z_\rho))=0$, which in turn means that $\Hom_\cT(f,Z_\rho)=0$, as wanted.
\end{proof}

\begin{example}
Let us go back to Example \ref{ex:fin-dim_alg}. Fix an object $X$ in $\cper (A)$ for a finite-dimensional algebra $A$ over a field $k$ and consider the rank function $\vartheta_X$ restricted to $\cper (A)$. Recall that $\vartheta_X$ was defined on objects as $\sum_{i\in\Z}\dim_k\Hom_{\cper (A)}(\Sigma^iX,-)$.   Analysing the long exact sequence \eqref{LongExSeqForEx} for a morphism $f:Y\to Z$ in $\cper(A)$ yields $\vartheta_X(f)=\sum_{i\in\Z}\dim_k \Imm (\Hom_{\cper (A)}(\Sigma^iX,f))$, which implies that the rank function $\vartheta_X$ is integral and that the machinery developed so far applies, as $\cper (A)$ is the subcategory of compact objects in $\D(A)$. Let us now additionally assume that the algebra $A$ is basic, that is $A/J(A)\cong k\times \dots \times k$, where $J(A)$ denotes the Jacobson radical of $A$. Let us also specify to $X=A$. In that case, $\vartheta_A$ is the rank function computing the dimension of the total cohomology. Recall that $\vartheta_A=\vartheta^{\bbD A}$. It is not hard to check that the function $\vartheta_A$ is basic and that the  basic $\Sigma$-invariant endofinite cohomological object in the derived category $\D(A)$, which corresponds to $\vartheta_A$ in Theorem \ref{thm:FundamentalCorrespondenceCompGen}, is $\coprod_{i\in\Z}\Sigma^i\bbD A$. Thus, on objects 
of $\cper (A)$ we have
\[
\vartheta_A(-)=\length_{\End_{\D(A)}(\coprod\limits_{i\in\Z}\Sigma^i\bbD A)} \Hom_{\D(A)}(-,\coprod_{i\in\Z}\Sigma^i\bbD A).
\]
\end{example}

\subsection{Idempotent rank functions and smashing localisations} As before, assume that $\cc$ is the subcategory $\cT^c$ of compact objects in a compactly generated triangulated category $\cT$.

\begin{lemma}
\label{lem:CE.idemp} 
A rank function on $\cc$ is exact if and only if it is idempotent. Consequently, the correspondence mapping $\pi:\cc \to \cE$ to $\ell_\cE^\pi$ induces a bijection between:
	\begin{enumerate}[(a)]
		\item equivalence classes of CE-quotient functors from $\cC$ to locally finite triangulated categories;
		\item idempotent basic rank functions on $\cC$.
	\end{enumerate}
\end{lemma}

\begin{proof}
Recall that an integral rank function $\rho$ on $\cC$ is exact if and only if $\Ker \rho$ is an exact ideal, as defined after Definition \ref{definition_exact}. By \cite[Corollary 12.6, Lemma 8.4]{Krause2005}, an ideal in $\cC$ is exact if and only if it is $\Sigma$-closed, cohomological, and idempotent. By Lemma~\ref{lem:kernel-cohom_ideal}, the ideal $\Ker \rho$ is $\Sigma$-closed and cohomological for any rank function $\rho$, hence, we get that $\Ker \rho$ is exact if and only if it is idempotent. The first statement now follows from the definition of idempotent rank functions. The bijection between (a) and (b) is nothing but the bijection in Theorem \ref{thm: FDT_generalised}. 
\end{proof}

In what follows, we provide a link between rank functions and smashing localisations. Recall that a triangulated subcategory of $\cT$ closed under coproducts is called \emphbf{localising}, and a localising subcategory of $\cT$ is said to be \emphbf{smashing} if the inclusion functor admits a right adjoint preserving coproducts. Smashing subcategories $\cS$ of $\cT$ occur naturally as kernels of localisation functors preserving coproducts, and they induce decompositions of $\cT$ into smaller triangulated categories, namely torsion-torsionfree triples of the form $(\cS,\cS^\perp,(\cS^\perp)^\perp)$ (see \cite{Nicolas}). A typical example of a smashing subcategory is a localising subcategory generated by compact objects, i.e.~the smallest localising subcategory containing a given set of compacts.  The claim that all smashing subcategories of $\cT$ arise in this way is often referred to as the telescope conjecture, which was proved in some cases, but fails in general (see, for example, \cite{Keller1994remark}). 
The following is the main result of this subsection:

\begin{theorem}
\label{thm_smashing}
	The assignment sending $\rho$ to \[ \operatorname{Filt} (\Ker \rho) \coloneqq \{Z \text{ in } \cT \mid \text{every morphism } X \rightarrow Z, \text{ with } X \text{ in }  \cC, \text{ factors through some morphism in } \Ker \rho \}\]
induces a bijection between:
	\begin{enumerate}[(a)]
	\item idempotent basic rank functions on $\cC$;
		\item smashing subcategories $\cS$ of $\cT$ with $(\cT/\cS)^c$ locally finite.
	\end{enumerate}
Moreover, the rank function  is localising if and only if the corresponding smashing subcategory is compactly generated.	
\end{theorem}

\begin{proof}
By the previous lemma, idempotent basic rank functions $\rho$ on $\cC$ are in bijection with CE-quotient functors $\pi:\cc \to \cE$ to locally finite triangulated categories. Following this assignment, $\Ker \rho=\Ann \pi$ is an exact ideal. By \cite[Corollary 12.5]{Krause2005}, exact ideals $I$ in $\cc$ are in bijection with smashing subcategories $\cS=\operatorname{Filt}I$ of $\cT$. Moreover, by \cite[Corollary 12.7]{Krause2005}, the category $\cE$ can be identified with a subcategory of $\cT/\cS$ and the localisation $\cT \to \cT/\cS$ restricts to the CE-quotient $\pi:\cc \to \cE$. The subcategory $(\cT/\cS)^c$ then identifies with the closure of $\cE$ under direct factors in $\cT/\cS$. Thus, $\cE$ is locally finite if and only if $(\cT/\cS)^c$ is locally finite and we get the bijection in the statement. Finally, a smashing subcategory is generated by compact objects if and only if the corresponding CE-quotient is equivalent to a Verdier localisation (\cite[Lemma 13.1, Proposition 13.2]{Krause2005}). This is, in turn, equivalent to the rank function being localising by Theorem~\ref{thm: FDT_generalised}.
\end{proof}

This bijection can be understood as a special case of the bijection between (a) and (d) in Theorem \ref{thm:FundamentalCorrespondenceCompGen}. To see this, observe that the localisation functor $\cT\rightarrow \cT/\cS$ for any given smashing subcategory $\cS$ of $\cT$ admits a fully faithful right adjoint, which identifies $\cT/\cS$ with $\cS^\perp$ inside of $\cT$. The category $\cT/\cS \overset\sim\to \cS^\perp$ is always a compactly generated triangulated category (see \cite[Theorem 11.1]{Krause2005}). If $\cS$ arises from an idempotent rank function on $\cc$, the subcategory $\cS^\perp$ is precisely the definable subcategory of $\cT$ associated with the rank function in Theorem \ref{thm:FundamentalCorrespondenceCompGen} (see \cite[Theorem 12.1]{Krause2005}). 
In fact, the rank function is idempotent if and only if the corresponding definable subcategory is triangulated. This follows from the next lemma known to experts but somewhat hidden in the literature (see, for example, \cite[Proposition 6.3, Remark 6.4]{LakingVitoria2020}).

\begin{lemma}
\label{lem:smashing-definable}
The assignments $\cS \mapsto \cS^\perp$ and $\cd \mapsto {}^\perp \cd$ induce mutually inverse bijections between:
\begin{enumerate}[(a)]
    \item smashing subcategories $\cS$ of $\cT$;
    \item triangulated definable subcategories $\cd$ of $\cT$.
\end{enumerate}
\end{lemma}

\begin{proof}
Let $\cS$ be a smashing subcategory of $\cT$. As already observed, $\cT/\cS$ identifies with $\cS^\perp$ inside $\cT$. Moreover, by \cite[Corollary 12.5]{Krause2005}, we have $\cS=\operatorname{Filt} I$ for some exact ideal $I$ in $\cc$. Hence, by \cite[Theorem 12.1]{Krause2005}, it follows that $\cS^\perp=I^\perp$, which is a definable subcategory by definition.

To prove the converse, first note that by \cite[Proposition 4.5]{angeleri2017torsion}, if $\cd$ is a triangulated definable subcategory of $\cT$, then $({}^\perp \cd, \cd)$ is a $t$-structure with ${}^\perp \cd$ triangulated. Since $\cd$ is closed under coproducts (as it is a definable subcategory of $\cT$), the result then follows from the equivalence between conditions (5) and (6) in \cite[Proposition 4.4.3]{Nicolas}.
\end{proof}

\begin{corollary}
\label{cor:exact-def_is_trian}
A cohomological $\Sigma$-closed ideal $I$ is exact if and only if the definable subcategory $I^\perp$ in $\cT$ is triangulated.
\end{corollary}

\begin{proof}
As discussed in the proof of Lemma \ref{lem:smashing-definable}, if $I$ is exact, then $I^\perp=(\operatorname{Filt} I)^{\perp}$ is triangulated. Conversely, if $\cd\coloneqq I^\perp$ is triangulated, then ${}^\perp \cd$ is a smashing subcategory by Lemma \ref{lem:smashing-definable}, and so  $\cd = J^{\perp}$ for some exact ideal $J$. Since $J$ is cohomological, it must coincide with $I$ by the fundamental correspondence between cohomological ideals and definable subcategories in \cite{Krause2002a}. Hence, $I$ is exact.
\end{proof}

As a consequence of Lemma \ref{lem:CE.idemp} and Corollary \ref{cor:exact-def_is_trian}, we get the desired characterisation of idempotent rank functions.

 \begin{corollary}
 \label{cor:idemp-cdrho_triang}
 Let $\rho$ be a basic rank function on $\cc$ with the associated definable subcategory $\cd_\rho\coloneqq(\Ker\rho)^\perp$ in $\cT$. Then $\rho$ is idempotent if and only if $\cd_\rho$ is a triangulated subcategory of $\cT$.
 \end{corollary}
 
We consider further consequences of the established link with smashing subcategories. Theorem \ref{thm_smashing} allows us to measure the failure of a restricted version of the telescope conjecture via rank functions. Note that the counterexample to the telescope conjecture presented in \cite{Keller1994remark} also serves as a counterexample to this restricted version (see Example \ref{ex:Wodzicki}).

\begin{corollary}
\label{cor_smashing}
Every smashing subcategory $\cS$ in $\ct$ with $(\cT/\cS)^c$ locally finite is generated by compact objects if and only if every idempotent rank function on $\cC$ is localising.
\end{corollary}

We now consider the special case where $\cT=\D(A)$ is the derived category of a dg algebra $A$ over a commutative ring $k$ and $\cc=\cT^c=\cper (A)$. We wish to extend the bijection between localising prime rank functions on $\cper(A)$ and equivalence classes of finite homological epimorphisms of dg algebras $A\to B$ with $B$ simple artinian, proved in \cite[Theorem 6.5]{Chuang2021}.

Recall that the category of dg algebras over $k$ carries a projective model structure and any dg algebra $A$ admits a cofibrant replacement quasi-isomorphic to $A$ (see \cite{Jardine}, \cite[Theorem 4.1(3), \S 5]{SchwedeShipley} for the construction of the model structure, and \cite[\S 2, Appendix A.4]{BazzoniStovicek} for the discussion of its relevance to the study of homological epimorphisms). From now on, we will assume that $A$ is cofibrant.  
A detailed exposition on homological epimorphisms can be found in \cite{Nicolas,BruningHuber2008,Pauksztello, NicolasSaorin, BazzoniStovicek}. We will give only the necessary definitions.
A morphism of dg algebras $f:A\to B$ is a \emphbf{homological epimorphism} if the map $B\otimes^L_A B \to B$ induced by multiplication is a quasi-isomorphism. This holds if and only if the  restriction functor $f_*:\D(B)\to \D(A)$ is fully faithful. 
The restriction functor $f_*$ has a left adjoint given by $-\otimes^L_A B$. 
Moreover, the essential image of $f_*$ is of the form $\cS^\perp$ for a smashing subcategory $\cS$ of $\D(A)$, where $$\cS=\{X \text{ in }\D(A)\mid X\otimes^L_A B=0\}.$$
We can identify the functor $-\otimes^L_A B: \D(A)\to \D(B)\cong \D(A)/\cS$ with the smashing localisation associated with $\cS$. Two homological epimorphisms $f:A\to B$ and $f':A\to B'$ are called \emphbf{equivalent}, if the essential images of the  restriction
functors $f_*$ and $f'_*$ coincide (equivalently, if there exists an isomorphism $g: B \to B'$ in the homotopy category of dg algebras such that $f' = g \circ f$). This yields a bijection between equivalence classes of homological epimorphisms $A\to B$ and smashing subcategories of $\D(A)$ (\cite[Theorem 5.4.4]{Nicolas}, \cite[Theorem 3]{BruningHuber2008},\cite[\S 2]{BazzoniStovicek}). Combining this with Theorem \ref{thm_smashing}, we obtain the following result:

\begin{theorem}
\label{cor_dg}
Let $A$ be a cofibrant dg algebra. There is a bijection between:
	\begin{enumerate}[(a)]
	\item idempotent basic rank functions on $\cper (A)$;
		\item equivalence classes of homological epimorphisms $A\to B$ with $\cper (B)$ locally finite. 
	\end{enumerate}
The idempotent basic rank function $\rho$ on $\cper (A)$ corresponding to the homological epimorphism $f:A\to B$ has the following counterpart on objects of $\cper (A)$:
\[
\rho_{ob}(X)=(\ell_{\cper (B)})_{ob}(X\otimes_A^L B).
\]
\end{theorem}

\begin{proof}
The bijection between (a) and (b) follows from Theorem \ref{thm_smashing} and the fact that smashing subcategories of $\D(A)$ and equivalence classes of homological epimorphisms starting in $A$ are in one-to-one correspondence. To check the description of $\rho_{ob}$, note that the bijection in Theorem \ref{thm_smashing} first associates with the smashing localisation $-\otimes^L_A B: \D(A)\to \D(B)$ a CE-quotient functor, which identifies, up to direct factors, with the restriction 
$-\otimes^L_A B: \cper(A)\to \cper(B)$.
The corresponding rank function is then of the form $(\ell_{\cper(B)})^{-\otimes^L_A B}$ by Lemma \ref{lem:CE.idemp}. When evaluated on objects, this yields $\ell_{\cper (B)}(X\otimes_A^L B)$.
\end{proof}

Note that in the case of prime localising rank functions, the bijection in Theorem \ref{cor_dg} restricts to the bijection in \cite[Theorem 6.5]{Chuang2021}. Indeed, both bijections go via the kernel of the Verdier localisation $-\otimes^L_A B: \cper(A)\to \cper(B)$ in this restricted setting.  

\begin{remark}
  In fact, since \cite[Theorem 5.4.4]{Nicolas} and \cite{NicolasSaorin} apply to homological epimorphisms of small dg categories, there is an analogous bijection, for a small cofibrant  dg category $\cA$ (using the model structure defined in \cite{Tabuada}), between idempotent basic rank functions on $\cper (\cA)$ and equivalence classes of homological epimorphisms of dg categories $\cA \to \cb$ with $\cper (\cb)$ locally finite. 
  In particular, using Remark~\ref{rem: algebraic,etc.=>compacts}, one gets a bijection between idempotent basic rank functions on a skeletally small algebraic triangulated category  with split idempotents and certain homological epimorphisms from its cofibrant dg enhancement.
\end{remark}

We finish this section with an example of an exact rank function that is not localising. It is defined on the subcategory of compact objects in a compactly generated triangulated category which does not satisfy the telescope conjecture. This example stems from the famous example by Wodzicki \cite[Example 4.7(3)]{Wodzicki1989}.

\begin{example}
\label{ex:Wodzicki}
Let $k$ be a field. Fix an integer $l\geq 2$. Consider the algebra
\[
B\coloneqq k[t,t^{l^{-1}}, t^{l^{-2}}, \dots]=\bigcup_{n=0}^{\infty}k[t^{l^{-n}}].
\]
Let $J$ be the augmentation ideal of $B$. The ideal $J$ is generated by $t,t^{l^{-1}}, t^{l^{-2}}, \dots$ Let $A$ be the localisation of $B$ at $J$, which is local. Denote by $I$ the unique maximal ideal of $A$. The ideal $I$ satisfies the following property: 
$\operatorname{Tor}^A_i(A/I,A/I)=0 \text{ for } i>0$ (see \cite{Keller1994remark}). By \cite[Proposition]{Keller1994remark}, the localising subcategory of $\operatorname{D(}\hspace{-2pt}A\hspace{-2pt}\operatorname{)}$ generated by $I$, considered as an $A$-module, is smashing, but does not contain any compact objects.  The map $A\to A/I$ is a homological epimorphism, which induces the cohomological quotient functor
\[
F\coloneqq -\otimes_A A/I: \cper(A) \to \cper(A/I)
\]
with $\Ker F=0$ and $\Ann F\neq 0$ (see \cite[\S 15]{Krause2005}). Note that the algebra $A/I$ is a field, and so $F$ is essentially surjective. Therefore, the functor $F$ is in fact a CE-quotient functor by \cite[Lemma 7.3]{Krause2005}.
 Thus, the triangulated category $\cper(A/I)$ is simple and there exists a unique irreducible rank function $\rho$ on $\cper(A/I)$. As before, we define the rank function $\rho^F$ on $\cper(A)$. On objects it takes the following form:
\[
\rho_{ob}^F(X)\coloneqq\rho_{ob}(F(X))\text{ for } X\text{ in } \cper(A).
\]
Then $\rho^F$ is an exact rank function which is not localising. Note that $\rho^F$ is prime by construction. Indeed, $F(A)=A/I$ is an indecomposable generator of
$\cper(A/I)$, thus $\rho_{ob}(A/I) = 1$ and $\rho_{ob}^F(A)= \rho_{ob}(A/I) = 1$.
\end{example}

In the example above, the target of the CE-quotient functor is a simple triangulated category and the corresponding rank function is prime. However, we do not know if the bijections from Theorem \ref{thm: FDT_generalised} and Lemma \ref{lem:CE.idemp} restrict to bijections between exact, respectively idempotent, prime rank functions and CE-quotient functors to simple triangulated categories (compare with Theorem \ref{thm:localising_CL}). 

\section{Example: cluster category of type \texorpdfstring{$A_3$}{A3}}
\label{sec:Example_Cluster}
In this section, we consider an example of a locally finite triangulated category.
We describe all integral rank functions on this category and show that one of the irreducible rank functions is prime, but neither idempotent nor exact or localising. 

Let $\cC_{A_3}$ be the cluster category of type $A_3$ over an algebraically closed field $k$. This category can be constructed as the orbit category of the category $\operatorname{D^b}(kA_3)$ modulo the action of $\tau^{-1}\Sigma$ for a fixed orientation of $A_3$ (\cite{Buan2006tilting}). Here $\tau$ denotes the Auslander--Reiten translate. The category $\cC_{A_3}$ is a locally finite triangulated category. 
It is Krull--Schmidt, $\Hom$-finite, 6-periodic and has Auslander--Reiten triangles. It has 9 indecomposable objects up to isomorphism.  

\vspace{-50pt}
\begin{figure}[h]
    \centering
    \[\begin{tikzcd}
\begin{tikzpicture}

\node (v1) at (-5,-3.5) {$T_1$};
\node (v2) at (-4,-2.5) {$T_2$};
\node (v3) at (-3,-1.5) {$T_3$};
\node (v5) at (-2,-2.5) {$\Sigma^{-1}T_2$};
\node (v7) at (-1,-3.5) {$\Sigma^{-2}T_1$};
\node (v4) at (-3,-3.5) {$\Sigma^{-1}T_1$};
\node (v6) at (-1,-1.5) {$\Sigma T_1$};
\node (v8) at (0,-2.5) {$\Sigma T_2$};
\node (v10) at (1,-3.5) {$\Sigma T_3$};
\node (v9) at (1,-1.5) {$T_1$};
\node (v11) at (2,-2.5) {$T_2$};
\node (v12) at (3,-3.5) {$T_3$};
\draw [->]  (v1) -- (v2);
\draw [->]  (v2) -- (v3);
\draw [->]  (v2) -- (v4);
\draw [->]  (v3) -- (v5);
\draw [->]  (v4) -- (v5);
\draw [->]  (v5) -- (v6);
\draw [->]  (v5) -- (v7);
\draw [->]  (v7) -- (v8);
\draw [->]  (v6) -- (v8);
\draw [->]  (v8) -- (v9);
\draw [->]  (v8) -- (v10);
\draw [->]  (v9) -- (v11);
\draw [->]  (v10) -- (v11);
\draw [->]  (v11) -- (v12);
\end{tikzpicture}
\end{tikzcd}\]
    \caption{The Auslander--Reiten quiver of $\cC_{A_3}$.}
   \label{fig:my_label}
\end{figure}
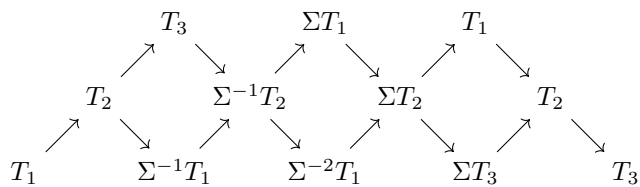

Simple objects in $\mod\cC_{A_3}$ correspond to Auslander--Reiten triangles in $\cC_{A_3}$ (for example, see \cite[Apendix A]{Krause2012}). Namely, every simple object $S$ in $\mod\cC_{A_3}$ admits a presentation 
\[
\begin{tikzcd}
 \Hom_{\cC_{A_3}} (-,Y) \ar[rr, "{\Hom_{\cC_{A_3}} (-,g)}"] && \Hom_{\cC_{A_3}} (-,Z) \ar[r] & S,
\end{tikzcd}
\] 
where 
\[
\begin{tikzcd}
 X \ar[r] & Y \ar[r, "g"] & Z \ar[r] & \Sigma X
\end{tikzcd}
\]
is an Auslander--Reiten triangle in $\cC_{A_3}$. Conversely, every Auslander--Reiten triangle in $\cC_{A_3}$ produces a simple object in $\mod\cC_{A_3}$ in that way. 

There are precisely 2 $\Sigma$-orbits of Auslander--Reiten triangles in $\cC_{A_3}$, with representatives 
\begin{equation}\label{seq1}
    \begin{tikzcd}
 \Sigma T_1 \ar[r] & \Sigma T_2 \ar[r] & T_1 \ar[r] & T_3
\end{tikzcd}
\end{equation}
and 
\begin{equation}\label{seq2}
    \begin{tikzcd}
 \Sigma T_2 \ar[r] & T_1 \oplus \Sigma T_3 \ar[r] & T_2 \ar[r] & \Sigma^{-1} T_2.
\end{tikzcd}
\end{equation}
The first orbit has 6 elements up to isomorphism, corresponding to 6 indecomposable objects on the rim of the Auslander--Reiten quiver of $\cC_{A_3}$, which can be drawn on the M\"obius strip.
The second orbit has 3 elements, corresponding to the objects in the middle. The injective envelope of the simple object $S_1$ corresponding to the triangle \eqref{seq1} is $\Hom_{\cC_{A_3}} (-,T_3)$, the envelope of the simple object $S_2$ corresponding to the triangle \eqref{seq2} is $\Hom_{\cC_{A_3}} (-,\Sigma^{-1} T_2)$. 

Since $\cC_{A_3}$ is  locally finite, taking injective envelopes provides a  bijection between  
simple objects in $\mod \cC_{A_3}$ and indecomposable injective objects in $\Mod \cC_{A_3}$ (see the proof of Proposition \ref{prop:IrrAdditiveFunctions<->InjFinEndolength2}). Thus, we have found all indecomposable injectives in $\Mod \cC_{A_3}$. By Lemma \ref{lem:injFinEndolength}, all injectives in $\Mod \cC_{A_3}$ have finite endolength, so any injective is a coproduct of indecomposable ones. Hence, $\Sigma$-orbits of indecomposable injectives are in bijection with irreducible rank functions (see  Corollary \ref{cor:periodic}), and the irreducible rank functions can be described using the formula from Theorem \ref{thm:FundamentalCorrespondenceCompGen}.

Thus, there are 2 $\Sigma$-irreducible additive functions on $\mod\cC_{A_3}$. The irreducible rank function $\rho_1$ on $\cC_{A_3}$ corresponding to the $\Sigma$-orbit of the first Auslander--Reiten triangle is given by
\[
{\rho_1}_{ob}\coloneqq\length_{\End_{\cC_{A_3}} (I_1)}\Hom_{\cC_{A_3}}(-,I_1), \text{ where } I_1=\bigoplus_{i=0}^5\Sigma^i T_3. 
\]
Similarly, the irreducible rank function $\rho_2$ on $\cC_{A_3}$ corresponding to the $\Sigma$-orbit of the second triangle is given by
\[
{\rho_2}_{ob}\coloneqq\length_{\End_{\cC_{A_3}} (I_2)}\Hom_{\cC_{A_3}}(-,I_2), \text{ where } I_2=\bigoplus_{i=0}^2\Sigma^i T_2. 
\]
These can also be described as follows:
\[
{\rho_1}_{ob}=\dim_k\Hom_{\cC_{A_3}}(-,I_1) \text{ and } 
{\rho_2}_{ob}=\dim_k\Hom_{\cC_{A_3}}(-,I_2). 
\]
Indeed, $\dim_k\Hom_{\cC_{A_3}}(-,I_1)$ induces an additive function on $\mod\cC_{A_3}$ whose values coincide with those of $\widetilde{\rho_1}$ on simple objects (equivalently, on the connecting morphisms of the triangles \eqref{seq1} and \eqref{seq2}). The same holds for $\rho_2$. 

Note that both $\rho_1$ and $\rho_2$ are not localising. Indeed, the values of ${\rho_i}_{ob}$ on objects are given on Figures \ref{fig:my_label_2} and \ref{fig:my_label_3}, respectively. Since they are non-zero on all indecomposable objects, they are non-zero on all objects. Thus, $\Ker\rho_i$ does not contain the identity map of any non-zero object. At the same time, $\Ker\rho_i\neq 0$ for both rank functions: these kernels contain the connecting morphisms of triangles (\ref{seq2}) and (\ref{seq1}), respectively. For example, consider the connecting morphism $f: T_1 \to T_3$ of triangle (\ref{seq1}). We have $\rho_2(f)=\frac{1+1-2}{2}=0$, as illustrated on Figure \ref{fig:my_label_3}. This can be seen explicitly from the triangle 
\[
\begin{tikzcd}
 T_1 \ar[r, "f"] &  T_3 \ar[r] & \Sigma^{-1}T_2 \ar[r] & \Sigma T_1.
 \end{tikzcd}
\]

Every indecomposable object of $\cC_{A_3}$ generates $\cC_{A_3}$. In particular, so does $T_1$, hence $\rho_2$ is prime (${\rho_2}_{ob}(T_1)~=~1$). The sum $\rho\coloneqq\rho_1+\rho_2$ gives the morphism-faithful rank function corresponding to the composition length on the category $\mod\cC_{A_3}$.

\begin{figure}[h]
    \centering
    \begin{tikzpicture}

\node (v1) at (-5,-3.5) {2};
\node (v2) at (-4,-2.5) {2};
\node (v3) at (-3,-1.5) {2};
\node (v5) at (-2,-2.5) {2};
\node (v7) at (-1,-3.5) {2};
\node (v4) at (-3,-3.5) {2};
\node (v6) at (-1,-1.5) {2};
\node (v8) at (0,-2.5) {2};
\node (v10) at (1,-3.5) {2};
\node (v9) at (1,-1.5) {2};
\node (v11) at (2,-2.5) {2};
\node (v12) at (3,-3.5) {2};
\draw [->]  (v1) -- (v2);
\draw [->]  (v2) -- (v3);
\draw [->]  (v2) -- (v4);
\draw [->]  (v3) -- (v5);
\draw [->]  (v4) -- (v5);
\draw [->]  (v5) -- (v6);
\draw [->]  (v5) -- (v7);
\draw [->]  (v7) -- (v8);
\draw [->]  (v6) -- (v8);
\draw [->]  (v8) -- (v9);
\draw [->]  (v8) -- (v10);
\draw [->]  (v9) -- (v11);
\draw [->]  (v10) -- (v11);
\draw [->]  (v11) -- (v12);
\end{tikzpicture}
    \caption{Values of ${\rho_1}_{ob}$ on indecomposable objects of $\cC_{A_3}$.}
    \label{fig:my_label_2}
\end{figure}
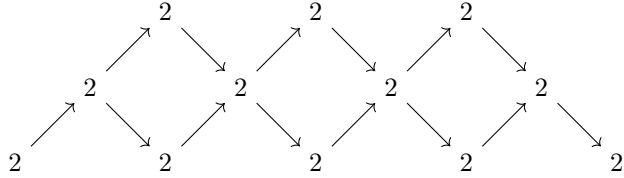

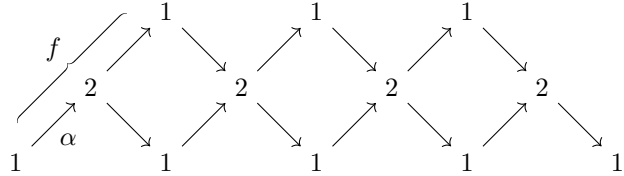
\begin{figure}[h]
    \centering
\begin{tikzpicture}

\node (v1) at (-5,-3.5) {1};
\node (v2) at (-4,-2.5) {2};
\node (v3) at (-3,-1.5) {1};
\node (v5) at (-2,-2.5) {2};
\node (v7) at (-1,-3.5) {1};
\node (v4) at (-3,-3.5) {1};
\node (v6) at (-1,-1.5) {1};
\node (v8) at (0,-2.5) {2};
\node (v10) at (1,-3.5) {1};
\node (v9) at (1,-1.5) {1};
\node (v11) at (2,-2.5) {2};
\node (v12) at (3,-3.5) {1};
\node (v13) at (-4.3,-3.2) {$\alpha$};
\draw [->]  (v1) -- (v2);
\draw [->]  (v2) -- (v3);
\draw [->]  (v2) -- (v4);
\draw [->]  (v3) -- (v5);
\draw [->]  (v4) -- (v5);
\draw [->]  (v5) -- (v6);
\draw [->]  (v5) -- (v7);
\draw [->]  (v7) -- (v8);
\draw [->]  (v6) -- (v8);
\draw [->]  (v8) -- (v9);
\draw [->]  (v8) -- (v10);
\draw [->]  (v9) -- (v11);
\draw [->]  (v10) -- (v11);
\draw [->]  (v11) -- (v12);
\node (v13) at (-5.1,-3.1) {};
\node (v14) at (-3.4,-1.4) {};
\node at (-4.5,-2) {$f$};
\draw [decorate, decoration = {calligraphic brace}] (v13) --  (v14);
\end{tikzpicture}
    \caption{Values of ${\rho_2}_{ob}$ on indecomposable objects of $\cC_{A_3}$.}
    \label{fig:my_label_3}
\end{figure}

The category $\cC_{A_3}$ is algebraic, idempotent complete, and has a generator. Thus, it can be realised as the category $\ct^c$ for some compactly generated triangulated category $\ct$ (see Remark \ref{rem: algebraic,etc.=>compacts}). Therefore, the classes of exact and idempotent rank functions on $\cC_{A_3}$ coincide. Let us check that $\rho_2$ is not idempotent, the fact that it is not exact will then follow by the above. We will check that $f$ cannot be written as $g\circ h$ with both $g$ and $h$ in $\Ker \rho_2$.

For any factorisation $f= g \circ h$ of $f$, the morphism $h:T_1\to \oplus_{i=1}^nX_i$ has components $h_i:T_1\to X_i$ of the form $c_i\id_{T_1}: T_1 \to T_1$, $c_i\alpha:T_1\to T_2$, $c_if:T_1\to T_3$, or $0$ (with $c_i\in k^*$). Note that at least one of the components has to be non-zero. The following inequality holds for any rank function $\rho$: $\rho(h)\geq\rho(h_i)$ for any $i$ (see Remark \ref{rmk:inequalities}). Thus, if $h$ has at least one component of the form $c_i\id_{T_1}$ or $c_i\alpha,$ we have $\rho_2(h)\geq {\rho_2}_{ob}(T_1)=1$ or $\rho_2(h)\geq \rho_2(\alpha)=1$. If all components of $h$ are of the form $c_i f$ or $0$, the map $g$ must have a component of the form $c_i^{-1} \id_{T_3}$, and using the same argument, we get that $\rho_2(g)\geq {\rho_2}_{ob}(T_3) = 1$. In either case, $f$ cannot be decomposed as $f = g\circ h$ with both $g$ and $h$ in $\Ker\rho_2$.

\appendix
\section{Proof of Theorem \ref{thm:lastPreliminaries}}
\label{Appendix A}

In this appendix, we prove Theorem \ref{thm:lastPreliminaries}. The next proposition provides the missing link between Theorem \ref{thm:IrrAdditiveFunctions<->InjFinEndolength} and Theorem \ref{thm:ZieglerSpecLocallyCohCat} in Subsection \ref{subsection lcc}.
 Recall that by $\widetilde{\rho}_H$ we denote the additive function corresponding to an injective object $H$ of finite endolength, as defined in \eqref{eq:addFuncFromObj}. 

\begin{proposition}
	\label{prop:IrrAdditiveFunctions<->InjFinEndolength2}
	Let $\cA$ be a locally coherent category. Suppose that $H$ is an injective object in $\cA$ of finite endolength, let $\widetilde{\rho}$ be an integral additive function on $\fp\cA$ and assume that $\cS$ is a Serre subcategory of $\fp \cA$. The following statements hold:
	\begin{enumerate}
		\item Let $H\cong \coprod_{i\in I} H_i$ be a decomposition of $H$ as a coproduct of indecomposable objects in $\cA$. Each $H_i$ has finite endolength. If $J\subseteq I$ is such that $(H_i)_{i\in J}$ contains each isoclass from $(H_i)_{i\in I}$ exactly once, then the integral additive function $\widetilde{\rho}_H$ can be  expressed 
		as the locally finite sum of irreducible additive functions 
		$\widetilde{\rho}_H=\sum_{i\in J} \widetilde{\rho}_{H_i}$. Moreover $\{[H_i]\mid i\in J\}$ coincides with $\Zg \cA\setminus\co(\Ker\widetilde{\rho}_H)$.
		\item If $\widetilde{\rho}=\sum_{i\in I} \widetilde{\rho}_i$ is a decomposition of $\widetilde{\rho}$ as a locally finite sum of irreducible additive functions and $J\subseteq I$ is such that $(\widetilde{\rho}_i)_{i\in J}$ contains each function from $(\widetilde{\rho}_i)_{i\in I}$ exactly once, then $\Ker\widetilde{\rho}=\Ker\widetilde{\sigma}$ for $\widetilde{\sigma}=\sum_{i\in J} \widetilde{\rho}_i$ and there exists a finite endolength injective object $H$ in $\cA$ satisfying $\widetilde{\sigma}=\widetilde{\rho}_H$.
		\item If $\cA/ \overrightarrow{\cS}$ is locally finite, then there exists an integral additive function $\widetilde{\rho}$ on $\fp \cA$ such that $\cS=\Ker\widetilde{\rho}$.
	\end{enumerate}
\end{proposition}
\begin{proof}
The statement in (1) is an extension of \cite[Proposition 3.1]{Krause2016} (see also \cite[Propositions 4.5, 4.6]{Krause2001}) to the context of locally coherent categories. For completeness, a proof is included. Let $H\cong \coprod_{i\in I} H_i$ be a decomposition of $H$ as a coproduct of indecomposable objects. Note that such a decomposition always exists and that the direct factors $H_i$ are injective of finite endolength. Recall that $\Ker {\widetilde{\rho}}_H$ is a Serre subcategory of $\fp \cA$ and consider the localisation sequence \eqref{eq:loc} for $\cS\coloneqq\Ker {\widetilde{\rho}}_H$. We will use the notation of \eqref{eq:loc}. Recall also that the functor $R$ commutes with filtered colimits and identifies $\cA/\overrightarrow{\cS}$ with the full subcategory of $\cA$ given by all objects $X$ with $\Hom_\cA(\overrightarrow{\cS},X)=0=\Ext_\cA^1(\overrightarrow{\cS},X)$.  If $Y$ lies in $\cS$, then $\widetilde{\rho}_H(Y)=0$, so $\Hom_{\cA}(Y,H)=0$ and $\Hom_\cA(\overrightarrow{\cS},H)=0$. Moreover, $\Ext_\cA^1(\overrightarrow{\cS},H)=0$, as $H$ is injective. The same holds for $H_i$. According to Subsection \ref{subsection:fundamental_corresp}, we must have $H\cong R(\overline{H})$ and $H_i\cong R(\overline{H}_i)$ for injectives of finite endolength $\overline{H}$ and $\overline{H}_i$ in $\cA/\overrightarrow{\cS}$ with $\coprod_{i\in I} \overline{H}_i\cong \overline{H}$. The last isomorphism holds, since the functors $E\circ R$ and $ \id_{\cA/\overrightarrow{\cS}}$ are naturally isomorphic and $E$ commutes with coproducts.
Moreover, using the notation in Lemma \ref{lem:AddFuncPrecomposedFunctor}, we get $\widetilde{\rho}_{\overline{H}}^E=\widetilde{\rho}_H$ and $\widetilde{\rho}_{\overline{H}_i}^E=\widetilde{\rho}_{H_i}$ (see Subsection \ref{subsection:fundamental_corresp} (4)).
Suppose that $(H_i)_{i\in J}$ is a complete irredundant list of indecomposable direct factors of $H$, up to isomorphism. As $R$ is fully faithful, $(\overline{H}_i)_{i\in J}$ contains each isoclass from $(\overline{H}_i)_{i\in I}$ exactly once. Note that by Proposition \ref{prop:LocAdditiveFunctions}, $\fp(\cA / \overrightarrow{\cS})$ is a length category. We claim that $\overline{H}$ cogenerates $\fp(\cA / \overrightarrow{\cS})$ (meaning that for every $Z$ in $\fp(\cA / \overrightarrow{\cS})$ there is a monomorphism in $\cA / \overrightarrow{\cS}$ from $Z$ to a product of copies of $\overline{H}$). In fact, if $X$ lies in $\fp \cA$ and $\Hom_{\cA/\overrightarrow{\cS}}(E(X),\overline{H})=0$, then $\Hom_{\cA}(X,H)=0$, so $X$ belongs to $\cS$ and hence $E(X)=0$. Since every simple in $\fp(\cA/\overrightarrow{\cS})$ is of the form $E(X)$ for $X\in \fp \cA$ (recall Subsection \ref{subsection:fundamental_corresp} (2)), the last observation implies that $\overline{H}$ cogenerates every simple in $\fp(\cA/\overrightarrow{\cS})$ and therefore every object in $\fp(\cA/\overrightarrow{\cS})$. We prove that the functions $\widetilde{\rho}_{\overline{H}}$ and $\sum_{i\in J} \widetilde{\rho}_{\overline{H}_i}$ coincide by showing that both compute the composition length of objects in $\fp (\cA/\overrightarrow{\cS})$. Let $L$ be a simple in $\fp (\cA/\overrightarrow{\cS})$. There exists a unique $i_0\in J$ such that $\overline{H}_{i_0}$ is the injective envelope of $L$. Mimicking the arguments in the proof of Lemma \ref{lem:injFinEndolength}, we conclude that $\Hom_{\cA/\overrightarrow{\cS}}(L,\overline{H}_{i_0})$ is a simple $\End_{\cA/\overrightarrow{\cS}}(\overline{H}_{i_0})$-module. By the same reasoning, $\Hom_{\cA/\overrightarrow{\cS}}(L,\overline{H})$ is a simple $\End_{\cA/\overrightarrow{\cS}}(\overline{H})$-module. Hence
	\[
\sum_{i\in J} \widetilde{\rho}_{\overline{H}_i}(L)= \widetilde{\rho}_{\overline{H}_{i_0}}(L)=1=\widetilde{\rho}_{\overline{H}}(L).
	\]
	This identity implies that $\widetilde{\rho}_{\overline{H}}=\sum_{i\in J} \widetilde{\rho}_{\overline{H}_i}$. Consequently, $\widetilde{\rho}_H=\sum_{i\in J} \widetilde{\rho}_{H_i}$. 
	
	According to \cite[\S 3, p.~523]{Herzog1997}, every indecomposable injective object in a locally coherent category is isomorphic to the injective envelope of some finitely generated object. Because the finitely presented objects in a locally finite category are noetherian, they coincide with the finitely generated ones (we refer to \cite[\S 7, p.~542]{Herzog1997}). In fact, note that taking injective envelopes of finitely presented simple objects yields a bijection between the isoclasses of finitely presented simples and the isoclasses of indecomposable injectives in any locally finite category. Hence, up to isomorphism, every indecomposable injective object in $\cA/\overrightarrow{\cS}$ must be a direct factor of  $\overline{H}$. The set of pairwise non-isomorphic direct factors $\{\overline{H}_i\mid i\in J\}$ of $\overline{H}$ is therefore a complete irredundant set of indecomposable injective objects in $\cA/\overrightarrow{\cS}$, up to isomorphism, and $\Zg (\cA / \overrightarrow{\cS}) =\{[\overline{H}_i]\mid i\in J\}$. We get that $\Zg \cA \setminus \co  (\cS)=\{[H_i]\mid i\in J\}$ and this is a closed subset of $\Zg\cA$ (see Subsection \ref{subsection:fundamental_corresp}).
	
	For part (2), it is clear that $\Ker\widetilde{\rho}=\Ker\widetilde{\sigma}$. By Theorem \ref{thm:IrrAdditiveFunctions<->InjFinEndolength}, there exist pairwise non-isomorphic indecomposable injective objects of finite endolength $H_i$ in $\cA$ such that $\widetilde{\rho}_i=\widetilde{\rho}_{H_i}$ for $i\in J$. Set $H\coloneqq\coprod_{i\in J}H_i$. We claim that $H$ is an injective of finite endolength. Consider the localisation sequence \eqref{eq:loc} for $\Ker\widetilde{\sigma}$. As $\Ker\widetilde{\sigma}=\bigcap_{i\in J}\Ker\widetilde{\rho}_{i}=\bigcap_{i\in J}\Ker{\widetilde{\rho}}_{H_i}$, we have $\Hom_{\cA}(Y,H_i)=0$ for every $Y$ in $\Ker\widetilde{\sigma}$. Hence $H_i\cong R(\overline{H}_i)$ for certain injective objects $\overline{H}_i$ in $\cA/\overrightarrow{\Ker\widetilde{\sigma}}$. As $\cA/\overrightarrow{\Ker\widetilde{\sigma}}$ is locally finite, the coproduct $\overline{H}\coloneqq\coprod_{i\in J}\overline{H}_i$ is an injective of finite endolength (see \cite[Proposition 7.1]{Herzog1997} and Lemma \ref{lem:injFinEndolength}). As before, we get $R(\overline{H})\cong H$ and $H$ must be an injective of finite endolength. From part (1), it follows that $\widetilde{\rho}_H=\sum_{i\in J}\widetilde{\rho}_{H_i}=\widetilde{\sigma}$.
	
	To prove the statement in (3), let $\cS$ be a Serre subcategory of $\fp\cA$ such that $\cA/ \overrightarrow{\cS}$ is locally finite. Recall that the composition length of objects in $\fp(\cA/ \overrightarrow{\cS})$ is an integral additive function with trivial kernel. Consequently, $\Ker\widetilde{\ell}^E=\Ker E = \cS$ for  the localisation $E:\cA \to \cA/\overrightarrow{\cS}$.
\end{proof}

In the proof of the following lemmas, we are going to use results from \cite{Krause1998} which are stated in the language of exactly definable categories.  In order to apply these results to a locally coherent category $\cA$ we use the bijections between locally coherent categories and skeletally small abelian categories from \cite[\S1.4, Theorem]{Crawley-Boevey1994} and between skeletally small abelian categories and exactly definable categories in \cite[Corollary 2.9]{Krause1998}. Going through these bijections identifies, up to equivalence, the corresponding skeletally small abelian category with $\fp\cA$ and the exactly definable category with the subcategory of fp-injective objects in $\cA$. The Ziegler spectrum of an exactly definable category can be identified with the Ziegler spectrum of the corresponding locally coherent category.  

\begin{lemma}\label{lem:IsoClosedIsoQuotient}
Let $\cA$ be a locally coherent category. Let $C$ be a closed subset of $\Zg\cA$ and let $\cS$ be the Serre subcategory of $\fp\cA$ such that $C=\Zg\cA \setminus \co(\cS)$. Then the isolation property holds for $C$ if and only if it holds for $\Zg(\cA/\overrightarrow{\cS})$. 
\end{lemma}

\begin{proof}
According to Subsection \ref{subsection:fundamental_corresp} and the notation therein, the functor $R:\cA/\overrightarrow{\cS}\to \cA$ induces a homeomorphism from $\Zg (\cA/\overrightarrow{\cS})$ to $C$. Hence, a point $\{[\overline{Q}]\}\subseteq\Zg (\cA/\overrightarrow{\cS})$ is open in its closure in $\Zg (\cA/\overrightarrow{\cS})$ if and only if $\{[R(\overline{Q})]\}\subseteq C$ is open in its closure in $C$. Since $C$ is closed, the latter condition holds exactly when $\{[R(\overline{Q})]\}$ is open in its closure in $\Zg \cA$. As $R$ induces a homeomorphism, every point $[Q]$ in $C$ such that $\{[Q]\}$ is open in its closure in $\Zg \cA$ is of the form $[Q]=[R(\overline{Q})]$ for some point $[\overline{Q}]$ in $\Zg (\cA/\overrightarrow{\cS})$ such that $[\overline{Q}]$
is open in its closure, and such special points $[Q]$ in $C$ and $[\overline{Q}]$ in $\Zg (\cA/\overrightarrow{\cS})$ are in bijection. Now take $[\overline{Q}]$ in $\Zg (\cA/\overrightarrow{\cS})$ and let $Q\coloneqq R(\overline{Q})$ be a representative of the corresponding element in $C$. Consider the Serre subcategories $\cS_{\overline{Q}}\coloneqq \Ker(\Hom_{\cA/\overrightarrow{\cS}} (-,\overline{Q})|_{\fp (\cA/\overrightarrow{\cS})})$ of $\fp(\cA/\overrightarrow{\cS})$ and  $\cS_Q\coloneqq \Ker(\Hom_\ca (-,Q)|_{\fp \cA})$ of $\fp \cA$. We claim that the categories $(\cA/\overrightarrow{\cS})/\overrightarrow{\cS_{\overline{Q}}}$ and $\cA/ \overrightarrow{S_Q}$ are equivalent. This will imply that $\fp(\cA/\overrightarrow{\cS})/ \cS_{\overline{Q}}$ and $(\fp\cA)/\cS_{Q}$ are also equivalent categories and so one of them contains a simple object if and only if the other does. That, together with the observations at the beginning of the paragraph, then proves that the isolation property holds for $\Zg(\cA/\overrightarrow{\cS})$ if and only if it holds for the closed set $C$ of $\Zg \cA$. To show that $(\cA/\overrightarrow{\cS})/\overrightarrow{\cS_{\overline{Q}}}$ and $\cA/ \overrightarrow{S_Q}$ are equivalent, consider the Serre localisation functors $E_Q:\cA \to \ca/\overrightarrow{\cS_Q}$, $E_{\overline{Q}}: \cA/\overrightarrow{\cS} \to (\cA/\overrightarrow{\cS}))/ \overrightarrow{\cS_{\overline{Q}}}$ and $E: \cA \to \cA/\overrightarrow{\cS}$, where $E$ is left adjoint to $R$. We show first that $E_Q$ factors through $E$. From the definition of $\co(\cS)$ and the fact that $[Q]\in C$ we get $\Hom_{\cA}(\cS,Q)=0$.  Therefore $\cS \subseteq \cS_Q$ and consequently $\overrightarrow{\cS}\subseteq \overrightarrow{\cS_Q}$. By the universal property of $E$, there is an exact functor $F:\cA/\overrightarrow{\cS} \to \cA/\overrightarrow{\cS_Q}$ such that $E_Q\cong F \circ E$. We now prove that $F$ factors through $E_{\overline{Q}}$. Recall that $R$ is fully faithful and note that
\[ X \text{ in } \cS_{\overline{Q}} \Rightarrow \Hom_{\cA/\overrightarrow{\cS}}(X,\overline{Q})=0=\Hom_{\cA}(R(X),R(\overline{Q})) \Rightarrow R(X) \text{ in } \cS_Q.\]
If $X$ is an object in $\overrightarrow{\cS_{\overline{Q}}}$, then $R(X)$ lies in $\overrightarrow{\cS_Q}$ as $R$ preserves filtered colimits, and so $F(X)\cong F ( E \circ R (X))\cong E_Q(R(X))=0$. Therefore $\overrightarrow{\cS_{\overline{Q}}} \subseteq \Ker F$ and the universal property of $E_{\overline{Q}}$ guarantees the existence of an exact functor $G:(\cA/\overrightarrow{\cS})/\overrightarrow{\cS_{\overline{Q}}} \to \cA/ \overrightarrow{\cS_Q}$ satisfying $F\cong G\circ E_{\overline{Q}}$. Lastly, we check that $E_{\overline{Q}}\circ E$ factors through $E_Q$. Note that
\[
X \text{ in } \cS_Q \Rightarrow\Hom_{\cA}(X, R(\overline{Q}))=0=\Hom_{\cA/\overrightarrow{\cS}}(E(X), \overline{Q}) \Rightarrow E(X) \text{ in } \cS_{\overline{Q}}.
\]
As before, if $X$ is an object in $ \overrightarrow{\cS_Q}$, then $E(X)$ is in $\overrightarrow{\cS_{\overline{Q}}}$, as $E$ is a left adjoint and therefore commutes with all colimits. Consequently, $\overrightarrow{\cS_Q} \subseteq  \Ker (E_{\overline{Q}}\circ E)$. By the universal property of $E_Q$, there exists an exact functor $H: \cA/ \overrightarrow{\cS_Q} \to (\cA/\overrightarrow{\cS})/ \overrightarrow{\cS_{\overline{Q}}}$ such that $E_{\overline{Q}} \circ E\cong H \circ E_Q$. Standard computations show that $G\circ H \circ E_Q\cong E_Q$ and $H \circ G \circ E_{\overline{Q}}\cong E_{\overline{Q}}$, so $G \circ H\cong \id_{\cA/\overrightarrow{\cS_Q}}$ and $H \circ G\cong \id_{(\cA/\overrightarrow{\cS})/ \overrightarrow{\cS_{\overline{Q}}}}$, as $E_{Q}$ and $E_{\overline{Q}}$ are localisations. Therefore, $G$ and $H$ are mutually quasi-inverse equivalences.
\end{proof}

\begin{lemma}\label{lem:IsolationLocFin}
Let $\cA$ be a locally coherent category. The following statements are equivalent:
\begin{enumerate}[(a)]
    \item $\cA$ is locally finite;
    \item $\Zg \cA$ is discrete and satisfies the isolation property;
    \item $\Zg \cA$ is discrete and the coproduct of representatives of all its elements is $\varSigma$-pure-injective.
\end{enumerate}
\end{lemma}
\begin{proof}
By \cite[Corollary 12.12]{Krause1998} and by the correspondences in \cite[Theorem 2.8, Corollary 2.9]{Krause1998}, the Krull--Gabriel dimension of $\fp\cA$ coincides with the Cantor--Bendixson rank of $\Zg \cA$ whenever $\Zg \cA$ satisfies the isolation property. By definition, the Cantor--Bendixson rank of a topological space is zero if and only if it is discrete. Similarly, the Krull--Gabriel dimension of a skeletally small abelian category is zero if and only if all its objects have finite length. In particular, under the assumption that $\Zg\ca$ has the isolation property, the discreteness of $\Zg\cA$ forces a locally coherent category $\cA$ to be locally finite. Conversely, one can prove directly that the Ziegler spectrum of a locally finite category is discrete and satisfies the isolation property, but this also follows from \cite[Corollaries 12.5, 12.12]{Krause1998}. This proves that (a) and (b) are equivalent. When $\cA$ is locally finite, the class of injective objects is closed under coproduts and all injectives have finite endolength (recall Lemma \ref{lem:injFinEndolength}). Consequently, (a) implies (c). Suppose now that the coproduct of the elements in $\Zg \cA$ is $\varSigma$-pure-injective. By \cite[\S3.5, Theorem 2]{Crawley-Boevey1994}, \cite[Lemma 3.1]{Krause1997} and \cite[Proposition 9.3]{Krause1998}, all objects in $\fp\cA$ must be noetherian. It then follows from \cite[Proposition 12.3(3), Corollary 12.5]{Krause1998} that $\Zg \cA$ has the isolation property. This shows that (c) implies (b).
\end{proof}

\begin{proof}[Proof of Theorem \ref{thm:lastPreliminaries}]
The bijection between (a) and (b) in the statement of Theorem \ref{thm:lastPreliminaries} follows from parts (1) and (2) of Proposition \ref{prop:IrrAdditiveFunctions<->InjFinEndolength2} and from Theorem \ref{thm:IrrAdditiveFunctions<->InjFinEndolength}. According to Theorem \ref{thm:ZieglerSpecLocallyCohCat} and Subsection \ref{subsection:fundamental_corresp}, (c) and (d) are in one-to-one correspondence. Consider now the assignment $\widetilde{\rho}\mapsto \overrightarrow{\Ker \widetilde{\rho}}$. By Proposition \ref{prop:LocAdditiveFunctions}, this provides a well-defined correspondence between (b) and (c). To see that this mapping is injective, note that any two distinct basic additive functions on $\fp \cA$ give rise to different open sets in $\Zg \cA$ (by part (1) of Proposition \ref{prop:IrrAdditiveFunctions<->InjFinEndolength2}) and use Theorem \ref{thm:ZieglerSpecLocallyCohCat}. Surjectivity follows from part (3) of Proposition \ref{prop:IrrAdditiveFunctions<->InjFinEndolength2}. To prove the bijection between (d) and (e), suppose that $\cS$ is a Serre subcategory of $\fp\cA$. By Lemma \ref{lem:IsolationLocFin}, $\cA/\overrightarrow{\cS}$ is locally finite if and only if $\Zg(\cA/\overrightarrow{\cS})$ is discrete and satisfies the isolation property.
Since $\Zg(\cA/\overrightarrow{\cS})$ and $\Zg(\cA)\setminus \co(\cS)$ are homeomorphic, Lemma \ref{lem:IsoClosedIsoQuotient} implies that the previous conditions hold if and only if the closed set $\Zg\cA\setminus\co(\cS)$ is discrete and has the isolation property. Together with Theorem \ref{thm:ZieglerSpecLocallyCohCat}, this shows that the assignment from (d) to (e) is well-defined and bijective.
\end{proof}

\section{Group-valued rank functions}
\label{appendix:group_valued}
 In this appendix, we provide a generalisation of Theorem \ref{thm:TransInvariantAddFunctions<->rankFunctions} that covers $d$-periodic rank functions considered in \cite{Chuang2021}. Throughout this section, $\cC$ is a skeletally small triangulated category. Furthermore, we work in the following setup.

\begin{assumption} \label{assumption:preordered}
 Let $Z$ be a module over a ring $R$ whose underlying abelian group $(Z, +)$ is endowed with a translation-invariant partial order $\leq$, that is, $x\leq y$ implies $x+ z \leq y + z$ for all $x, y, z \in Z$. Let $q$ be a fixed element in $R$.
 \end{assumption}
  
  One elementary example is  $Z=\mathbb{Z}$ or  $Z=\mathbb{R}$, considered as a module over $R=\mathbb{Z}$. 
 Another typical example is $Z = \mathbb{R}[x]/(x^d - 1)$ considered as a module over the polynomial ring $R=\mathbb{R}[x]$, with $R$ acting on $Z$ by multiplication. The partial order on $\mathbb{R}[x]/(x^d - 1)$ is induced from the partial order on $\mathbb{R}[x]$ after expressing all the elements in the standard basis. The partial order on $\mathbb{R}[x]$ is given as follows: $f(x)\leq g(x)$ if $g(x)-f(x)\in\mathbb{R}_{\geq 0}[x]$.
 
Setup \ref{assumption:preordered} is sufficient to define $Z$-valued counterparts of all three types of functions considered in Subsection \ref{subsec:RankFun<->SigmaInvAddFun}, with $q$ appearing in the suitable replacement of the translation invariance property.   
 
 \begin{definition}\label{def:app1}
 A \emphbf{$q$-$\Sigma$-invariant $Z$-valued additive function} $\widetilde{\rho}$ on $\mod \cC$ is an assignment of an element of $Z_{\geq 0}$ to each object in $\mod \cC$, which is additive on short exact sequences and satisfies $\widetilde{\rho}({\Sigma}^*F)=q \widetilde{\rho}(F)$ for every $F$ in $\mod \cC$.
 \end{definition}
 
 \begin{definition}[{cf.~Definition \ref{def:rankfunction}}]\label{def:app2} 
 A \emphbf{$Z$-valued $q$-rank function} $\rho$ on $\cC$ is an assignment of an element $\rho(f) \in Z$ to each morphism $f$ in $\cC$, so that $\rho$ satisfies axioms (M1), (M2) and (M3), together with the following additional condition:
 \begin{itemize}
     \item [(M4q)] \emphbf{$q$-$\Sigma$-invariance:} $\rho(\Sigma f) = q\rho(f)$ for every morphism $f$ in $\cC$.
 \end{itemize}
 \end{definition}
 
 \begin{definition}[{cf.~Definition \ref{def:rankfunctionobj}}]\label{def:app3} 
 A \emphbf{$Z$-valued $q$-rank function on objects} of $\cC$ is an assignment $\rho_{ob}$ of an element $\rho_{ob}(X)\in {Z}$ to every object $X$ in $\cC$, so that $\rho_{ob}$ satisfies axioms (O1), (O2), and the following additional conditions:
	\begin{enumerate}
	   \item[(O3q)] \emphbf{Congruence:} for every $f:X \to Y$ in $\cC$ there exists $z_f\in Z_{\geq 0}$ such that 
	   \[
	   \rho_{ob}(Y)-\rho_{ob}(\Cone f)+\rho_{ob}(\Sigma X)=(q+1)z_f;\]
	   \item[(O4q)] \emphbf{$q$-$\Sigma$-invariance:} $\rho_{ob}(\Sigma X) = q\rho_{ob} (X)$ for every $X$ in $\cC$.
	\end{enumerate}
 \end{definition}
 
 We say that an element $r \in R$ is \emphbf{$Z$-regular} if the left multiplication by $r$ yields an injective map $Z\to Z$.
 Note that $z_f$ in (O3q) is unique if $q + 1$ is regular. By abuse of notation, in the theorem below we will write $\frac{\rho_{ob}(Y)-\rho_{ob}(\Cone f)+\rho_{ob}(\Sigma X)}{q + 1}$ instead of $z_f$. It is not difficult to deduce that a $Z$-valued $q$-rank function on objects satisfies the triangle inequality whenever $q+1$ is $Z$-regular. To see this, apply (O3q) to $f$ and $\Sigma f$ and use $q$-$\Sigma$-invariance.
 
 \begin{theorem}[{cf.~Theorem~\ref{thm:TransInvariantAddFunctions<->rankFunctions},~\cite[Propositions 2.4, 2.8, 2.12]{Chuang2021}}] 
 \label{thm:GroupValuedCorrespondence} Let $q+1\in R$ be $Z$-regular. There is a bijective correspondence between:
	\begin{enumerate}[(a)]
		\item $q$-$\Sigma$-invariant $Z$-valued additive functions $\widetilde{\rho}$ on $\mod \cC$;
		\item  $Z$-valued $q$-rank functions $\rho$ on $\cC$;
		\item $Z$-valued $q$-rank functions  $\rho_{ob}$ on objects of $\cC$.
	\end{enumerate}
	
	The assignments between (a) and (b), and (b) and (c) are given by:
	\begin{align*}
	\widetilde{\rho} \mapsto \rho: \qquad &\rho(f) \coloneqq \widetilde{\rho}(\Imm \Hom_\cC(-,f)); \\ 
	\rho \mapsto \widetilde{\rho}: \qquad &\widetilde{\rho}(F) \coloneqq \rho(f)\text{ for }  F\cong \Imm \Hom_\cC(-,f);\\
	\rho \mapsto \rho_{ob}: \qquad &\rho_{ob}(X) \coloneqq \rho(\id_X);\\
	\rho_{ob} \mapsto \rho: \qquad &\rho(f: X \to Y) \coloneqq \frac{\rho_{ob}(Y)-\rho_{ob}(\Cone f)+\rho_{ob}(\Sigma X)}{q + 1}.
	\end{align*}
	\end{theorem}
	
	\begin{proof}
	To show the bijection between (a) and (b), follow the proof of Theorem~\ref{thm:TransInvariantAddFunctions<->rankFunctions} word for word, replacing $\Sigma$-invariance by $q$-$\Sigma$-invariance and $2$ by $q + 1$ in \eqref{eq:additiveToRank}, and adapting \eqref{eq:qTransInv}.
	For the bijection between (b) and (c), one may follow the strategy of proofs in  \cite[Propositions 2.4, 2.8, 2.12]{Chuang2021}. We sketch the proof for completeness. 
	Let $\rho$ be a $Z$-valued $q$-rank function on $\cC$.
	Straightforward verification shows that the map $\rho_{ob}$, defined by $\rho_{ob}(X)\coloneqq\rho(\id_X)$, satisfies axioms (O1) and (O2). The identity $\Sigma \id_X=\id_{\Sigma X}$ implies that $\rho_{ob}$ is $q$-$\Sigma$-invariant. Finally, Lemma \ref{lem:PropPreRankFunctions} guarantees that $\rho_{ob}$ is congruent. Conversely, let $\rho_{ob}$ be a $Z$-valued $q$-rank function on objects of $\cC$. Consider the map $\rho$ given by $\rho(f)\coloneqq z_f$ for $f:X \to Y$ in $\cC$ and $z_f\in Z_{\geq 0}$ such that $(q+1)z_f=\rho_{ob}(Y)-\rho_{ob}(\Cone f)+\rho_{ob}(\Sigma X)$. It is not difficult to show that the map $\rho$ is $q$-$\Sigma$-invariant and additive: this follows from the fact that $\rho_{ob}$ is $q$-$\Sigma$-invariant and additive and that $q+1$ is $Z$-regular. In order to check the rank-nullity condition, consider a triangle
	\[
	\begin{tikzcd}
	X \ar[r, "f"] & Y \ar[r,"g"] & Z \ar[r,"h"] & \Sigma X.
	\end{tikzcd}
	\]
	Note that $(q+1)z_f=\rho_{ob}(Y)-\rho_{ob}(Z)+\rho_{ob}(\Sigma X)$ and $(q+1)z_g=\rho_{ob}(Z)-\rho_{ob}(\Sigma X)+\rho_{ob}(\Sigma Y)$, hence
	\[(q+1)(z_f+z_g)=\rho_{ob}(Y)+\rho_{ob}(\Sigma Y)=(q+1)z_{\id_Y}.
	\]
	As $q+1$ is $Z$-regular, we obtain $\rho(f)+\rho(g)=\rho(\id_Y)$.
	\end{proof}

    \begin{example}\label{ex:app1}
    Take $Z = \mathbb{R}$ considered as a module over $R = \mathbb{Z}$, and take $q = 1$. Definitions \ref{def:app1}, \ref{def:app2} and \ref{def:app3} recover $\Sigma$-invariant additive functions and rank functions considered in Section \ref{sec:RankFunctions}. If we consider $Z = \mathbb{Z}$ instead, Definitions \ref{def:app1} and \ref{def:app2} correspond to the integral case from Section \ref{sec:Integral}.
    \end{example}
    
    \begin{example}
    Let $R$ be the polynomial ring $\mathbb{R}[x]$ and let $Z$ be its quotient $\mathbb{R}[x]/(x^d - 1)$, with $R$ acting on $Z$ by multiplication. Take $q=x$. The element $x + 1 \in \R[x]$ is $Z$-regular if and only if $d$ is odd.  In this case, $Z$-valued $q$-rank functions are precisely the $d$-periodic rank functions considered in \cite{Chuang2021}. The bijection between (b) and (c) in Theorem \ref{thm:GroupValuedCorrespondence} recovers \cite[Propositions 2.8, 2.12]{Chuang2021}. Taking $d=1$, we recover Example \ref{ex:app1}.
    The notion of a $d$-periodic rank function is defined for even $d$ as well (see \cite[Definition 2.6]{Chuang2021}). However, the regularity of $q + 1$ is crucial in the proof of  Theorem~\ref{thm:GroupValuedCorrespondence}. Note that for even $d$, $Z$-valued $q$-rank functions may differ from $d$-periodic rank functions.
    \end{example}
    
    \begin{example}
    Let $R = Z$ be the Laurent polynomial ring $\mathbb{R}[x, x^{-1}]$ acting on itself by multiplication, and take $q=x$. The element $x + 1$ is $Z$-regular. In this case, $Z$-valued $q$-rank functions are precisely the $\infty$-periodic rank functions considered in \cite{Chuang2021}. Again, the bijection between (b) and (c) in Theorem \ref{thm:GroupValuedCorrespondence} recovers \cite[Propositions 2.8, 2.12]{Chuang2021}. For a ring $A$, there is a one-to-one correspondence between $\infty$-periodic rank functions $\rho$ on $\cper (A)$ satisfying $\rho(\id_A) = 1$ and Sylvester rank functions
    (see \cite[Theorem 3.7]{Chuang2021}).
\end{example}

Note that each $d$-periodic or $\infty$-periodic rank function on $\cC$ induces a usual, i.e. $1$-periodic rank function on $\cC$ (see \cite[Remark 2.7]{Chuang2021}). In particular, each Sylvester rank function induces a rank function on $\cper (A)$.

\bibliographystyle{alpha}
\bibliography{biblio}

\end{document}